\documentclass[twoside]{article}

\usepackage[accepted]{aistats2022}
\usepackage{microtype}
\usepackage{graphicx}
\usepackage{subfigure}
\usepackage{booktabs} 
\usepackage{epsfig,epsf,fancybox}
\usepackage{amsmath}
\usepackage{mathrsfs}
\usepackage{amssymb}
\usepackage{color}
\usepackage{multirow}
\usepackage{paralist}
\usepackage{verbatim}
\usepackage{algorithm}
\usepackage{algorithmic}
\usepackage{galois}
\usepackage{boxedminipage}
\usepackage{accents}
\usepackage{stmaryrd}
\usepackage[bottom]{footmisc}

\usepackage{natbib}
\usepackage{url}
\usepackage[colorlinks,linkcolor=blue,citecolor=blue]{hyperref}

\newtheorem{theorem}{Theorem}[section]

\newtheorem{lemma}[theorem]{Lemma}

\newtheorem{assumption}[theorem]{Assumption}

\newcommand{\RR}{\mathbf R}

\newcommand{\ba}{\begin{array}}
\newcommand{\ea}{\end{array}}

\newcommand{\red}{\color{red}}

\usepackage{enumitem,xspace}

\def\ROOTSGDV1{\textsf{ROOT\mbox{-}SGD}\xspace}

\def\RATE{\textsf{RATE}\xspace}
\newcommand{\Bb}{\mathbf{B}}
\newcommand{\gb}{\mathbf{g}}

\newcommand{\Ib}{\mathbf{I}}
\newcommand{\Mb}{\mathbf{M}}
\renewcommand{\RR}{\mathbb{R}}
\newcommand{\Exs}{\ensuremath{{\mathbb{E}}}}
\newcommand{\Quan}{\mathcal{Q}}
\newcommand{\prefactor}{\mathcal{P}}

\def\RATE{\textsf{RATE}\xspace}
\def\xholder{\mathbf{x}}
\def\yholder{\mathbf{y}}
\def\zholder{\mathbf{z}}

\def\Tholder{K}
\def\epoch{\mathsf{epoch}}
\def\Epoch{\mathsf{Epoch}}
\def\beq{\begin{equation}}
\def\eeq{\end{equation}}
\def\lambdaetaprime{\lambda^*(\eta)}

\newtheorem{innercustom}{}
\newenvironment{custom}[1]
  {\renewcommand\theinnercustom{#1}\innercustom}
  {\endinnercustom}

\usepackage{enumitem}

\def\red#1{}

\begin{document}

%


\runningauthor{C.~J.~Li, Y.~Yu, N.~Loizou, G.~Gidel, Y.~Ma, N.~Le Roux, M.~I.~Jordan}

\twocolumn[

\aistatstitle{On the Convergence of Stochastic Extragradient for Bilinear Games using Restarted Iteration Averaging}

\aistatsauthor{
Chris Junchi Li$^{\diamond, \star}$
\And
Yaodong Yu$^{\diamond, \star}$
\And
Nicolas Loizou$^{\triangleleft}$
\And
Gauthier Gidel$^{\dagger,\ddagger}$
}
\aistatsauthor{
Yi Ma$^\diamond$
\And
Nicolas Le Roux$^{\dagger,\ddagger,\S,\Box}$
\And
Michael I. Jordan$^\diamond$
}

\vspace{0.1in}


\centering{University of California, Berkeley$^\diamond$ \quad  Johns Hopkins University$^\triangleleft \quad  $Mila$^\dagger$}

\vspace{0.05in}
\centering{ Université de Montréal$^\ddagger$ \quad McGill University$^\S$ \quad Microsoft Research$^\Box$ }
\vspace{0.4in}

]

\begin{abstract}
We study the stochastic bilinear minimax optimization problem, presenting an analysis of the same-sample Stochastic ExtraGradient (SEG) method with constant step size, and presenting variations of the method that yield favorable convergence. In sharp contrasts with the basic SEG method whose last iterate only contracts to a fixed neighborhood of the Nash equilibrium, SEG augmented with iteration averaging provably converges to the Nash equilibrium under the same standard settings, and such a rate is further improved by incorporating a scheduled restarting procedure. In the interpolation setting where noise vanishes at the Nash equilibrium, we achieve an optimal convergence rate up to tight constants. We present numerical experiments that validate our theoretical findings and demonstrate the effectiveness of the SEG method when equipped with iteration averaging and restarting.
\end{abstract}

\section{INTRODUCTION}\label{sec_intro}
The \emph{minimax optimization} framework provides solution concepts useful in game theory~\citep{morgenstern1944theory}, statistics~\citep{bacheta} and online learning~\citep{blackwell1956analog,cesa2006prediction}.
It has recently been prominent in the deep learning community due to its application to generative modeling~\citep{goodfellow2014generative, arjovsky2017wasserstein} and robust prediction~\citep{madry2017towards, zhang2019theoretically}.
There remains, however, a gap between minimax characterizations of solutions and algorithmic frameworks that provably converge to such solutions in practice.

In standard single-objective machine learning applications, the traditional algorithmic realization of optimization frameworks is stochastic gradient descent (SGD, or one of its variants), where the full gradient is formulated as an expectation over the data-generating mechanism.
In general minimax optimization problems, however, naive use of SGD leads to pathological behavior due to the presence of rotational dynamics~\citep{goodfellow2016nips,balduzzi2018mechanics}.

One way to overcome these rotations is to use gradient-based methods specifically designed for the minimax setting (or more generally for the multi-player game setting).
A key example of such a method is the celebrated \emph{extragradient method}.
Originally introduced by~\citep{g.m.korpelevichExtragradientMethodFinding1976}, it addresses general minimax optimization problems and yields optimal convergence guarantees in the batch setting~\citep{azizian2020accelerating}. 
In the stochastic setting, however, it has only been analyzed in special cases, such as the constrained case \citep{juditsky2011solving}, the bounded-noise case \citep{hsieh2020explore}, and the interpolatory case \citep{vaswani2019painless}.
In the current paper, we study the general stochastic bilinear minimax optimization problem, also known as the bilinear saddle-point problem,
\beq\label{Sminimax}
\min_\xholder \max_\yholder~
\xholder^\top \Exs_\xi[\Bb_{\xi}] \yholder 
+ 
\xholder^\top \Exs_\xi[\gb^\xholder_\xi] 
+ 
\Exs_\xi[(\gb^\yholder_\xi)^\top] \yholder \, ,
\eeq
where the index $\xi$ denotes the randomness associated with stochastic sampling. 
Following standard practice we assume that the expected coupling matrix $\Bb = \Exs[\Bb_\xi]$ is nonsingular, and that the intercept vectors $\gb^\xholder_\xi$ and $\gb^\yholder_\xi$ have zero mean:
$\Exs[\gb^\xholder_\xi] = \mathbf{0}_n$ and $\Exs[\gb^\yholder_\xi] = \mathbf{0}_m$.
Thus the Nash equilibrium point is $[\xholder^*; \yholder^*] = [\mathbf{0}_n; \mathbf{0}_m]$.
Such assumptions are standard in the literature on bilinear optimization~\citep[see, e.g., ][]{vaswani2019painless,mishchenko2020revisiting}.%
\footnote{In the case of a square, nonsingular coupling matrix $\Bb$ this assumption is feasible without loss of generality, while in the rectangular matrix case we simply restrict ourselves to the nonsingular $\min(n,m)$-dimensional subspace induced by singular value decomposition. The nonzero component of the intercept vectors $[\gb^\xholder_\xi; \gb^\yholder_\xi]$ projected onto such a subspace is not taken into account in the SEG dynamics.
}

In this work, we present theoretical results in the general setting of bilinear minimax games for a version of the Stochastic ExtraGradient (SEG) method that incorporates iteration averaging and scheduled restarting. 
The introduction of stochasticity in the matrix $\Bb_{\xi}$ together with an unbounded domain presents technical challenges that have been a major stumbling block in earlier work~\citep[cf.][]{dieuleveut2016harder}.
Here we show how to surmount these challenges.
Formally, we introduce the following SEG method composed of an extrapolation step (half-iterates) and an update step:
\beq\label{SEGupdate}\begin{aligned}
\xholder_{t-1/2}	&=	\xholder_{t-1} - \eta_t\left[ \Bb_{\xi, t} \yholder_{t-1} + \gb^\xholder_{\xi,t}\right],
\\
\yholder_{t-1/2}	&=	\yholder_{t-1} + \eta_t\left[ \Bb_{\xi, t}^\top \xholder_{t-1} + \gb^\yholder_{\xi,t}\right],
\\ 
\xholder_t          &=	\xholder_{t-1} - \eta_t\left[ \Bb_{\xi, t} \yholder_{t-1/2} + \gb^\xholder_{\xi,t}\right],
\\
\yholder_t  	    &=	\yholder_{t-1} + \eta_t\left[ \Bb_{\xi, t}^\top \xholder_{t-1/2} + \gb^\yholder_{\xi,t}\right].
\end{aligned}
\eeq
Here and throughout we adopt a \textbf{\emph{same-sample-and-step-size}} notation in which the extrapolation and extragradient steps share the same stochastic sample \citep{gidel2019variational,mishchenko2020revisiting} and step size $\eta_t$; i.e., the updates in~Eq.~\eqref{SEGupdate} use the same samples of $\Bb_{\xi}$, $\gb^\xholder_\xi$ and $\gb^\yholder_\xi$.
Note that there exist counterexamples~\citep[see, e.g.,][Theorem 1]{chavdarova2019reducing} where the SEG iteration~\citep{juditsky2011solving} persistently diverges when using independent samples. 
The same-sample stochastic extra gradient (SEG) method aims to address this issue~\citep{gidel2019variational,mishchenko2020revisiting}.
In practice, for the bilinear game problems we consider in this paper as well as other application problems, including generative adversarial networks and adversarial training, it is easy to perform the same-sample SEG updates: in most machine learning applications one can re-use a sample without significant extra cost.

\smallskip\noindent\textbf{Main contributions.}
We provide an in-depth study of SEG on bilinear games and we show that, unlike in the minimization-only setting, in the minimax optimization setting the last-iterate SEG algorithm with the same sample and step sizes \emph{cannot} converge in general even when the step sizes are diminishing to zero [Theorems \ref{theo_SEG_A} and \ref{theo_SEG_lower_bound}].
This motivates our study of averaging and restarting in order to obtain meaningful convergence rates:

\begin{enumerate}[topsep=0pt,parsep=0pt,partopsep=0pt, leftmargin=*,label=(\roman*)]
\item
We prove that in the bilinear game setting, under mild assumptions, iteration averaging allows SEG to converge at the rate of $1/\sqrt{\Tholder}$ [Theorem \ref{theo_SEG_B}], $\Tholder$ being the number of samples the algorithm has processed.  
This rate is statistically optimal up to a constant multiplier.
Additionally, we can further boost the convergence rate when we combine iteration averaging with scheduled restarting [Theorem \ref{theo_SEG_C}] when the lower bound of the smallest eigenvalue in the coupling matrix is known to the system.
In this case, exponential forgetting of the initialization and an optimal statistical rate are achieved.

\item
In the special case of the interpolation setting, we are able to show that SEG with iteration averaging and scheduled restarting achieves an accelerated rate of convergence, faster than (last-iterate) SEG [Theorem \ref{theo_SEGg_interpolation_C}], reducing the dependence of the rate on the condition number to a dependence on its square root. 
We achieve state-of-the-art rates comparable to the full batch optimal rate~\citep{azizian2020accelerating}, with access only to a stochastic estimate of the gradient, improving upon~\citet{vaswani2019painless}.

\item
We provide the first convergence result on SEG with unbounded noise.
The only existing result of which we are aware of for the unbounded noise setting is the work of \citet{vaswani2019painless} in the interpolation setting.
Our theoretical results are further validated by experiments on synthetic data.

\end{enumerate}

\subsection{Related Work}
\noindent\textbf{Bilinear minimax optimization.} The study of the bilinear example as a tool to understand minimax optimization originated with~\citet{daskalakis2018training}, who studied an optimistic gradient descent-ascent (OGDA) algorithm to solve that minimax problem.
They were able to prove sublinear convergence for this method.
Later,~\citet{mokhtari2020unified} proposed to analyze OGDA and the related ExtraGradient (EG) method as perturbations of the Proximal Point (PP) method.
They were able to prove a linear convergence rate for both EG and OGDA with an iteration complexity of $O(\kappa\log (1/\epsilon))$, where $\kappa \equiv \lambda_{\max}(\Bb^\top \Bb)/\lambda_{\min}(\Bb\Bb^\top)$ is the condition number of problem Eq.~\eqref{Sminimax}.
Highly related to the current work is that of \citet{gidel2019variational}, who studied the bilinear case and proved an $O(\kappa\log (1/\epsilon))$ iteration complexity for EG with a better constant than~\citet{mokhtari2020unified}.
  \citet{wei2021linear} studied Optimistic Multiplicative Weights Update (OMWU) for solving constrained bilinear games and established the  linear last-iterate convergence.

Regarding optimal methods, a combination of \citet{ibrahim2020linear} and \citet{zhang2019lower} established a general lower bound, which specializes to a lower bound of $\Omega(\sqrt{\kappa} \log (1/\epsilon))$ for the case of the bilinear minimax game setting.
\citet{azizian2020accelerating} proved linear convergence results for a series of algorithms that achieve this lower bound and also provided an alternative proof for this lower bound by using spectral arguments.
However,~\citet{azizian2020accelerating} did not provide accelerated rates for OGDA and provided an accelerated rate for EG with momentum but with an unknown constant.
In this work, we completely close that gap by providing accelerated convergence rates for (stochastic) EG with relatively tight constants.
In another work, \citet{azizian2020tight} proved a full-regime result for EG without momentum where they show that the $O(\kappa\log (1/\epsilon))$ iteration complexity for EG  is optimal among the methods using a fixed number of composed gradient evaluations and only the last iterate (excluding momentum and restarting).
A similar iteration complexity, (with an unknown constant) can be derived from the seminal work by~\citet{tseng1995linear} on EG.

\noindent\textbf{Stochastic bilinear minimax and variational inequalities.} The standard assumptions made in the literature on stochastic variational inequalities~\citep{nemirovski2009robust,juditsky2011solving} is that the set of parameters and the variance of the stochastic estimate of the vector field are bounded.
These two assumptions do not hold in the stochastic bilinear case, because it is unconstrained and the noise increases with the norm of the parameters.
Recently, \citet{hsieh2020explore} provided results on stochastic EG with different step sizes, without the bounded domain assumption but still requiring the bounded noise assumption. 
{\citet{iusem2017extragradient} and \citet{bot2019forward} studied the independent-sample, minibatch setting where the summation of inverse batchsize converges.} 
\citet{mishchenko2020revisiting} discussed how using the same mini-batch for the two gradients in stochastic EG gives stronger guarantees.
Using a Hamiltonian viewpoint, \citet{loizou2020stochastic} provided the first set of global non-asymptotic last-iterate convergence guarantees for a stochastic game over a non-compact domain, in the absence of strong monotonicity assumptions.
In particular, their stochastic Hamiltonian gradient methods come with last-iterate convergence guarantees in the finite-sum stochastic bilinear game as well.
In our work, we provide an accelerated convergence rate for EG in the bilinear setting with unbounded domain and unbounded noise.

\noindent\textbf{Restarting and acceleration.} {Restarting has long been introduced as an effective approach to accelerate first-order methods in the optimization literature~\citep{o2015adaptive, roulet2020sharpness, renegar2021simple}. In particular, \citet{o2015adaptive} proposed an adaptive restarting technique that significantly improves the convergence rate of Nesterov's accelerated gradient descent method. \citet{roulet2020sharpness} developed optimal restarting methods for solving convex optimization problems that satisfy the sharpness assumption. \citet{renegar2021simple} considered a more general set of problems than \citet{roulet2020sharpness} and presented a simple and near-optimal restarting scheme.
Our variant restarting achieves acceleration via a fundamentally different idea that is inspired by modern variance-reduction ideas.
}

\noindent\textbf{Averaging in convex-concave games.}
\cite{golowich2020last} studied the effect of averaging for EG in the smooth convex-concave setting.
They showed that the last iterate converges at a rate of $O(1/\sqrt{\Tholder})$ in terms of the square root of the Hamiltonian (and also the duality gap), while it is known that iteration averaging enjoys an $O(1/\Tholder)$ rate \citep{nemirovski2004prox}.
A tight lower bound was also proved to justify an assertion of optimality in the last-iterate setting.
Such a result provides a convincing argument in favor of restarting the algorithm from an average of the iterates.  This is a theme that we pursue in the current paper.

\noindent\textbf{Stability of limit points in minimax games.}
GDA dynamics often encounter limit cycles or non-Nash stable limiting points~\citep{daskalakis2018limit,adolphs2019local,berard2019closer,mazumdar2019finding}.
To mitigate this, \citet{adolphs2019local} and \citet{mazumdar2019finding} proposed to exploit the curvature associated with the stable limit points that are not Nash equilibria.  While appealing theoretically, such methods generally involve costly inversion of Jacobian matrices at each step.

\noindent\textbf{Over-parameterized models and interpolation.} Recently it was shown that popular stochastic gradient methods, like SGD and its momentum variants, converge considerably faster when the underlying model is sufficiently over-parameterized as to interpolate the data \citep{gower2019sgd,gower2021sgd,loizou2020momentum,vaswani2019fast, loizou2021stochastic, sebbouh2020convergence}.
In the minimax optimization setting, an analysis that also covers the interpolation regime is rare.
To the best of our knowledge the only paper that provides convergence guarantees for SEG in this setting is \cite{vaswani2019painless}, where SEG with line search is proposed and analyzed.
In our work we provide convergence guarantees in the interpolation regime as corollaries of our main theorems but with a tight $1/e$-prefactor in the linear convergence.

\noindent\textbf{Organization.}
The remainder of this paper is organized as follows.
\S\ref{sec_assumptions} details the basic setup and assumptions for our main results.
\S\ref{sec_SEGg} presents our convergence results for SEG with averaging and restarting.
\S\ref{sec_experiment} provides experiments that validate our theory.
\S\ref{sec_conclusions} concludes this paper with future directions.
All technical analyses along with auxiliary results are relegated to later sections in the supplementary materials.

\noindent\textbf{Notation.}
Throughout this paper we use the following notation.
For two real symmetric matrices, $\Bb_1,\Bb_2$, we denote $\Bb_1\preceq \Bb_2$ when $\mathbf{v}^\top\Bb_1\mathbf{v} \le \mathbf{v}^\top\Bb_2 \mathbf{v}$ holds for all vectors $\mathbf{v}$.
Let $\lambda_{\max}(\Bb)$~(resp.~$\lambda_{\min}(\Bb)$) be the largest (resp.~smallest) eigenvalue of a generic (real symmetric) matrix $\Bb$.
Let $\|\Bb\|_{op}$ denotes the operator norm of $\Bb$.
Let $\mathcal{F}_t$ be the filtration generated by the stochastic samples, $\Bb_{\xi,s}, \gb_{\xi,s}$, $s=1,\dots,t$, in the bilinear game.
Let $\max(a,b)$ or $a\lor b$  denote the maximum value of $a,b\in \RR$, and let $\min(a;b)$ or $a\land b$ denote the minimum.
For two real sequences, $(a_n)$ and $(b_n)$, we write $a_n = O(b_n)$ to mean that $|a_n|\le C b_n$ for a positive, numerical constant $C$, for all $n\ge 1$, and let $a_n = \tilde{O}(b_n)$ mean that $|a_n|\le C b_n$ where $C$ hides a logarithmic factor in relevant parameters.
We also denote $
\widehat{\Mb}_\xi	\equiv	\Bb_\xi^\top \Bb_\xi
$ and $
\Mb_\xi			\equiv	\Bb_\xi\Bb_\xi^\top
$ for brevity, each being positive semi-definite for each realization of $\xi$.
Finally, let $[n] = \{1,\dots,n\}$ for $n$ being a natural number.

\section{SETUP FOR MAIN RESULTS}\label{sec_assumptions}
In this section, we introduce the basic setup and assumptions needed for our statement of the convergence of the stochastic extragradient (SEG) algorithm.
We first make the following assumptions on  $\Bb_\xi$. Let us recall that 
$\widehat{\Mb} \equiv \Exs_\xi\widehat{\Mb}_{\xi} \equiv \Exs_\xi[\Bb_{\xi}^\top \Bb_{\xi}]
$ and $
\Mb \equiv \Exs_\xi\Mb_{\xi} \equiv \Exs_\xi[\Bb_{\xi} \Bb_{\xi}^\top]
$.

\begin{assumption}[Assumption on $\Bb_\xi$]\label{assu_boundednoise_A}
Denote $\Bb = \Exs_\xi[\Bb_{\xi}]$ for $\Bb\in \RR^{n\times m}$ and impose the following regularity conditions:
$
\lambda_{\max}(\Bb^\top\Bb) > 0$
and 
$
\lambda_{\min}(\Mb) \land \lambda_{\min}(\widehat{\Mb}) > 0
$.
We assume that there exist $\sigma_{\Bb}, \sigma_{\Bb,2}\in [0,\infty)$ such that
\beq\label{sigmaAsq}
\begin{aligned}
\| \Exs_\xi[(\Bb_{\xi} - \Bb)^\top (\Bb_{\xi} - \Bb)] \|_{op}
\le
\sigma_{\Bb}^2,
\\
\| \Exs_\xi\left[(\Bb_{\xi} - \Bb) (\Bb_{\xi} - \Bb)^\top\right] \|_{op}
\le
\sigma_{\Bb}^2,
\end{aligned}
\eeq
and
\beq\label{sigmaA2sqinit}
\begin{aligned}
\| \Exs_\xi[\Bb_{\xi}^\top \Bb_{\xi} - \widehat{\Mb}]^2 \|_{op}
\le
\sigma_{\Bb,2}^2,
\\
\| \Exs_\xi[\Bb_{\xi}\Bb_{\xi}^\top - \Mb]^2 \|_{op}
\le
\sigma_{\Bb,2}^2.
\end{aligned}
\eeq
\end{assumption}
The assumption of $n\ge m$ (i.e.~$\Bb$ is tall) is without loss of generality;
we can convert the SEG iterates with a wide coupling matrix to that of its transpose.
Note also $\sigma_{\Bb} = 0$ corresponds to the nonrandom $\Bb_{\xi} = \Bb$ case.
The stochasticity introduced in $\Bb_\xi$ allows us to conclude the first convergence result under the unbounded noise condition.%
\footnote{As a comparison, \citet{hsieh2020explore} only provides a proof for the bounded noise case.}
Next we impose an assumption on the intercept vector $\gb_\xi$.
\begin{assumption}[Assumption on $\gb_\xi$]\label{assu_boundednoise_B}
There exists a $\sigma_{\gb}\in [0,\infty)$ such that
$$
\Exs_\xi\left[ \|\gb_{\xi}^\xholder\|^2 + \|\gb_{\xi}^\yholder\|^2\right]
\le
\sigma_{\gb}^2
< 
\infty
.
$$
Furthermore, we let $\Exs_\xi[\gb^\xholder_\xi] = \mathbf{0}_n$, $\Exs_\xi[\gb^\yholder_\xi] = \mathbf{0}_m$ and assume independence between the stochastic matrix $\Bb_{\xi}$ and the vector  $[\gb^\xholder_\xi; \gb^\yholder_\xi]$.
\end{assumption}
We remark that the independence assumption in Assumption \ref{assu_boundednoise_B} significantly simplifies our analysis.%
\footnote{In practice, such independence can be \textit{approximately} achieved via the following decoupling argument: we formulate the random Jacobian-vector product and the random intercept using two independent random samples, separately.
Note an approximate knowledge of the Nash equilibrium is required in this decoupling argument.
}
In particular, it ensures
$
\Exs[\Bb_{\xi}\gb^\yholder_\xi] = \mathbf{0}_n
$ and $
\Exs[\Bb_{\xi}^\top\gb^\xholder_\xi] = \mathbf{0}_m
$, so the Nash equilibrium is the equilibrium point that the last-iterate SEG oscillates around.
The independence structure of $\Bb_\xi$ and $[\gb^\xholder_\xi; \gb^\yholder_\xi]$ in Assumption~\ref{assu_boundednoise_B} is crucial for our analysis, which is satisfied in certain statistical models.
Specially, when one of the $\Bb_{\xi}$ and $[\gb^\xholder_\xi; \gb^\yholder_\xi]$ is nonrandom this is always satisfied.
Our analysis can be further generalized to more relaxed assumptions on  zero correlation between $[\gb^\xholder_\xi; \gb^\yholder_\xi]$ and the first three moments of $\Bb_{\xi}$, with a second-moment condition similar to $
\Exs_\xi[
\|\Bb_{\xi} \gb_{\xi}^\yholder\|^2
+
\|\Bb_{\xi}^\top \gb_{\xi}^\xholder\|^2
]
	\le
C(\lambda_{\max}(\Mb) \lor \lambda_{\max}(\widehat{\Mb}))\sigma_{\gb}^2$.
We defer the full development of this extension to future work.
With Assumptions~\ref{assu_boundednoise_A} and \ref{assu_boundednoise_B} at hand, we are ready to state our main results on the convergence of SEG variants.

\section{SEG WITH AVERAGING AND RESTARTING}\label{sec_SEGg}
Recall that in contrast to SGD theory in convex optimization, the last iterate of SEG does \emph{not} converge to an arbitrarily small neighborhood of the Nash equilibrium even for the case of a converging step size~\citep{hsieh2020explore}.
We accordingly turn to an analysis of the \emph{averaged iterate} of $\xholder_t$ and $\yholder_t$, $t=0,1,\dots,\Tholder$, denoted as
\beq\label{averaged_iterate}\begin{aligned}
\overline{\xholder}_{\Tholder}
	\equiv
\frac{1}{\Tholder+1} \sum_{t=0}^\Tholder \xholder_t
,
\quad
\overline{\yholder}_{\Tholder}
	\equiv
\frac{1}{\Tholder+1} \sum_{t=0}^\Tholder \yholder_t
.
\end{aligned}\eeq
For simplicity we focus on the case in which $\Bb_{\xi}, \Bb$ are square matrices. 
Let us define $\eta_{\Mb}$ as follows, which is the maximal step size that the SEG algorithm analysis takes:
\beq\label{etaMdef}
\eta_{\Mb}
\equiv
\frac{1}{\sqrt{\rho_{1}\lor\rho_{2}
}},
\eeq
where $\rho_{1} = \lambda_{\max}\big(\Mb^{-1/2} [\Exs_\xi \Mb_\xi^2] \Mb^{-1/2}\big)$ and $\rho_{2} = \lambda_{\max}\big(\widehat{\Mb}^{-1/2} [\Exs_\xi \widehat{\Mb}_\xi^2] \widehat{\Mb}^{-1/2}\big)$.
We introduce the following variants:
\beq\label{eta_choice}\begin{aligned}
\hat\eta_{\Mb}(\alpha)
&\equiv
\frac{\eta_{\Mb}}{\sqrt{2}}
\land
\frac{\alpha\lambda_{\min}(\Bb\Bb^\top)}{2\sigma_{\Bb}^2 \sqrt{\lambda_{\max}(\Bb^\top\Bb)}}
,
\\
\bar\eta_{\Mb}(\alpha)
&\equiv
\eta_{\Mb}
\land
\frac{\alpha\lambda_{\min}(\Bb\Bb^\top)}{2\sigma_{\Bb}^2 \sqrt{\lambda_{\max}(\Bb^\top\Bb)}}
,
\end{aligned}\eeq
which reduce to $1/\sqrt{2\lambda_{\max}(\Bb^\top\Bb)}$ and $1/\sqrt{\lambda_{\max}(\Bb^\top\Bb)}$ when $\Bb_\xi$ is nonrandom.
We state our first main result on SEG with iteration averaging, Theorem \ref{theo_SEG_B}, whose proof is provided in \S\ref{sec_proof,theo_SEG_B}:

\begin{theorem}[SEG Averaged Iterate]\label{theo_SEG_B}
Let Assumptions~\ref{assu_boundednoise_A} and \ref{assu_boundednoise_B} hold with $n=m$.
Prescribing an $\alpha\in (0,1)$, when the step size $\eta$ is chosen as $\hat\eta_{\Mb}(\alpha)$ as defined in Eq.~\eqref{eta_choice}, we have for all $\Tholder\ge 1$ the following convergence bound for the averaged iterate:
\beq\label{Asingleloop_SEG_B}
\begin{aligned}
&\Exs\big[
\|\overline{\xholder}_{\Tholder}\|^2 + 
\|\overline{\yholder}_{\Tholder}\|^2
\big]
\\
&\le\,\,\tau_{1}
\cdot
\frac{\|\xholder_0\|^2 + \|\yholder_0\|^2}{\color{red}(\Tholder+1)^2}
+\tau_{2}
\cdot
\frac{\sigma_{\gb}^2}{\color{red}\Tholder+1},
\end{aligned}
\eeq
where $\tau_1, \tau_2$ depending on $\sigma_{\Bb}, \sigma_{\Bb,2}$ are defined as
$$\begin{aligned}
\tau_1
&=
\frac{16 + 8\kappa_\zeta}{(1-\alpha)\hat\eta_{\Mb}(\alpha)^2\lambda_{\min}(\Bb\Bb^\top)}
,
\\
\tau_2
&=
\frac{18 + 12\kappa_\zeta}{(1-\alpha)\lambda_{\min}(\Bb\Bb^\top)}
,
\end{aligned}$$
and $
\kappa_\zeta
    \equiv
\frac{\sigma_{\Bb}^2 + \hat\eta_{\Mb}(\alpha)^2\sigma_{\Bb,2}^2}{\lambda_{\min}(\Mb)\land\lambda_{\min}(\widehat{\Mb})}
$ denotes the effective noise condition number of problem Eq.~\eqref{Sminimax}.
\end{theorem}

Measured by the Euclidean metric, Theorem \ref{theo_SEG_B} indicates an $O(1/\sqrt{\Tholder})$ leading-order convergence rate for the averaged iterate of SEG in the general stochastic setting, which is known to be statistically optimal up to a constant multiplier.
We provide detailed comparisons with previous related work in \S\ref{sec_comp,theo_SEG_B}.
Nevertheless, the iteration slowly forgets initial conditions at a polynomial rate, and this result can be improved if we utilize a restarting scheme and take advantage of the knowledge of the smallest eigenvalue of $\Bb\Bb^\top$.
Indeed, in the following result, we boost the convergence rate shown in Eq.~\eqref{Asingleloop_SEG_B}, when the smallest eigenvalue $\lambda_{\min}(\Bb\Bb^\top)$ is available to the system, via a novel  restarting procedure at specific times.
The rationale behind this analysis is akin to that used in boosting sublinear convergence in convex optimization to  linear convergence when the designer has (an estimate of) the strong convexity parameter.

We now develop this argument in detail. 
We continue to assume the case of square matrices $\Bb_\xi,\Bb$.
In Algorithm \ref{algo_iasgd_restart} we run SEG with averaging and restart the iteration at chosen timestamps, $\{\mathcal{T}_i\}_{i\in [\Epoch-1]}\subseteq [\Tholder]$, initializing at the averaged iterate of the previous epoch.
The principle behind our choice of parameters in this algorithm is that we trigger the restarting when the expected squared Euclidean metric $\Exs\left[ \|\xholder_{\Tholder}\|^2 + \|\yholder_{\Tholder}\|^2 \right]$ decreases by a factor of $1/e^2$,
and we halt the restarting procedure once the last iterate reaches stationarity in squared Euclidean metric in the sense of Theorem \ref{theo_SEG_A}:\footnote{
The choice of the discount factor $1/e^2$ is to be consistent with our optimal choice in the interpolation setting, where in the $\sigma_{\Bb} = 0$ case the total complexity is minimized to $e\sqrt{\lambda_{\max}(\Bb^\top\Bb) / \lambda_{\min}(\Bb\Bb^\top)}$.
}
$$
\|\xholder_0\|^2 + \|\yholder_0\|^2
	\approx
\frac{3\sigma_{\gb}^2}{\lambda_{\min}(\Mb)\land\lambda_{\min}(\widehat{\Mb})}.
$$
Given these choices, summarized in Algorithm \ref{algo_iasgd_restart}, we obtain the following theorem:

\begin{algorithm}[!t]
\caption{Iteration Averaged SEG with Scheduled Restarting}
\begin{algorithmic}[1]
\REQUIRE Initialization $\xholder_0$, step sizes $\eta_t$, total number of iterates $\Tholder$, restarting timestamps $\{\mathcal{T}_i\}_{i\in [\Epoch-1]}\subseteq [\Tholder]$ with the total number of epoches $\Epoch\ge 1$, index $s\leftarrow 0$
\FOR{$t=1, 2,\dots,\Tholder$}\label{linerestart}
\STATE
$s\leftarrow s+1$
\STATE
Update $\xholder_t$, $\yholder_t$ via Eq.~\eqref{SEGupdate}
\STATE
Update $\hat{\xholder}_t$, $\hat{\yholder}_t$ via 
$$\begin{aligned}
\hspace{-.1in}
\hat{\xholder}_{t}	\leftarrow	\frac{s-1}{s}\hat{\xholder}_{t-1} + \frac{1}{s} \xholder_{t},
\,\,
\quad
\hat{\yholder}_{t}	\leftarrow	\frac{s-1}{s}\hat{\yholder}_{t-1} + \frac{1}{s} \yholder_{t}
\end{aligned}$$
\IF{$t\in \{\mathcal{T}_i\}_{i\in [\Epoch-1]}$}
	\STATE
	Overload $\xholder_{t}\leftarrow \hat{\xholder}_{t}$, $\yholder_{t}\leftarrow \hat{\yholder}_{t}$, and set $s\leftarrow 0$
	\hfill \text{//restarting procedure is triggered}
\ENDIF
\ENDFOR
\STATE {\bfseries Output:}
$\hat{\xholder}_{\Tholder}, \hat{\yholder}_{\Tholder}$
\end{algorithmic}
\label{algo_iasgd_restart}
\end{algorithm}

\begin{theorem}[SEG with Averaging/Restarting]\label{theo_SEG_C}
Let Assumptions~\ref{assu_boundednoise_A} and \ref{assu_boundednoise_B} hold with $n=m$.
For any prescribed $\alpha\in (0,1)$, choose the step size $\hat\eta_{\Mb}(\alpha)$ as in Eq.~\eqref{eta_choice} and assume a proper restarting schedule.
For all $\Tholder \ge \Tholder_{\operatorname{complexity}}+1$ we have the following convergence bound for the output $\hat\xholder_{\Tholder}, \hat\yholder_{\Tholder}$ of Algorithm \ref{algo_iasgd_restart}:
\beq\label{Asingleloop_SEG_C}\begin{aligned}
\Exs\left[
\|\hat{\xholder}_{\Tholder}\|^2 + \|\hat{\yholder}_{\Tholder}\|^2
\right]
\le
C_1
\cdot
\frac{\sigma_{\gb}^2}{\color{red} \Tholder - \Tholder_{\operatorname{complexity}} + 1}
,
\end{aligned}
\eeq
where  
$$C_1 \equiv \frac{18}{(1-\alpha)\lambda_{\min}(\Bb\Bb^\top)}\cdot\Bigg[
1 + 
\underbrace{
\frac{
O(\sigma_{\Bb}^2 + \hat\eta_{\Mb}(\alpha)^2\sigma_{\Bb,2}^2)}{
 \lambda_{\min}(\Mb) \land \lambda_{\min}(\widehat{\Mb})
}
}_{\text{higher-order term $O(\kappa_\zeta)$}}
\Bigg],$$
where $\Tholder_{\operatorname{complexity}}$ is the fixed \emph{burn-in complexity} defined as
\beq\label{Tcomp_prime}
\begin{aligned}
\frac{
\mbox{logarithmic factor}
}{
\frac{1}{e}\sqrt{(1-\alpha)\bar\eta_{\Mb}(\alpha)^2\lambda_{\min}(\Bb\Bb^\top)}
- C_2
}
,
\end{aligned}
\eeq
with $C_2$ being 
$$O\Big(
\bar\eta_{\Mb}(\alpha)^{3/2}
(\lambda_{\min}(\Bb\Bb^\top))^{1/4}
\sqrt{\sigma_{\Bb}^2 + \bar\eta_{\Mb}(\alpha)^2\sigma_{\Bb,2}^2}
\Big).$$
\end{theorem}
The proof of Theorem \ref{theo_SEG_C} is provided in \S\ref{sec_proof,theo_SEG_C}.
Here we not only achieve the optimal $O(1/\sqrt{\Tholder})$ convergence rate for the averaged iterate, but the proper restarting schedule allows us to achieve a convergence rate bound for iteration-averaged SEG in Eq.~\eqref{Asingleloop_SEG_C} that forgets the initialization at an exponential rate instead of the polynomial rate that is obtained without restarting [cf.\ Theorem \ref{theo_SEG_C}].

Finally, we consider the interpolation setting, where the noise vanishes at the Nash equilibrium.
That is, $\gb^\xholder_\xi = \mathbf{0}_n$ and $\gb^\yholder_\xi = \mathbf{0}_m$; i.e.~$\sigma_{\gb} = 0$ in Assumption \ref{assu_boundednoise_B}.
In that setting, we prove that SEG with iteration averaging achieves an accelerated linear convergence rate.
Set the (constant) interval length of restarting timestamps $\Tholder_{\operatorname{thres}}(\alpha)$ as
\beq\label{Tholder_epoch}
\frac{2}{
\frac{1}{e}\sqrt{(1-\alpha)\bar\eta_{\Mb}(\alpha)^2\lambda_{\min}(\Bb\Bb^\top)}
- C_3
}
,
\eeq
with $C_3$ being 
$$O\left(
\bar\eta_{\Mb}(\alpha)^{3/2}
(\lambda_{\min}(\Bb\Bb^\top))^{1/4}
\sqrt{\sigma_{\Bb}^2 + \bar\eta_{\Mb}(\alpha)^2\sigma_{\Bb,2}^2}
\right).$$
We present an analysis of this algorithm in the following theorem, which can be seen as a corollary of Theorem \ref{theo_SEG_C} but benefits from a refined analysis where tight constant prefactor sits in each term of the bound:

\begin{figure*}[!tb]
\begin{center}
\subfigure[General setting.]{\includegraphics[width=0.4\textwidth]{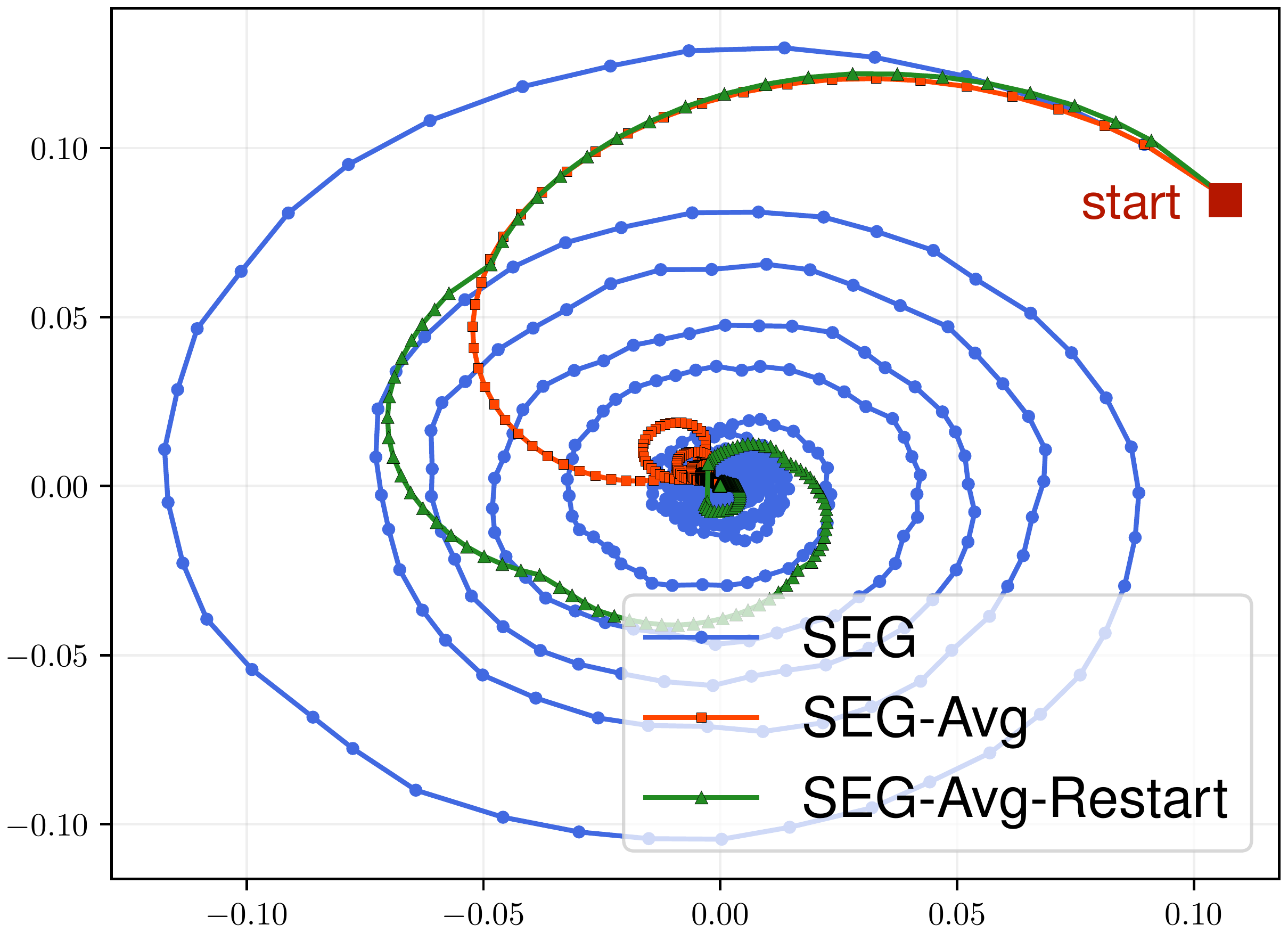}
}
\subfigure[Interpolation setting.]{
\includegraphics[width=0.4\textwidth]{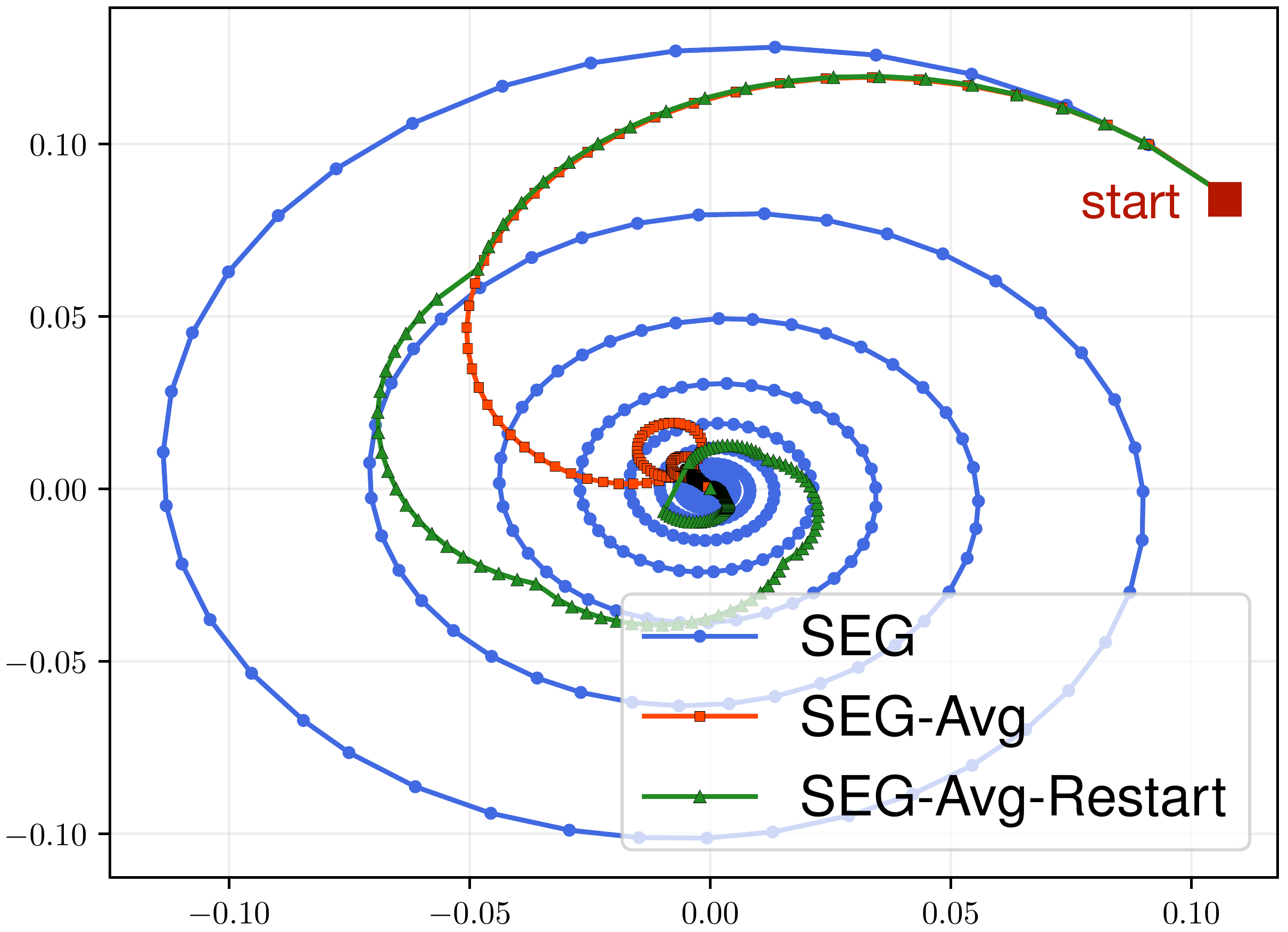}
}
\vspace{-0.1in}
\caption{Illustration (in two dimensions) of the stochastic extragradient (SEG) algorithm, stochastic extragradient with iteration averaging (SEG-Avg), and stochastic extragradient with restarted iteration averaging (SEG-Avg-Restart) on the stochastic minimax optimization problem defined in Eq.~\eqref{Sminimax}.  Here the Nash equilibrium is $[\xholder^*; \yholder^*] = [\mathbf{0}_n; \mathbf{0}_m]$. 
(\textbf{a}) General setting. 
(\textbf{b}) Interpolation setting, where noise vanishes at the Nash equilibrium.
}
\end{center}
\vspace{-0.25in}
\end{figure*}

\begin{figure*}[t]
  \begin{center}
    \subfigure[\label{fig:general}General setting.]{
    \includegraphics[width=0.475\textwidth]{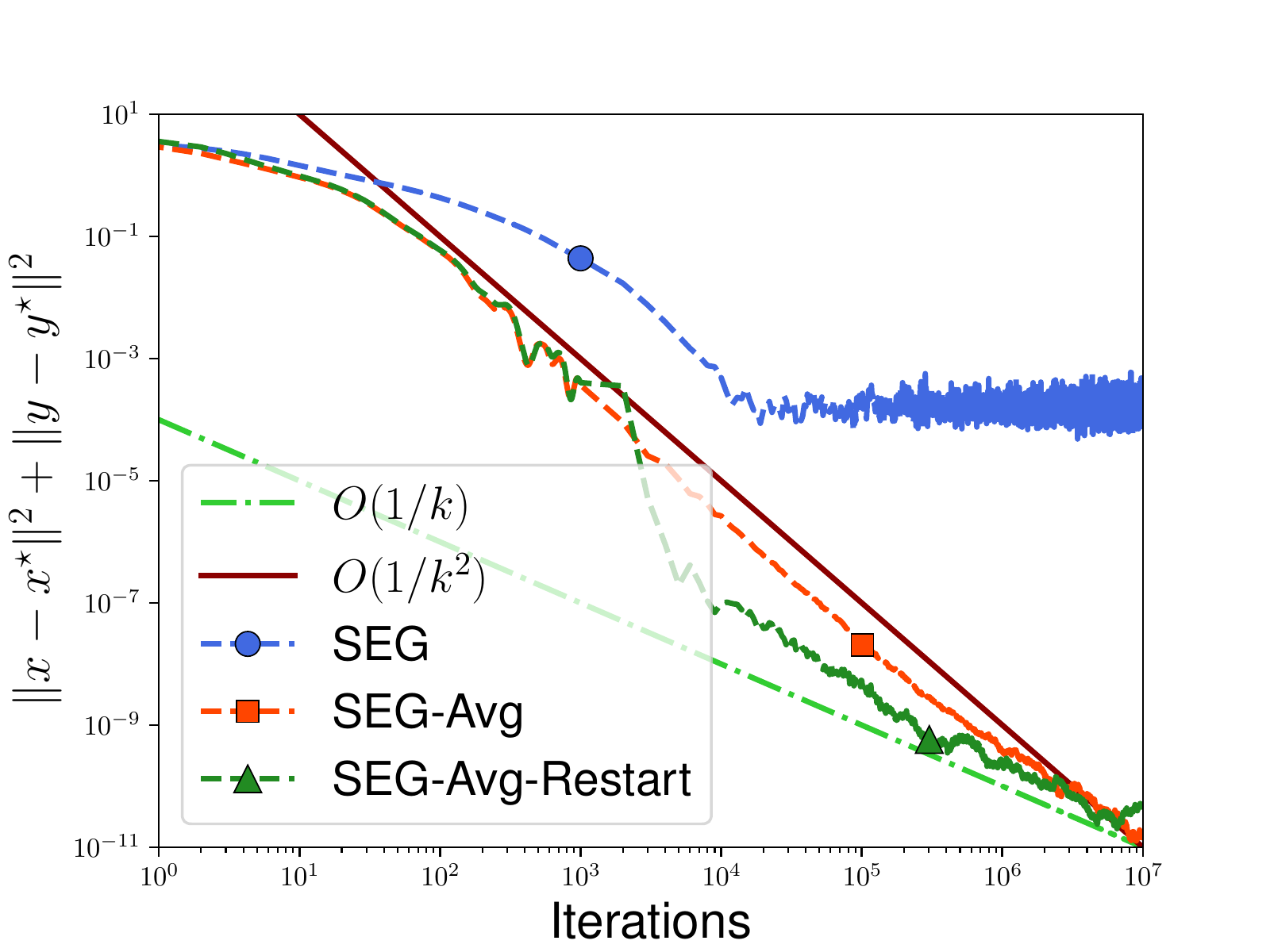}
    }
    \subfigure[\label{fig:interpolation}Interpolation setting.]{
    \includegraphics[width=0.475\textwidth]{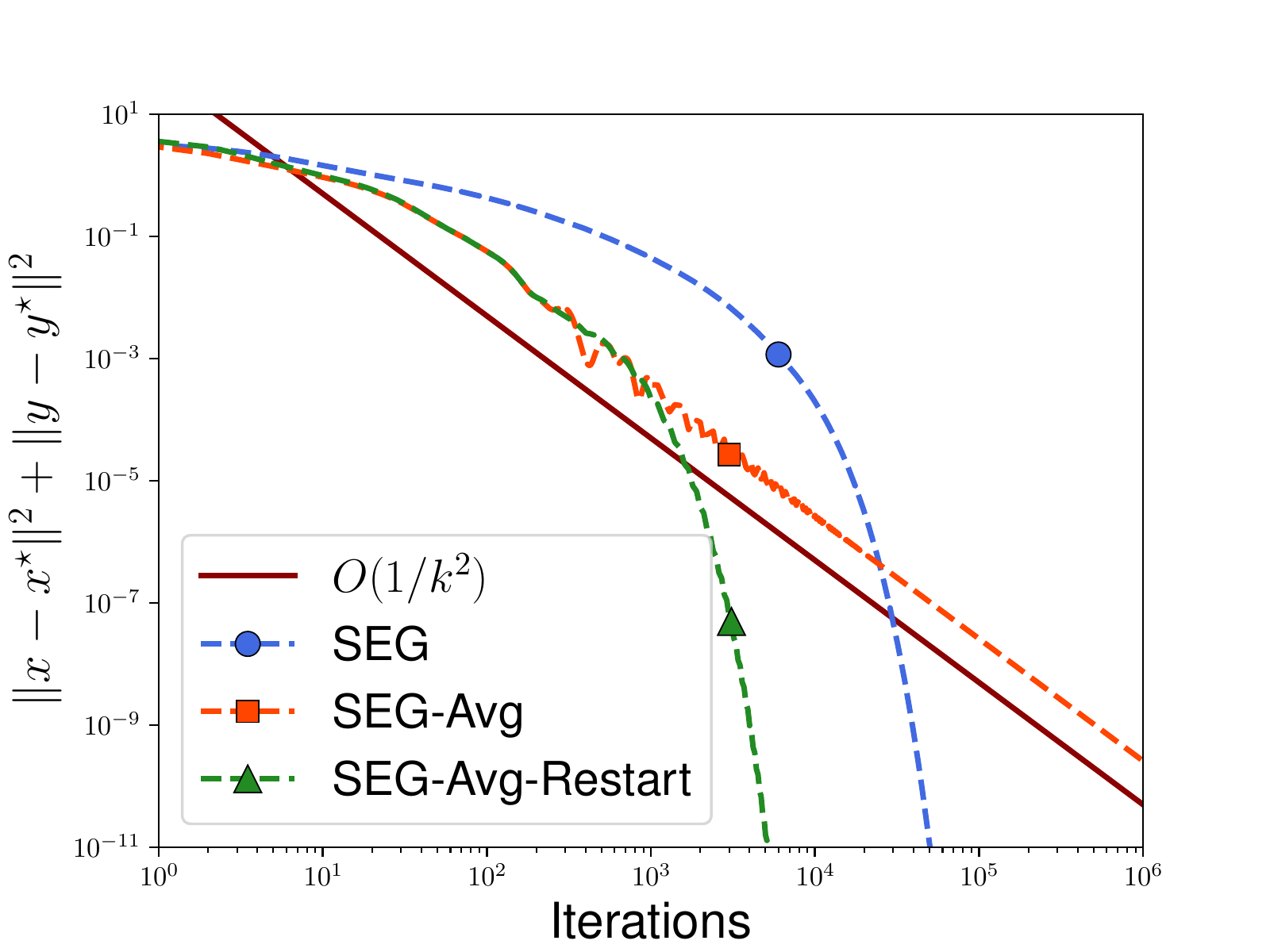}
    }
    \caption{Comparing SEG, SEG-Avg, and SEG-Avg-Restart on a stochastic bilinear optimization problem. The horizontal axis represents the iteration number, and vertical axis represents the square $\ell_2$-distance to the Nash equilibrium. 
    (\textbf{a}) General setting ($d=100, \operatorname{std}_\Bb=0.1, \operatorname{std}_\gb=0.01$). 
    (\textbf{b}) Interpolation setting ($d=100, \operatorname{std}_\Bb=0.1, \operatorname{std}_\gb=0.0$).}
  \end{center}
\vspace{-0.2in}
\end{figure*}

\begin{theorem}[Interpolation Setting]\label{theo_SEGg_interpolation_C}
Let Assumptions~\ref{assu_boundednoise_A} and \ref{assu_boundednoise_B} hold with $n=m$ and $\sigma_{\gb} = 0$.
For any prescribed $\alpha\in (0,1)$ choosing the step size $\eta = \bar\eta_{\Mb}(\alpha)$ as in Eq.~\eqref{eta_choice} and the restarting timestamps $\mathcal{T}_i = i\cdot \Tholder_{\operatorname{thres}}(\alpha)$ where $\Tholder_{\operatorname{thres}}(\alpha)$ was defined as in Eq.~\eqref{Tholder_epoch}, we conclude for all $\Tholder \ge 1$ that is divisible by $\Tholder_{\operatorname{thres}}(\alpha)$ the following convergence rate for the output $\hat\xholder_{\Tholder}, \hat\yholder_{\Tholder}$ of Algorithm \ref{algo_iasgd_restart}:
\begin{align}\label{Asingleloop_SEG_interpolation_C}
&\Exs \left[\|\hat\xholder_{\Tholder}\|^2 + \|\hat\yholder_{\Tholder}\|^2\right]
\\
\le\,\, & e^{-
\frac{\Tholder}{e}\sqrt{(1-\alpha)\bar\eta_{\Mb}(\alpha)^2\lambda_{\min}(\Bb\Bb^\top)}
+ C_4
}
\left[\|\xholder_0\|^2 + \|\yholder_0\|^2\right]
\nonumber
,
\end{align}
with $C_4$ being
$$
O\left(
\Tholder\bar\eta_{\Mb}(\alpha)^{3/2}
(\lambda_{\min}(\Bb\Bb^\top))^{1/4}
\sqrt{\sigma_{\Bb}^2 + \bar\eta_{\Mb}(\alpha)^2\sigma_{\Bb,2}^2}
\right).
$$
\end{theorem}
The proof of Theorem \ref{theo_SEGg_interpolation_C} is provided in \S\ref{sec_proof,theo_SEGg_interpolation_C}.
The idea behind Theorem \ref{theo_SEGg_interpolation_C} is, in plain words, to trigger restarting whenever the last-iterate SEG has travelled through a full cycle, giving insights on the design of $\Tholder_{\operatorname{thres}}(\alpha)$ in the restarting mechanism.
Compared with Eq.~\eqref{xygrowth_SEG_A} in Theorem \ref{theo_SEG_A} with $\sigma_{\gb}$ equal to zero, the contraction rate (in terms of its exponent) to the Nash equilibrium $-
\frac{\eta_{\Mb}^2}{4}\cdot\left( \lambda_{\min}(\Mb)\land\lambda_{\min}(\widehat{\Mb})
\right)
$  improves to $-
\frac{1}{e}\sqrt{(1-\alpha)\bar\eta_{\Mb}(\alpha)^2\lambda_{\min}(\Bb\Bb^\top)}
$ plus higher-order moment terms involving $\Bb_\xi$.
It is worth mentioning that Algorithm \ref{algo_iasgd_restart} achieves this accelerated convergence rate in Eq.~\eqref{Asingleloop_SEG_interpolation_C} via simple restarting and does \emph{not} require an explicit Polyak- or Nesterov-type momentum update rule \citep{NESTEROV[Lectures]}.
In the case of nonrandom $\Bb_\xi$, this rate matches the lower bound \citep{ibrahim2020linear,zhang2019lower},%
\footnote{\citet{ibrahim2020linear} paper provides the stated lower bound $\sqrt{\kappa}\log(1/\epsilon)$. Although the argument in \citet{zhang2019lower} does not achieve this bound directly (since they did not consider the bilinear-coupling case), modifying their arguments easily extends it to the same lower bound in the bilinear-coupling case. Theorem~\ref{theo_SEGg_interpolation_C} matches this lower bound in the nonrandom case.}
and the only algorithm that achieves this optimal rate to our best knowledge is \citet{azizian2020accelerating} without an explicit $1/e$-prefactor on the right hand of Eq.~\eqref{Asingleloop_SEG_interpolation_C}.

We end this section with some remarks.
For the results in this section, we can forgo fully optimizing the prefactor over $\alpha$ and simply set a step size $\eta$ as in Eq.~\eqref{eta_choice}.
Both the analyses of Theorems \ref{theo_SEG_B} and \ref{theo_SEG_C} adopt a step size of $\eta_{\Mb}/\sqrt{2}$, capped by some $\alpha$-dependent threshold, due to the fact that our analysis relies heavily on the last-iterate convergence to stationarity.
In the meantime, Theorem \ref{theo_SEGg_interpolation_C} does not rely on such an argument and accommodates the larger (thresholded) $\eta_{\Mb}$ as the step size.
Lastly, we emphasize that the knowledge of $\lambda_{\min}(\Bb\Bb^\top)$ is required for the algorithm to achieve the accelerated rate.
Considerations regarding such knowledge are related to the topic of adaptivity of stochastic gradient algorithms~\citep[see, e.g.,][]{lei2020adaptivity}.

\begin{figure*}[t]
  \begin{center}
    \subfigure[\label{fig:d=100}$d=100$]{
    \includegraphics[width=0.31\textwidth]{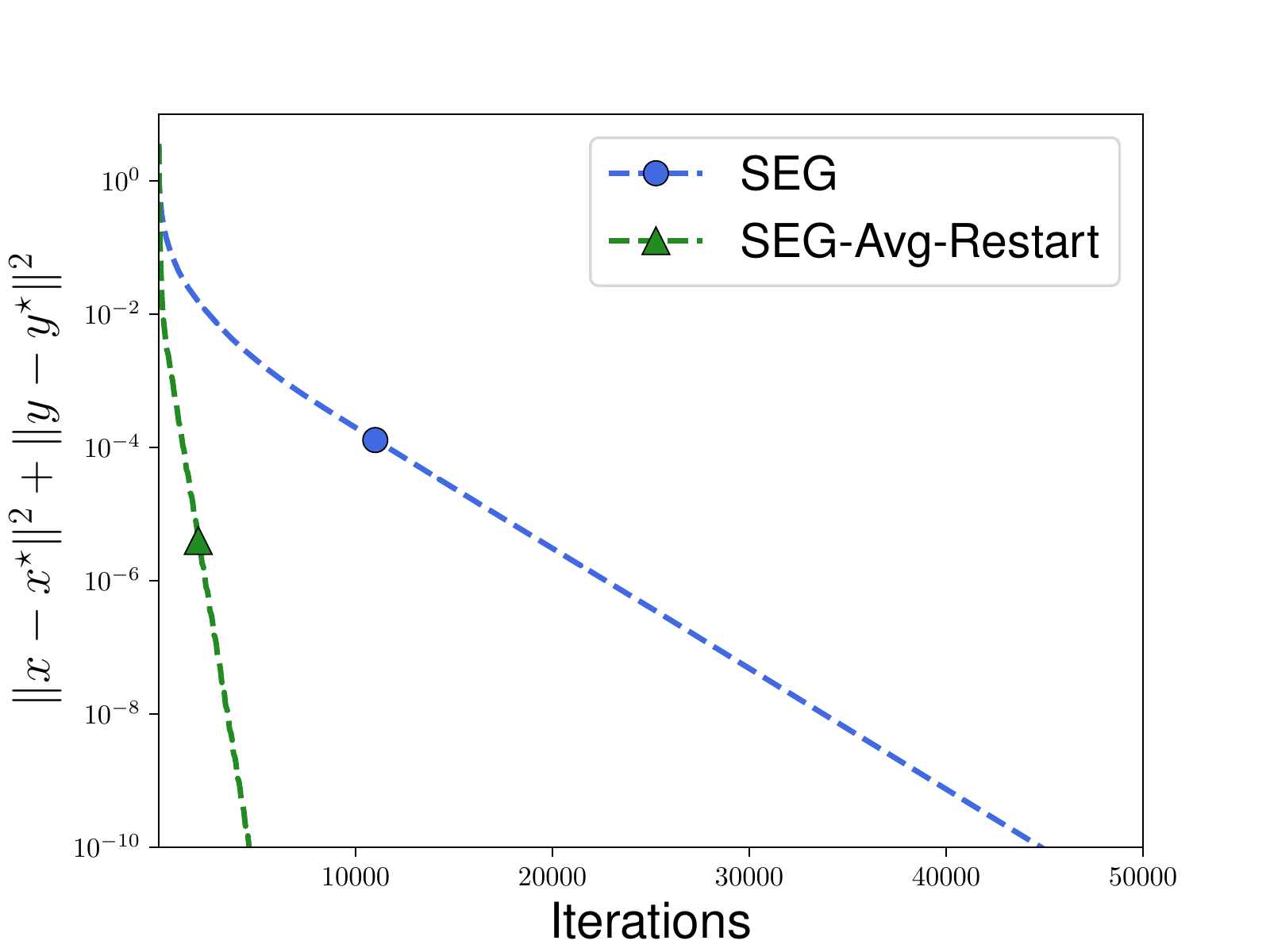}}
    \subfigure[\label{fig:d=200}$d=200$]{
    \includegraphics[width=0.31\textwidth]{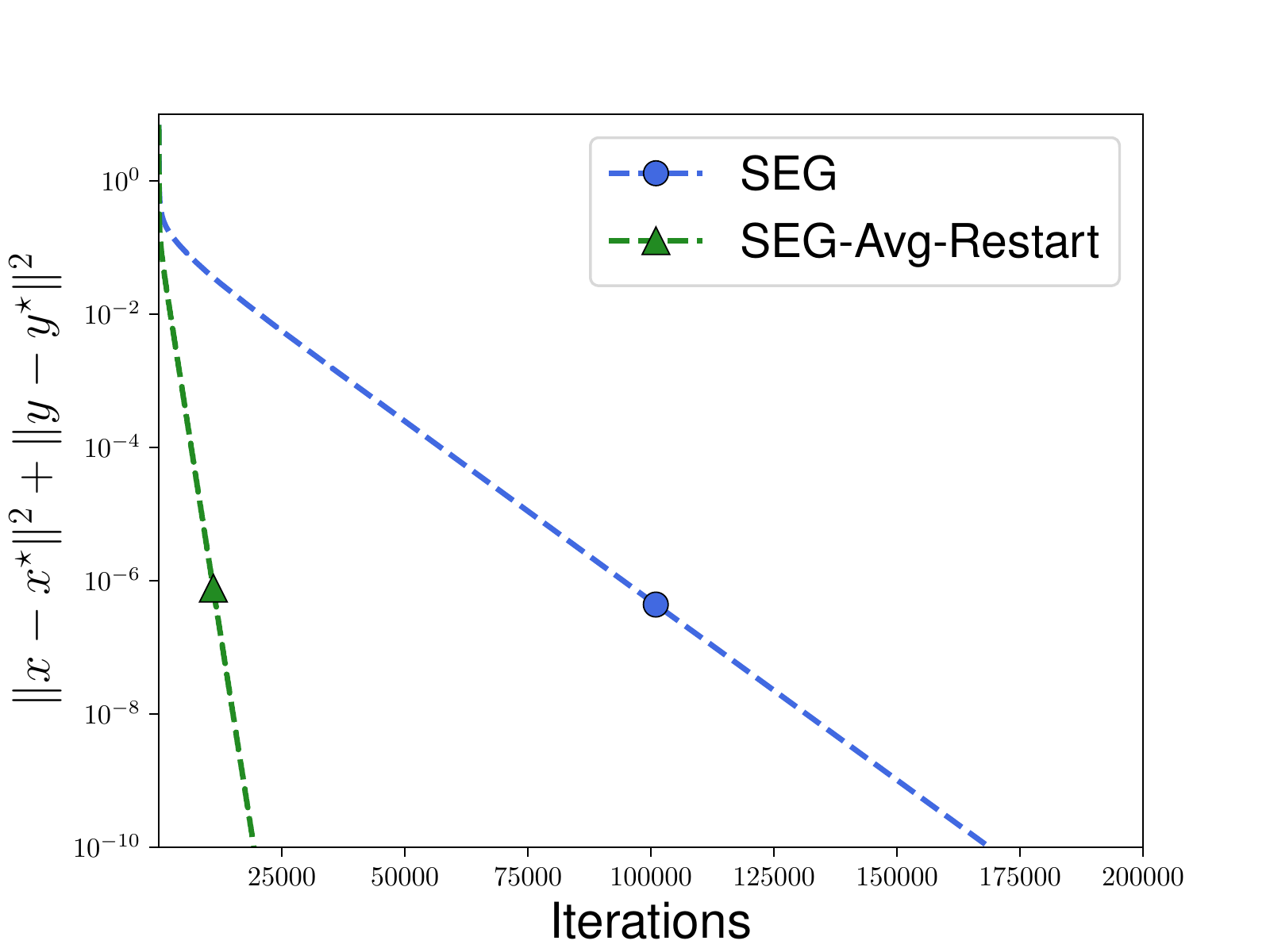}}
    \subfigure[\label{fig:seg-avg-restart-zoomin}SEG-Avg-Restart ($d=200$)]{
    \includegraphics[width=0.31\textwidth]{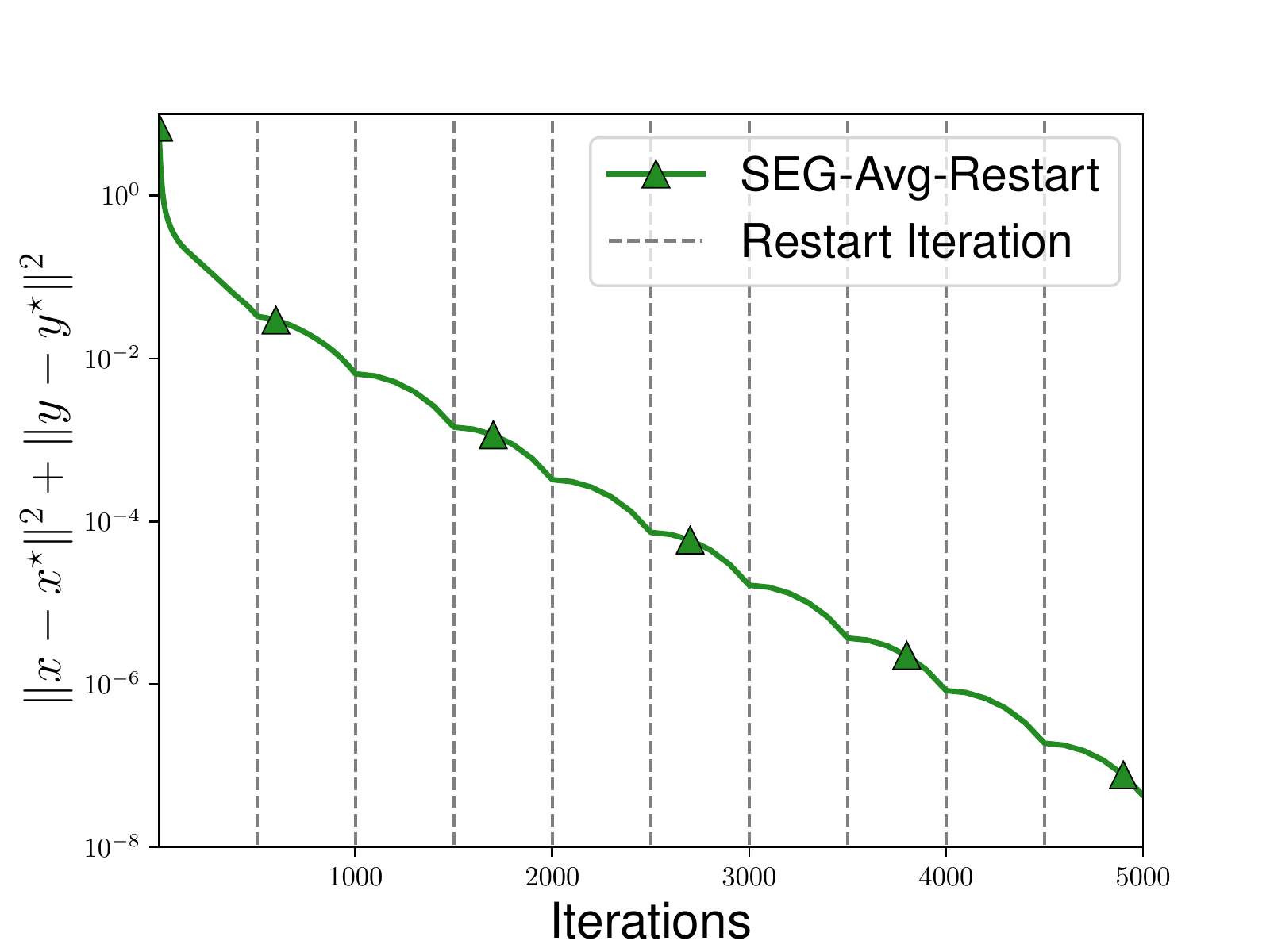}}
    \caption{Comparing SEG and SEG-Avg-Restart on a stochastic bilinear optimization problem in the interpolation setting. The horizontal axis represents the iteration number, and the vertical axis represents the squared $\ell_2$-distance to the Nash equilibrium. 
    (\textbf{a}) Comparison on dimension $d=100$  $(\operatorname{std}_\Bb=0.1, \operatorname{std}_\gb=0.0$). 
    (\textbf{b}) Comparison on dimension $d=200$  $(\operatorname{std}_\Bb=0.1, \operatorname{std}_\gb=0.0$). (\textbf{c}). Zoomed-in visualization of SEG-Avg-Restart on dimension $d=200$ $(\operatorname{std}_\Bb=0.1, \operatorname{std}_\gb=0.0$).}
  \end{center}
  \label{fig:compare-seg-seg-avg-restart-interpolation}
  \vspace{-0.15in}
\end{figure*}

\begin{figure*}[h]
\begin{center}
\subfigure[\label{fig:seg-step-size}Different step size $\eta$.]{
\includegraphics[width=0.45\textwidth]{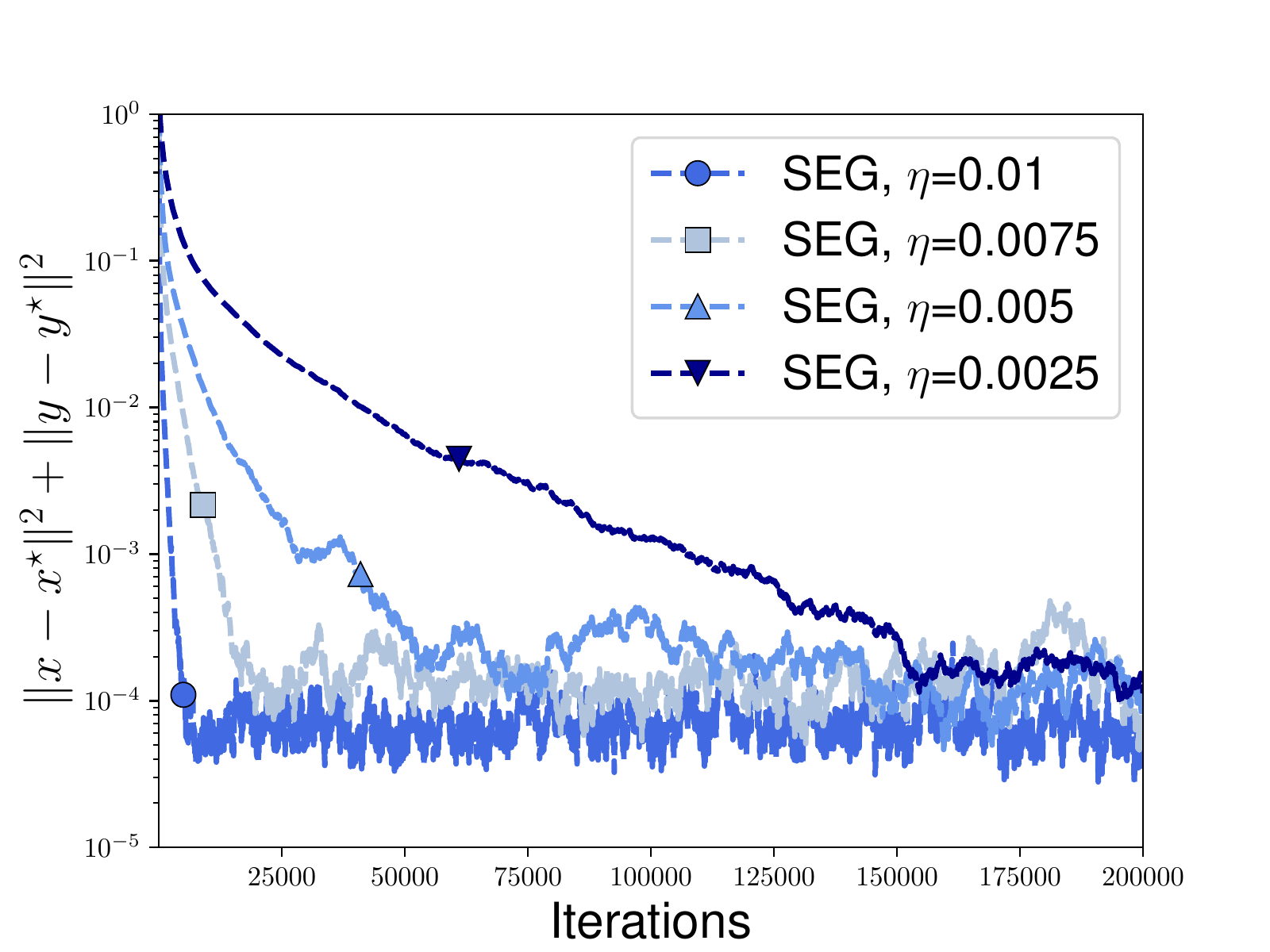}
}
\hspace{.1in}
\subfigure[\label{fig:seg-noise}Different noise $\operatorname{std}_\gb$.]{
\includegraphics[width=0.45\textwidth]{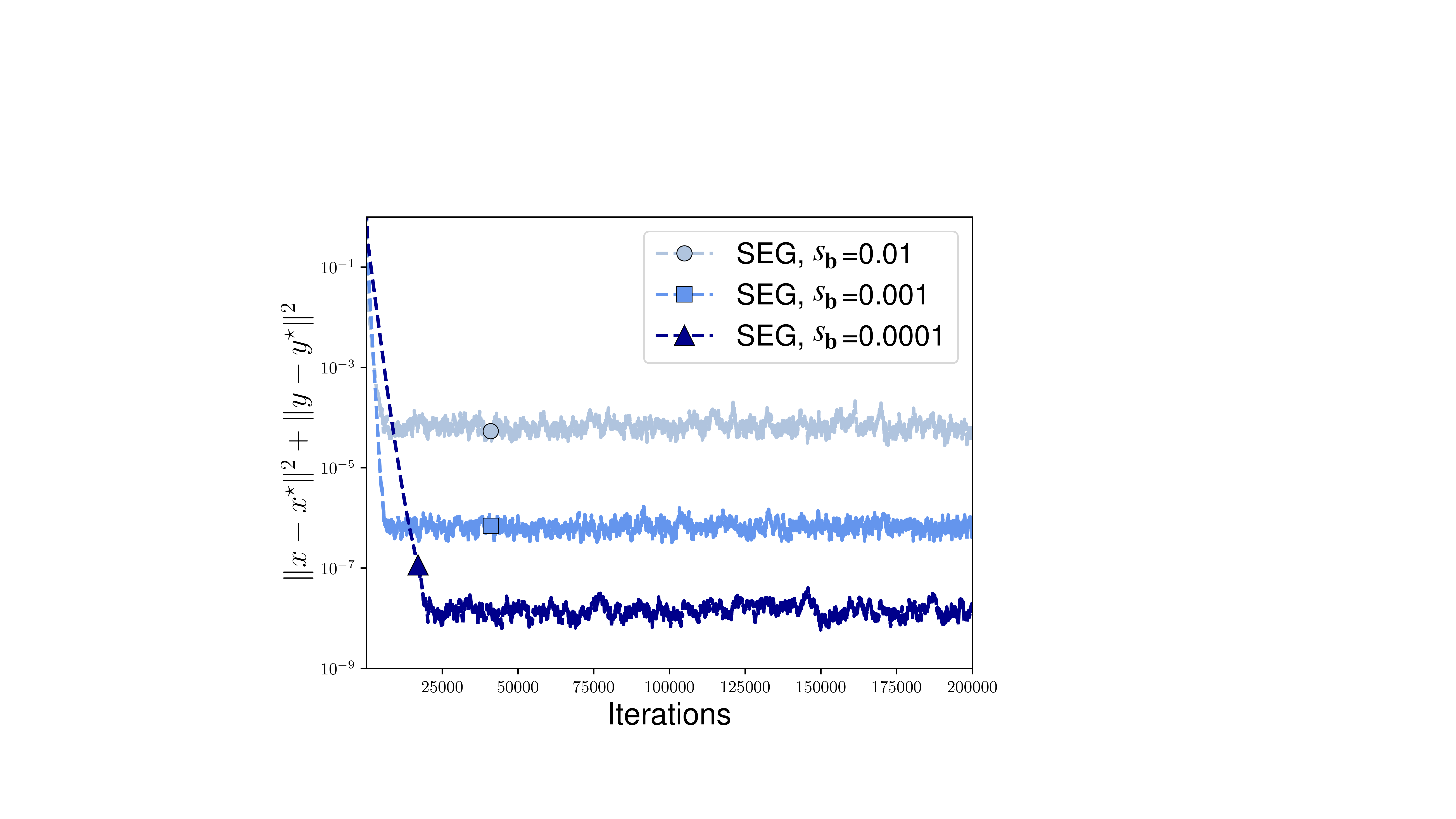}
}
\caption{Comparison of SEG (without averaging) with different step sizes $\eta$ and  noise magnitudes $\operatorname{std}_\gb$ on a stochastic bilinear optimization problem in the general setting. The horizontal axis represents the iteration number, and the vertical axis represents the squared $\ell_2$-distance to the Nash equilibrium. 
(\textbf{a}) Comparison with respect to varying step size $\eta \in \{0.01, 0.0075, 0.005, 0.0025\}$  $(\operatorname{std}_\Bb=0.1, \operatorname{std}_\gb=0.01$). 
(\textbf{b}) Comparison with respect to varying noise $\operatorname{std}_\gb \in \{0.01, 0.001, 0.0001\}$   with step size $\eta=0.01$ $(\operatorname{std}_\Bb=0.1)$.}
\end{center}
\vspace{-0.3in}
\end{figure*}

\begin{figure*}[h]
  \begin{center}
    \subfigure[\label{fig:general-DSEG}General setting.]{
    \includegraphics[width=0.31\textwidth]{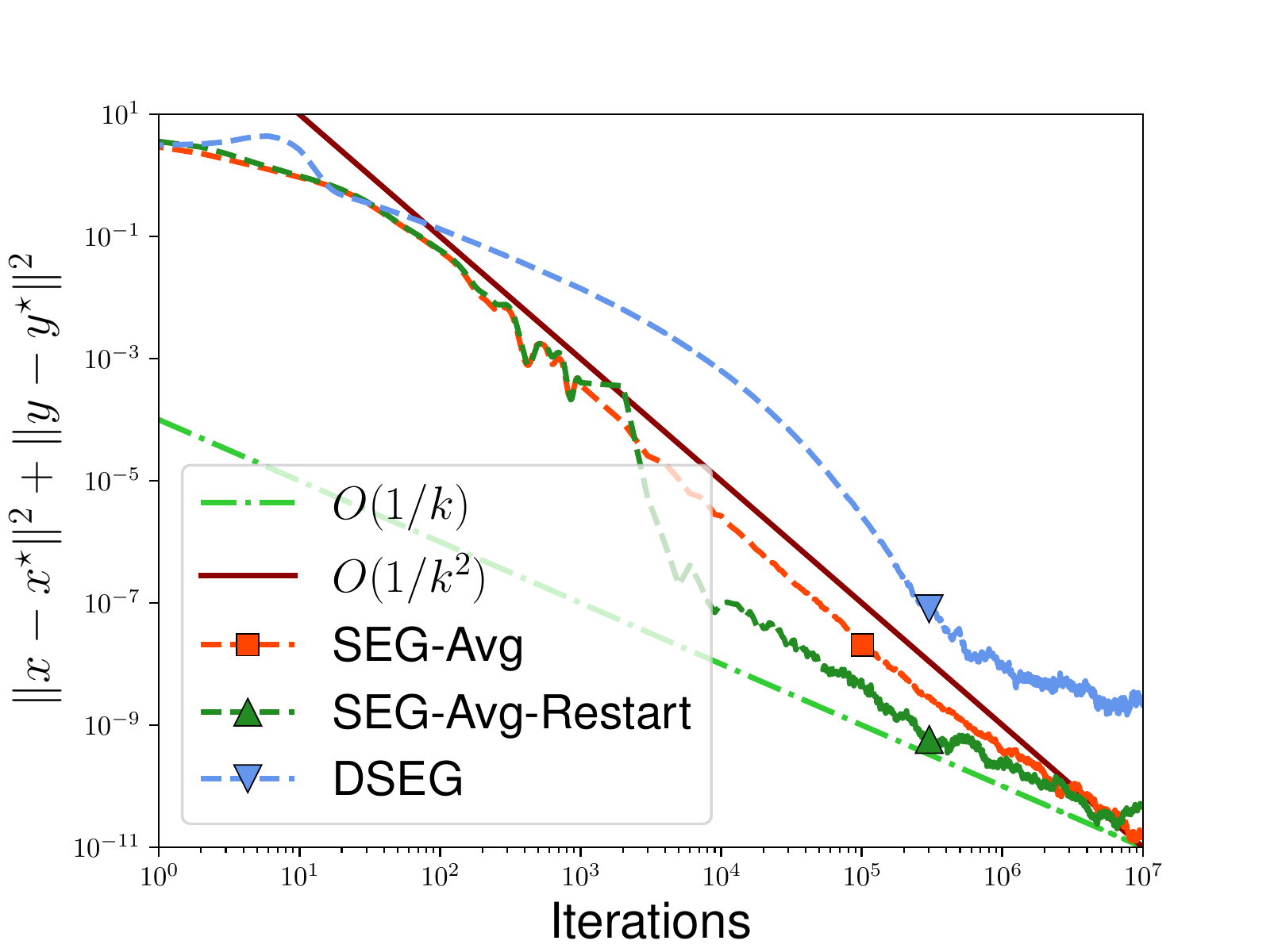}
    }
    \subfigure[\label{fig:interpolation-DSEG}Interpolation setting.]{
    \includegraphics[width=0.31\textwidth]{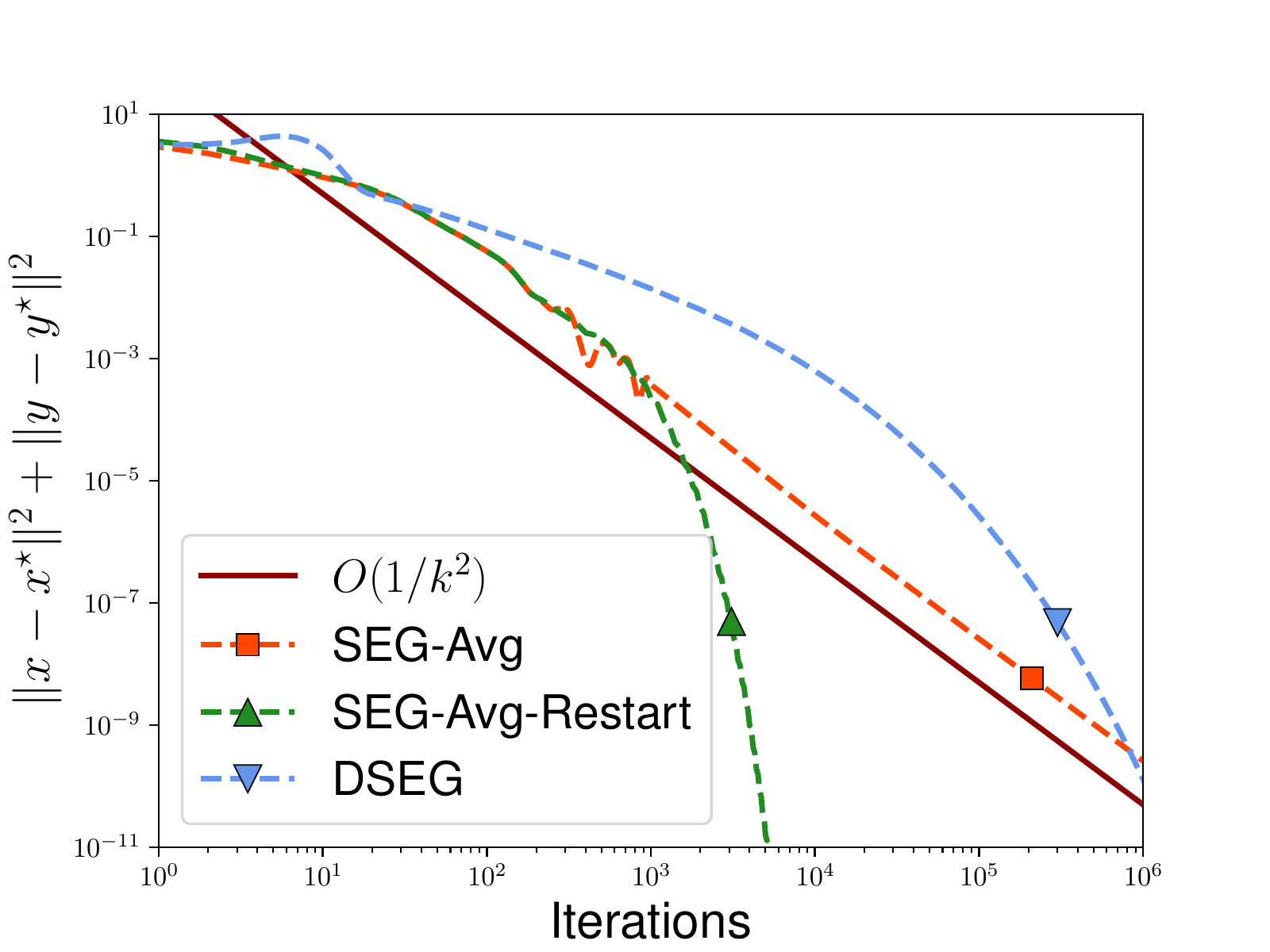}
    }
    \subfigure[\label{fig:interpolation-DSEG-semilogy}Interpolation setting.]{
    \includegraphics[width=0.31\textwidth]{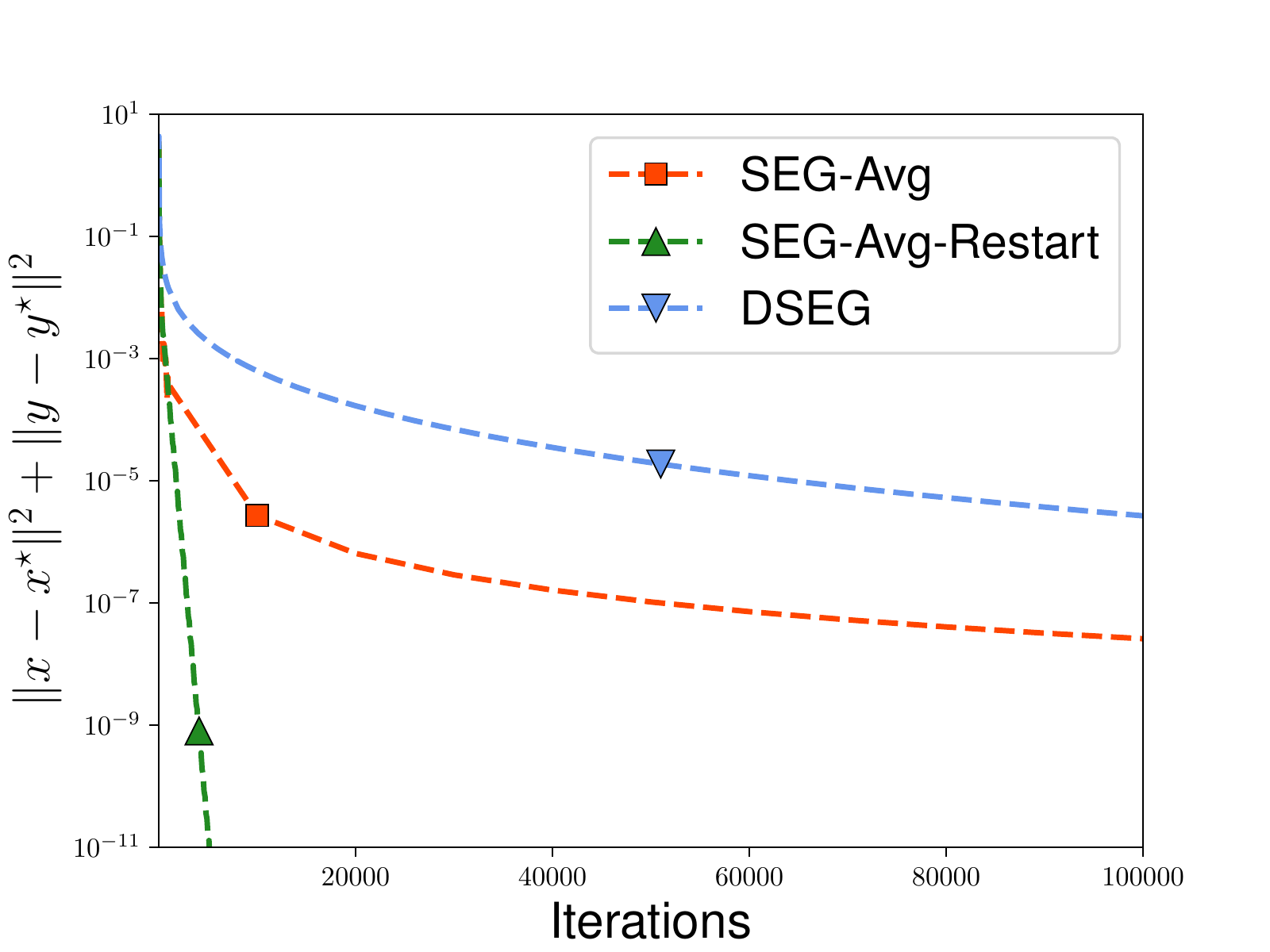}
    }
    \caption{Comparing SEG-Avg, SEG-Avg-Restart, and DSEG methods~\citep{hsieh2020explore} on the stochastic bilinear optimization problem. The horizontal axis represents the iteration number, and the vertical axis represents the square $\ell_2$-distance to the Nash equilibrium. 
    (\textbf{a}) General setting ($d=100, \operatorname{std}_\Bb=0.1, \operatorname{std}_\gb=0.01$). 
    (\textbf{b}) Interpolation setting ($d=100, \operatorname{std}_\Bb=0.1, \operatorname{std}_\gb=0.0$). (\textbf{c}). Interpolation setting ($d=100, \operatorname{std}_\Bb=0.1, \operatorname{std}_\gb=0.0$) under the semi-log scale in the vertical.}
  \end{center}
\vspace{-0.25in}
\end{figure*}

\vspace{0.05in}

\section{EXPERIMENTS}\label{sec_experiment}
In this section, we present the results of numerical experiments on stochastic bilinear minimax optimization problems, including both the general setting and the interpolation setting (i.e., zero noise at the Nash equilibrium). 
The objective function we study remains the same as Eq.~\eqref{Sminimax}, repeated here for convenience:
\beq\tag{\ref{Sminimax}}
\min_\xholder \max_\yholder~
\xholder^\top \Exs_\xi[\Bb_{\xi}] \yholder 
+ 
\xholder^\top \Exs_\xi[\gb^\xholder_\xi] 
+ 
\Exs_\xi[(\gb^\yholder_\xi)^\top] \yholder.
\eeq
Here we assume $\Bb_\xi$ is a square matrix of dimension $d\times d$ where $d = m = n$. 
To generate $\Bb_{\xi}$ for each $\xi$, where $\xi$ corresponds to one iteration in our experiments, we first generate a random vector $\mathbf{u} \in \mathbb{R}^{d}$, where each element of the vector $\mathbf{u}$ is sampled from a uniform distribution, $\mathbf{u}_{j} \sim \text{Unif}\,[1, d+1]$, for $j \in [d]$. Then we define $\Bb = \textsf{Diag}(\mathbf{u})$ and generate $\Bb_{\xi} \in \mathbb{R}^{d\times d}$ as follows:
$$
\Bb_{\xi} = \Bb + \mathbf{E}_{\xi}
    , \quad \text{and}\,\,
[\mathbf{E}_{\xi}]_{ij} \sim \mathcal{N}(0, \operatorname{std}_\Bb^{2})
,
$$
where $\mathbb{E}_{\xi}[\Bb_{\xi}] = \Bb$, and $\Bb$ is a fixed matrix for all $\Bb_{\xi}$. We generate the  noise vectors $\gb_{\xi}^{\xholder} \sim \mathcal{N}(\gb^{\xholder}, \operatorname{std}_\gb^{2}\mathbf{I}_{d\times d})$ and $\gb_{\xi}^{\yholder} \sim \mathcal{N}(\gb^{\yholder}, \operatorname{std}_\gb^{2}\mathbf{I}_{d\times d})$, where we generate the means as follows: $\gb^{\xholder}$, $\gb^{\yholder} \sim \mathcal{N}(0, 0.1 \cdot \mathbf{I}_{d\times d})$ (note that $\gb^{\xholder}$, $\gb^{\yholder}$ are fixed for all $\gb_{\xi}^{\xholder}$, $\gb_{\xi}^{\yholder}$). More specifically, for each iteration, we randomly generate $\{\Bb_{\xi}, \gb_{\xi}^{\xholder}, \gb_{\xi}^{\yholder}\}$ according to the above procedure. 
When $\operatorname{std}_\Bb=\operatorname{std}_\gb=0$, the objective in Eq.~\eqref{Sminimax} equals $ \xholder^{\top}\Bb\yholder +\xholder^{\top}\gb^{\xholder} + (\gb^{\yholder})^{\top}\yholder$, where the Nash equilibrium is $\xholder^{\star} = -(\Bb^{\top})^{-1}\gb^{\yholder}$ and $\yholder^{\star} = -\Bb^{-1}\gb^{\xholder}$.

We study three algorithms in this section: Stochastic ExtraGradient (\textbf{SEG}), Stochastic ExtraGradient with iteration averaging (\textbf{SEG-Avg}), and Stochastic ExtraGradient with Restarted iteration averaging (\textbf{SEG-Avg-Restart}).%
\footnote{Straightforward calculation gives $\sigma_\Bb = \operatorname{std}_\Bb\sqrt{d}$ and $\sigma_\gb = \operatorname{std}_\gb\sqrt{2d}$ in our example, as in Assumptions \ref{assu_boundednoise_A}, \ref{assu_boundednoise_B}.}

\noindent\textbf{General setting ($\sigma_\gb > 0$).}
We first set $\operatorname{std}_\gb = 0.01$ and $\operatorname{std}_\Bb = 0.1$.  The results comparing the three algorithms are shown in Figure~\ref{fig:general}. 
We find that SEG can only converge to a neighborhood of the Nash equilibrium, whereas SEG-Avg and SEG-Avg-Restart can converge to the equilibrium. 
From Figure~\ref{fig:general}, we also observe that the convergence rate of SEG-Avg is ${O}(1/\Tholder^2)$ at the beginning, and then the convergence rate of SEG-Avg changes to ${O}(1/\Tholder)$. 
Similar to the interpolation setting, SEG-Avg-Restart converges faster than both SEG and SEG-Avg. 
We also study the effect of the step size $\eta$ and the noise parameter $\operatorname{std}_\gb$ for SEG. As shown in Figure~\ref{fig:seg-step-size}, we observe that SEG cannot converge to a smaller neighborhood of the Nash equilibrium with smaller step size $\eta$, which aligns well with our theoretical results. We summarize the varying noise experimental results in Figure~\ref{fig:seg-noise}, where we observe that SEG  converges to a smaller neighborhood of the Nash equilibrium when we decrease the noise parameter $\operatorname{std}_\gb$.

\noindent\textbf{Comparisons with DSEG.}
As shown in Figures~\ref{fig:general-DSEG}, ~\ref{fig:interpolation-DSEG}, and \ref{fig:interpolation-DSEG-semilogy}, we provide experimental results on comparing SEG-Avg, SEG-Avg-Restart with the \emph{Double Stepsize Extragradient} (DSEG) method, proposed in \citet{hsieh2020explore}, which allows the step sizes of the extrapolation step and gradient step admitting different scales.
We follow the optimized hyperparameter setup described in \citet{hsieh2020explore} and select the step size constants to achieve faster convergence. 
From Figure~\ref{fig:general-DSEG}, for the general setting, we find that the convergence rate of DSEG is $O(1/K)$ and both SEG-Avg and SEG-Avg-Restart converge faster than DSEG.
For the interpolation setting in Figures~\ref{fig:interpolation-DSEG} and \ref{fig:interpolation-DSEG-semilogy}, we observe that the convergence rate of DSEG is significantly slower than SEG-Avg-Restart.

\noindent\textbf{Interpolation setting ($\sigma_\gb = 0$).}
We first set the noise parameter $\operatorname{std}_\gb = 0$, and set $\operatorname{std}_\Bb = 0.1$. The performance of  SEG, SEG-Avg, and SEG-Avg-Restart is summarized in Figure~\ref{fig:interpolation}, where we set the dimension $d=100$. We observe that the convergence rate of SEG-Avg is ${O}(1/\Tholder^{2})$, which aligns with our theoretical analysis. Meanwhile, we find that SEG-Avg-Restart converges faster than SEG under this interpolation setting. As shown in Figures~\ref{fig:d=100} and \ref{fig:d=200}, we compare the convergence rate of SEG and  SEG-Avg-Restart on a semi-log plot, since both algorithms converge exponentially to the Nash equilibrium in the interpolation setting. We  observe that  SEG-Avg-Restart converges faster than SEG (for both $d=100$ and $d=200$) as suggested by our theory. We also present a zoomed-in plot of SEG-Avg-Restart in Figure~\ref{fig:seg-avg-restart-zoomin}.

\section{CONCLUSIONS}\label{sec_conclusions}
We have presented an analysis of the classical Stochastic ExtraGradient (SEG) method for stochastic bilinear minimax optimization.
Despite that the last iterate only contracts to a fixed neighborhood of the Nash equilibrium and the diameter of the neighborhood is independent of the step size, we show that SEG accompanied by iteration averaging converges to Nash equilibria at a sublinear rate.
Moreover, the forgetting of the initialization is optimal when we use a scheduled restarting procedure in both the general and interpolation settings.
Numerical experiments further validate this use of iteration averaging and restarting in the SEG setting.

Further directions for research include justification of the optimality of our convergence result, improvement of the convergence of SEG for nonlinear convex-concave optimization problems with relaxed assumptions, and connection to the Hamiltonian viewpoint for bilinear minimax optimization.

\section*{Acknowledgements}
We would like to acknowledge support from the Mathematical Data Science program of the Office of Naval Research under grant number N00014-18-1-2764.
Gauthier Gidel and Nicolas Le Roux are supported by Canada CIFAR AI Chairs.
Part of this work was done while Nicolas Loizou was a postdoctoral research fellow at Mila,  Université de Montréal, supported by the  IVADO Postdoctoral Funding Program.
Yi Ma acknowledges support from ONR grant N00014-20-1-2002 and the joint Simons Foundation-NSF DMS grant number 2031899.

\bibliographystyle{plainnat}
\bibliography{SAILreference}

\onecolumn

\newpage
\appendix


\clearpage
\appendix

\thispagestyle{empty}

\onecolumn \makesupplementtitle

\section{ISSUES WITH LAST-ITERATE CONVERGENCE}\label{sec_introissue}
In this section, we revisit the last-iterate convergence of SEG under our setting.
In contrast with minimization problems where stochastic gradient methods with a constant step size converge to a neighborhood of the optimum whose size depends on the step size~\citep{schmidt2014convergence}, solving the stochastic bilinear minimax optimization problem with Stochastic ExtraGradient (SEG) method under standard settings leads to a last iterate contracting to a fixed neighborhood of the Nash equilibrium whose diameter is independent of the step size.
Hence, a classical diminishing step size strategy is not sufficient.

We recall the following notations.
Let $\xi$ be an abstract random variable that is equi-distributed as $\xi_t$.
The expectations, also positive semi-definite, are denoted by $
\widehat{\Mb} = \Exs_\xi\widehat{\Mb}_{\xi} \equiv \Exs_\xi[\Bb_{\xi}^\top \Bb_{\xi}]
$ and $
\Mb = \Exs_\xi\Mb_{\xi} \equiv \Exs_\xi[\Bb_{\xi} \Bb_{\xi}^\top]
$.
It is easy to verify that both matrices are symmetric and positive semi-definite.
Recall that $\eta_{\Mb}$ is the \emph{maximal step size} the SEG algorithm analysis takes, defined earlier as
\beq\tag{\ref{etaMdef}}
\eta_{\Mb}
	\equiv
\frac{1}{\sqrt{
\lambda_{\max}\big(\Mb^{-1/2} [\Exs_\xi \Mb_\xi^2] \Mb^{-1/2}\big)
\lor
\lambda_{\max}\big(\widehat{\Mb}^{-1/2} [\Exs_\xi \widehat{\Mb}_\xi^2] \widehat{\Mb}^{-1/2}\big)
}}.
\eeq
When $\Bb_\xi$ is nonrandom the value $\eta_{\Mb}$ simply reduces to $1/\sqrt{\lambda_{\max}(\Bb^\top\Bb)}$.
It is worth highlighting that the spectral knowledge of matrices involving moments of $\Bb_{\xi}$ that we assume is mild, as analogous spectral information has been traditionally assumed in the online stochastic optimization literature.

We remind the readers of the following result on last-iterate SEG~(extension of \citet{hsieh2020explore}):

\begin{theorem}[SEG Last Iterate]\label{theo_SEG_A}
Under proper assumptions [e.g.,~Assumptions~\ref{assu_boundednoise_A} and \ref{assu_boundednoise_B} in \S\ref{sec_assumptions}],
if $\eta$ is chosen as $\eta_{\Mb} / \sqrt{2}$ where $\eta_{\Mb}$ is defined as in Eq.~\eqref{etaMdef}, we have the following upper bound for the last iterate, $(\xholder_{\Tholder}, \yholder_{\Tholder})$ generated by the algorithm in Eq.~\eqref{SEGupdate}, for all $\Tholder\ge 1$:
\beq\label{xygrowth_SEG_A}
\Exs \left[\|\xholder_\Tholder\|^2 + \|\yholder_\Tholder\|^2\right]
\leq
e^{
-\frac{\eta_{\Mb}^2}{4}\cdot (\lambda_{\min}(\Mb)\land\lambda_{\min}(\widehat{\Mb})) \cdot\Tholder
}\left[\|\xholder_0\|^2 + \|\yholder_0\|^2 \right]
+
\frac{3\sigma_{\gb}^2}{\lambda_{\min}(\Mb)\land\lambda_{\min}(\widehat{\Mb})}
.
\eeq
\end{theorem}
With our chosen step size, as $\Tholder\to\infty$ the expected squared Euclidean norm converges linearly in Eq.~\eqref{xygrowth_SEG_A}, i.e.,
$$
\limsup_{\Tholder\to\infty}\, \Exs \left[\|\xholder_\Tholder\|^2 + \|\yholder_\Tholder\|^2\right]
	\le
\frac{3\sigma_{\gb}^2}{\lambda_{\min}(\Mb) \land\lambda_{\min}(\widehat{\Mb})},
$$
which is, in the $\sigma_{\gb}>0$ case, bounded away from zero.%
\footnote{In contrast to $\lambda_{\min}(\Bb\Bb^\top)$ being zero when $\Bb\in \RR^{n\times m}$ with $n>m$, the $\lambda_{\min}(\Mb) \land\lambda_{\min}(\widehat{\Mb})$ can be positive when $\Bb_\xi$ is random in general.
A standard instance will be $\Bb_\xi$ being $n\times m$ a Gaussian random matrix consisting of independent standard normals.}
A version of Theorem \ref{theo_SEG_A} was provided by \cite{hsieh2020explore}, where a two-timescale method was proposed to remedy this lack of convergence to zero, with a large step size update of gradient step followed by a smaller step size update of the extragradient step. In this case the asymptotic neighborhood size is proportional to the square root of their ratio.
However, \cite{hsieh2020explore} only provide a proof under an assumption of bounded noise.
In the interpolation case where $\sigma_{\gb} = 0$, \cite{vaswani2019painless} showed a weaker version of Theorem \ref{theo_SEG_A} that incorporates an exact line-search step.
To the best of our knowledge, the statement of Theorem \ref{theo_SEG_A} is the first to identify the maximal step size $\eta_{\Mb}$ that can be taken by the SEG method in Eq.~\eqref{etaMdef}.
For completeness, we provide the proof of our version of Theorem \ref{theo_SEG_A} in \S\ref{sec_proof,theo_SEG_A}.

In addition, we introduce the following negative result that establishes a lower bound that accommodates a broader range of step sizes.
This result shows that the upper bound on SEG convergence rate in Theorem \ref{theo_SEG_A} is not improvable even for the case of diminishing step sizes, which limits the applicability of the last-iterate output of the SEG algorithm in general \citep{hsieh2020explore}.

\begin{theorem}[Lower bound for SEG, extension of \cite{hsieh2020explore}]\label{theo_SEG_lower_bound}
Under the assumptions of Theorem~\ref{theo_SEG_A}, there exist $n,m\ge 1$, a distribution $\mathbb{P}$ supported on $\mathbb{R}^{n\times m} \times \mathbb{R}^{n} \times \mathbb{R}^{m}$ for $\{(\Bb_{\xi}, \gb^{\xholder}_{\xi}, \gb^{\yholder}_{\xi})\}$, and an initialization $(\xholder_0, \yholder_0)$ satisfying $\|\xholder_0\|^2 + \|\yholder_0\|^2\ge C_1 \sigma_\gb^{2}$ such that, for any sequence of step sizes $\eta_t\in [0, \eta_{\mathbf{M}}]$, the last-iterate SEG $(\xholder_{\Tholder}, \yholder_{\Tholder})$ generated by Eq.~\eqref{SEGupdate} satisfies $
\Exs \left[\|\xholder_\Tholder\|^2 + \|\yholder_\Tholder\|^2\right] \ge C_2\sigma_\gb^2
$ for any $\Tholder\ge 1$, where $C_1, C_2$ are positive, numerical constants.
\end{theorem}
In this work, we remedy the lack of convergence that this results indicate via a convergence analysis of the \emph{averaged} iterates.
We show in the main text of this paper that SEG with properly scheduled restarting and iteration averaging achieves a statistically optimal rate of convergence, as well as an exponentially mixing (forgetting) of the initialization; see \S\ref{sec_SEGg}.

\section{COMPARISON OF THEOREM~\ref{theo_SEG_B} WITH EXISTING WORK} \label{sec_comp,theo_SEG_B}

\begin{table}[t!]
    \centering
    \begin{tabular}{lc}
    \toprule
    $\sigma_{\Bb} > 0$ and $\sigma_{\gb} = 0$ & Convergence Rate
    \\
    \midrule
    \citet{juditsky2011solving} &  $O\left(\max\left\{ \frac{R^2}{\Tholder^2},\, \frac{1}{\Tholder}\right\}\right)^{\dagger}$
    \\
    This work & $O\left(\frac{\|\xholder_0\|^2 + \|\yholder_0\|^2}{\Tholder^2}\right)$
    \\
    \midrule
    $\sigma_{\Bb}=0$ and $\sigma_{\gb} > 0$  & Convergence Rate
    \\
    \midrule
    \citet{hsieh2020explore} &  $O\left(\frac{\sigma_{\gb}^{2}}{\Tholder}\right)+ o\left(\frac{1}{\Tholder}\right)$
    \\
    This work & $O\left(\frac{\sigma_{\gb}^2}{\Tholder}\right) + O\left(\frac{\|\xholder_0\|^2 + \|\yholder_0\|^2}{\Tholder^2}\right)$
    \\
    \bottomrule
    \end{tabular}
    \vspace{-0.05in}
    \caption{Comparing convergence rates with \citet{juditsky2008large}  and \citet{hsieh2020explore}.  $^\dagger R$ is the squared domain radius. }
    \label{tab:convergence_rate_table}
\vspace{-0.1in}
\end{table}

In this section, we compare our results with existing work.
We first provide a few remarks regarding the convergence rate in Theorem \ref{theo_SEG_B}:
\begin{enumerate}[label=(\roman*)]
\item
In the general stochastic setting ($\sigma_{\gb} > 0$), the step size of our algorithm is not sensitive to the number of iteration ($\Tholder$), i.e., simply picking the constant step size would guarantee the sharp convergence of (same-sample) SEG to the optimal solution, which benefits from the intrinsic linearity of our problem. In comparison, the algorithms in \citep{juditsky2011solving, mishchenko2020revisiting} rigidly select the step size $\eta = O(1/\sqrt{\Tholder})$. Meanwhile, our algorithm does not require the  projection step compared with \citet{juditsky2011solving, mishchenko2020revisiting}.

\item
Our analysis of Theorem \ref{theo_SEG_B} indicates that the ``forgetting rate'' of the dependency on initialization $\|\xholder_0\|^2 + \|\yholder_0\|^2$ can be improved to $O(1/\Tholder^2)$, achieving an optimal overall rate that is faster than existing work.
Mathematically we concluded \eqref{Asingleloop_SEG_B} which is recapped here:
\beq\tag{\ref{Asingleloop_SEG_B}}\hspace{-.2in}\begin{aligned}
\Exs\left[
\left\|\overline{\xholder}_{\Tholder}\right\|^2 + \left\|\overline{\yholder}_{\Tholder}\right\|^2
\right]
	&\le
\frac{16 + 8\kappa_\zeta}{(1-\alpha)\hat\eta_{\Mb}(\alpha)^2\lambda_{\min}(\Bb\Bb^\top)}
\cdot
\frac{\|\xholder_0\|^2 + \|\yholder_0\|^2}{(\Tholder+1)^2}
+
\frac{18 + 12\kappa_\zeta}{(1-\alpha)\lambda_{\min}(\Bb\Bb^\top)}
\cdot
\frac{\sigma_{\gb}^2}{\Tholder+1}
,
\end{aligned}\eeq
where we recall that $
\kappa_\zeta
    \equiv
\frac{\sigma_{\Bb}^2 + \hat\eta_{\Mb}(\alpha)^2\sigma_{\Bb,2}^2}{\lambda_{\min}(\Mb)\land\lambda_{\min}(\widehat{\Mb})}
$ denotes the effective noise condition number of problem Eq.~\eqref{Sminimax}.
Nevertheless in our upcoming restarting analysis, we present an alternative convergence rate bound for the averaged iterate as follows: for arbitrary $\gamma\in (0,\infty)$
\beq\label{Asingleloop_SEG_B-appendix}\hspace{-.2in}\begin{aligned}
\Exs\left[
\|\overline{\xholder}_{\Tholder}\|^2 + \|\overline{\yholder}_{\Tholder}\|^2
\right]
&\le	
\frac{8(1+\gamma)}{(1-\alpha)\hat\eta_{\Mb}(\alpha)^2\lambda_{\min}(\Bb\Bb^\top)}
\cdot
\frac{\|\xholder_0\|^2 + \|\yholder_0\|^2}{\color{red}(\Tholder+1)^2}
\\&\hspace{-1in}\,
\frac{
2\left(1+\frac{1}{\gamma}\right)(\sigma_{\Bb}^2 + \hat\eta_{\Mb}(\alpha)^2\sigma_{\Bb,2}^2)\left[
\|\xholder_0\|^2 + \|\yholder_0\|^2 + \frac{3\sigma_{\gb}^2}{\lambda_{\min}(\Mb)\land\lambda_{\min}(\widehat{\Mb})}
\right]
+
9(1+\gamma)\sigma_{\gb}^2
}{(1-\alpha)\lambda_{\min}(\Bb\Bb^\top)}
\cdot
\frac{1}{\color{red}\Tholder+1}
,
\end{aligned}\eeq
which is slightly better (in the case of $\gamma=1$) for our restarting analysis.
See the discussion paragraph on pp.~\pageref{rema_K2inv_upper} for more on this.
\end{enumerate}

\paragraph{Comparison with \citet{hsieh2020explore}}
\cite{hsieh2020explore} considered the independent-sample double-stepsize SEG, and our work focuses on the same-sample extra-gradient methods. The convergence rate of DSEG~\cite{hsieh2020explore} in the general stochastic bilinear minimax optimization problem ($\sigma_{\Bb}=\sigma_{\Bb, 2}=0$ and $\sigma_{\gb} > 0$) is
$$
\Exs\left[
\|{\xholder}_{\Tholder}\|^2 + \|{\yholder}_{\Tholder}\|^2
\right]
\le
\frac{\lambda_{\max}(\Bb^\top\Bb)}{\lambda_{\min}^{2}(\Bb\Bb^\top)}\cdot\frac{\sigma_{\gb}^{2}}{\Tholder}
+
o\left(\frac{1}{\Tholder}\right)
.
$$
In contrast, our convergence rate is, in a coarse manner, (by setting $\sigma_{\Bb}=\sigma_{\Bb, 2}=0$ in Eq.~\eqref{K2inv_upper})
$$
\Exs\left[
\|\overline{\xholder}_{\Tholder}\|^2 + \|\overline{\yholder}_{\Tholder}\|^2
\right]
    \lesssim
\frac{1}{\lambda_{\min}(\Bb\Bb^\top)}\cdot\frac{\sigma_{\gb}^2}{\Tholder} 
+
\frac{\lambda_{\max}(\Bb^\top\Bb)}{\lambda_{\min}(\Bb\Bb^\top)}\cdot\frac{\|\xholder_0\|^2 + \|\yholder_0\|^2}{\Tholder^2}
.
$$
We observe in the above two displays that our rate is sharper than the rate of \citet{hsieh2020explore} in terms of the coefficient of $\sigma_{\gb}^{2}/\Tholder$, which is the dominant term in both bounds.
In particular, when the step size is chosen properly our convergence rate bound is sharper in the interpolation setting where $\sigma_{\Bb} > 0$ and $\sigma_{\gb} = 0$.

\paragraph{Comparison with \citet{juditsky2011solving}}
We first provide the connection between restricted gap  and distance to the Nash equilibrium. Suppose we consider the bounded domain setting for the bilinear minimax optimization problem where $Z = \{\|\xholder\|\le R, \|\yholder\|\le R\}$ and $R$ is the domain radius, and the variational inequality with monotone mapping $F(\zholder) = \begin{bmatrix}
\Bb\yholder \\
-\Bb^{\top}\xholder
\end{bmatrix}$ where $\zholder = \begin{bmatrix}
\xholder \\
\yholder
\end{bmatrix}$.
Then the restricted gap (i.e., merit function) can be expressed as
\begin{equation*}
\text{Err}_{\text{vi}}(\zholder_{K})
=
\max_{\zholder\in Z} \langle F(\zholder), \zholder_{K} - \zholder  \rangle
=
\max_{\|\yholder\|\le R} \xholder_{K}^{\top} \Bb\yholder - \min_{\|\xholder\|\le R} \yholder_{K}^{\top} \Bb^{\top}\xholder = R\left(\|\Bb^{\top}\xholder_{K}\| +\|\Bb\yholder_{K}\|  \right)
.
\end{equation*}
Therefore, the restricted gap can be lower bounded as
\begin{equation*}
    \text{Err}_{\text{vi}}(\zholder_{K}) \geq R\sqrt{\lambda_{\min}(\Bb\Bb^{\top})}\left(\|\xholder_{K}\| +\|\yholder_{K}\|  \right).
\end{equation*}
With this relation at hand, the convergence rate in \citet{juditsky2011solving} when calibrated to the interpolation setting ($\sigma_{\Bb} > 0$ and $\sigma_{\gb} = 0$) is 
$$
\Exs\left[
\|\overline{\xholder}_{\Tholder}\| + \|\overline{\yholder}_{\Tholder}\| 
\right] ^2
    \lesssim
\frac{1}{\lambda_{\min}(\Bb\Bb^{\top})}\cdot \frac{\sigma_{\Bb}^2}{\Tholder}
+
\frac{\lambda_{\max}(\Bb^{\top}\Bb)}{\lambda_{\min}(\Bb\Bb^{\top})}\cdot \frac{R^2}{\Tholder^2}
.
$$
In comparison with \citet{juditsky2011solving} our convergence rate in Eq.~\eqref{Asingleloop_SEG_B} spells
$$
\Exs\left[
\|\overline{\xholder}_{\Tholder}\|^{2} + \|\overline{\yholder}_{\Tholder}\|^{2} 
\right] 
\lesssim
\left(
\frac{1}{\hat\eta_{\Mb}(\alpha)^2} 
+
\frac{\sigma^{2}_{\Bb} + \hat\eta_{\Mb}(\alpha)^2\sigma_{\Bb,2}^2
}{\hat\eta_{\Mb}(\alpha)^2(\lambda_{\min}(\Mb)\land\lambda_{\min}(\widehat{\Mb}))}
\right)
\cdot
\frac{1}{\lambda_{\min}(\Bb\Bb^{\top})}
\cdot
\frac{\|\xholder_0\|^2 + \|\yholder_0\|^2}{\Tholder^2}
,
$$
where our convergence rate is significantly better in terms of the $\sigma_{\Bb}$-dependency.%
\footnote{
In the above five displays, $a_n\lesssim b_n$ denotes $a_n = O(b_n)$ for the two positive sequences.
}

\paragraph{Other related work on stochastic min-max problems}
\citet{alacaoglu2021stochastic} proposed stochastic variance reduced algorithms for solving variational inequalities with the finite-sum structure.  
For more recent results on stochastic iterative methods for solving min-max problems we refer the interested reader to \citet{loizouNeurips2021stochastic, gorbunov2021stochastic, gorbunov2021extragradient} and the references therein.

\section{TECHNICAL ANALYSIS OF LAST-ITERATE SEG}\label{sec_introissue_technical}
In this section we present the technical details of our theoretical results in \S\ref{sec_introissue}, focusing on the last-iterate Theorem \ref{theo_SEG_A}.%
\footnote{The proof of Theorem \ref{theo_SEG_lower_bound} can be found in \cite{hsieh2020explore} and hence we omit it in this work.}
We first introduce a lemma without proof, which is a standard result in linear algebra \citep[Lecture 5]{TREFETHEN-BAU} stating the relations between spectrum of relevant matrices:
\begin{lemma}[Spectral properties]\label{lemm_spectrum}
For our coupling matrix $\Bb\in \RR^{n\times m}$ with $n\ge m$ (tall matrix), $\Bb^\top \Bb$ and $\Bb\Bb^\top$ share the same spectrum or eigenvalues except zeroes:
$$
\sigma(\Bb\Bb^\top)
=
\sigma(\Bb^\top\Bb)		\cup		(\underbrace{0,\dots,0}_{n-m})
.
$$
Furthermore, both $\sigma(\Bb^\top\Bb)$ and $\sigma(\Bb\Bb^\top)$ are subsets of the nonnegative reals, so we always have
$$
\lambda_{\max}(\Bb \Bb^\top) = \lambda_{\max}(\Bb^\top \Bb)
,
$$
and
\beq\label{largeMsmallEV}
\lambda_{\min}\left(\Bb^\top\Bb\right)
\ge
\lambda_{\min}\left(\Bb\Bb^\top\right)
.
\eeq
In special, when $n=m$ we have $\lambda_{\min}(\Bb^\top\Bb) = \lambda_{\min}(\Bb\Bb^\top)$.
The $\lambda_{\min}(\Bb \Bb^\top)$ might be different from $\lambda_{\min}(\Bb^\top \Bb)$ when $\Bb$ is nonsquare, in which case $\lambda_{\min}(\Bb\Bb^\top)$ simply reduces to 0 whenever $n>m$.
\end{lemma}
Next, we introduce the \emph{contraction parameter} that plays a key role in our analysis.
\beq\label{lambdaeta}
\lambdaetaprime
\equiv
\lambda_{\min}\left(\Mb - \eta^2 [\Exs_\xi\Mb_{\xi}^2] \right) 
\land
\lambda_{\min}\left(\widehat{\Mb} - \eta^2 [\Exs_\xi\widehat{\Mb}_{\xi}^2] \right)
.
\eeq
Note that $\lambdaetaprime$ is \emph{not} necessarily nonnegative for positive $\eta$s.
We have the following lemma establishing various inequalities regarding $\lambdaetaprime$ and $\eta_{\Mb}$ as in Eq.~\eqref{etaMdef}:

\begin{lemma}\label{lemm_etaMbound}
Under Assumption \ref{assu_boundednoise_A} we have
\begin{enumerate}[label=(\arabic*)]
\item
For all $\eta > 0$, it holds that
\beq\label{lambdaetaprime_upper}
\eta^2\lambdaetaprime	\le	1/4
.
\eeq

\item
For all $\eta\in \left(0,\eta_{\Mb}\right]$ where $\eta_{\Mb}$ is defined as in Eq.~\eqref{etaMdef}, it holds that
\beq\label{lambdaetaprime_lower}
\lambdaetaprime
	\ge
\left(1-\frac{\eta^2}{\eta_{\Mb}^2}\right)
\left(\lambda_{\min}(\Mb)
\land
\lambda_{\min}(\widehat{\Mb})\right)
	\ge
0
.
\eeq

\item
$\eta_{\Mb}$ defined as in Eq.~\eqref{etaMdef} satisfies, for any $\eta > 0$ such that $\lambdaetaprime \ge 0$,
\beq\label{etaMbound}
0	<	\eta_{\Mb}		\le
\frac{1}{\sqrt{
\lambda_{\max}(\Mb)	\lor	\lambda_{\max}(\widehat{\Mb})
}
}
	\le
\frac{1}{\sqrt{\lambda_{\max}(\Bb^\top \Bb)}}
.
\eeq
When $\Bb_\xi = \Bb$ a.s., both equalities hold in the above Eq.~\eqref{etaMbound}.
\end{enumerate}
\end{lemma}
The proof of Lemma \ref{lemm_etaMbound} is detailed in \S\ref{sec_proof,lemm_etaMbound}.

\subsection{Analysis of Theorem~\ref{theo_SEG_A}}\label{sec_proof,theo_SEG_A}
\begin{custom}{Theorem \ref{theo_SEG_A}, Full Version}
Let Assumptions~\ref{assu_boundednoise_A} and \ref{assu_boundednoise_B} hold.
For any positive $\eta$ we have for all $\Tholder\ge 1$
\beq\label{xygrowth_SEGg}\begin{aligned}
	&\quad\,
\Exs \left[\|\xholder_\Tholder\|^2 + \|\yholder_\Tholder\|^2\right]
	\\&\le
\left(1 - \eta^2 \lambdaetaprime\right)^\Tholder \left[\|\xholder_0\|^2 + \|\yholder_0\|^2 \right]
+
\eta^2\Quan_{\Tholder}(\eta)
\left[1 + \eta^2\left(\lambda_{\max}(\Mb) \lor \lambda_{\max}(\widehat{\Mb})\right)\right] \sigma_{\gb}^2
,
\end{aligned}\eeq
where we denote
\beq\label{Quandef}
\Quan_{\Tholder}(\eta)
\equiv
\sum_{t=1}^{\Tholder} \left(1 - \eta^2 \lambdaetaprime\right)^{t-1}
\quad\text{which is upper bounded by $
\Tholder \land \frac{1}{\eta^2 \lambdaetaprime}
$
}
,
\eeq
and $\lambdaetaprime$ was earlier defined as in Eq.~\eqref{lambdaeta}.
For all $\eta\in \left(0,\eta_{\Mb}\right]$ where $\eta_{\Mb}$ is defined as in Eq.~\eqref{etaMdef}, we have for all $\Tholder\ge 1$
\beq\label{xygrowth_SEG_Aprime}\begin{aligned}
	&\quad\,
\Exs\left[\|\xholder_\Tholder\|^2 + \|\yholder_\Tholder\|^2\right]
	\\&\le
\left(1 - \eta^2\left(1-\frac{\eta^2}{\eta_{\Mb}^2}\right) \left(
\lambda_{\min}(\Mb)
\land
\lambda_{\min}(\widehat{\Mb})
\right)\right)^\Tholder \left[\|\xholder_0\|^2 + \|\yholder_0\|^2 \right]
	\\&\hspace{1in}
+
\eta^2\Quan_{\Tholder}(\eta)
\left[1 + \eta^2\left(\lambda_{\max}(\Mb) \lor \lambda_{\max}(\widehat{\Mb})\right)\right] \sigma_{\gb}^2
	\\&\le
\exp\left(
-\frac{\eta_{\Mb}^2}{4} \left(\lambda_{\min}(\Mb)\land\lambda_{\min}(\widehat{\Mb})\right)\cdot\Tholder
\right)\left[\|\xholder_0\|^2 + \|\yholder_0\|^2 \right]
+
\frac{3\sigma_{\gb}^2}{\lambda_{\min}(\Mb)\land\lambda_{\min}(\widehat{\Mb})}
	\\&\hspace{1in}
\text{(when $\eta = \eta_{\Mb}/\sqrt{2}$)}
.
\end{aligned}\eeq
\end{custom}
Analogous to our remarks immediately following the statement of Theorem \ref{theo_SEG_A} in \S\ref{sec_introissue}, for a given range of step size such that $\lambdaetaprime$ is positive, $\Quan_{\Tholder}(\eta) \to 1/(\eta^2\lambdaetaprime)$ as $\Tholder\to\infty$ the squared Euclidean norm approaches
$$
\Exs \left[\|\xholder_\Tholder\|^2 + \|\yholder_\Tholder\|^2\right]
	\to
\frac{1}{\lambdaetaprime}
\left[1 + \eta^2\left(
\lambda_{\max}(\Mb) \lor \lambda_{\max}(\widehat{\Mb})
\right)\right] \sigma_{\gb}^2
,
$$
which is bounded below, due to Eq.~\eqref{lambdaetaprime_upper}, by
$
\frac{\sigma_{\gb}^2}{\lambda_{\min}(\Mb)\land\lambda_{\min}(\widehat{\Mb})}
$
due to $
\lambdaetaprime	\le
\lambda_{\min}(\Mb)\land\lambda_{\min}(\widehat{\Mb})
$
and hence bounded away from 0.
Optimizing the choice of $\eta$ achieves, as observed in Eq.~\eqref{xygrowth_SEG_A}, a limiting upper bound that triples the above display so the bandwidth of the limiting points is rather narrow (within a triple bandwidth).

We now turn to prove Theorem~\ref{theo_SEG_A}.

\begin{proof}[Proof of Theorem~\ref{theo_SEG_A}]
We denote for short $
\Mb_{\xi} \equiv \Bb_{\xi}\Bb_{\xi}^\top
$ and $
\widehat{\Mb}_{\xi} \equiv \Bb_{\xi}^\top \Bb_{\xi}
$, and $
\xholder \equiv \xholder_t
$, $
\yholder \equiv \yholder_t
$, $
\xholder^- \equiv \xholder_{t-1}
$, $
\yholder^- \equiv \yholder_{t-1}
$, $
\Bb_\xi \equiv \Bb_{\xi,t}
$, $
\gb_\xi \equiv \gb_{\xi,t}
$, as well as the conditional expectation $\Exs_\xi [\cdot] = \Exs\left[\cdot \mid \mathcal{F}_{t-1} \right]$.
Recall~Eq.~\eqref{SEGupdate} combined gives the SEG update rules is in total
\beq\label{SEGupdate_combined}\begin{aligned}
\xholder
	&=
\xholder^- - \eta^2 \Bb_{\xi} \Bb_{\xi}^\top \xholder^- - \eta\left[ \Bb_{\xi} \yholder^- + \gb^\xholder_{\xi}\right] - \eta^2 \Bb_{\xi} \gb^\yholder_{\xi}
	\\
\yholder
	&=
\yholder^- - \eta^2 \Bb_{\xi}^\top \Bb_{\xi} \yholder^- + \eta\left[ \Bb_{\xi}^\top \xholder^- + \gb^\yholder_{\xi}\right] - \eta^2 \Bb_{\xi}^\top \gb^\xholder_{\xi}
.
\end{aligned}\eeq
By analyzing equation~Eq.~\eqref{SEGupdate_combined} we derive
$$\begin{aligned}
\Exs_\xi \left[\|\xholder\|^2 + \|\yholder\|^2 \right]
	&=
\Exs_\xi \left\|
\left(\Ib - \eta^2 \Bb_\xi \Bb_\xi^\top\right)\xholder^- - \eta \Bb_\xi \yholder^- 
- \eta \gb^\xholder_\xi - \eta^2\Bb_\xi \gb^\yholder_\xi
\right\|^2
	\\&\quad\,
+
\Exs_\xi \left\|
\left(\Ib - \eta^2 \Bb_\xi^\top \Bb_\xi\right)\yholder^- + \eta \Bb_\xi^\top \xholder^- 
+ \eta \gb^\yholder_\xi -\eta^2 \Bb_\xi^\top \gb^\xholder_\xi
\right\|^2
	\\&=
\Exs_\xi \left\|
\left(\Ib - \eta^2 \Bb_\xi \Bb_\xi^\top\right)\xholder^- - \eta \Bb_\xi \yholder^- 
\right\|^2
+
\Exs_\xi \left\|
- \eta \gb^\xholder_\xi - \eta^2\Bb_\xi \gb^\yholder_\xi
\right\|^2
	\\&\quad\,
\underbrace{+
2\Exs_\xi \left\langle
\left(\Ib - \eta^2 \Bb_\xi \Bb_\xi^\top\right)\xholder^- - \eta \Bb_\xi \yholder^- 
,
- \eta \gb^\xholder_\xi - \eta^2\Bb_\xi \gb^\yholder_\xi
\right\rangle
}_{\text{cross term}}
	\\&\quad\,
+
\Exs_\xi \left\|
\left(\Ib - \eta^2 \Bb_\xi^\top \Bb_\xi\right)\yholder^- + \eta \Bb_\xi^\top \xholder^- 
\right\|^2
+
\Exs_\xi \left\|
\eta \gb^\yholder_\xi -\eta^2 \Bb_\xi^\top \gb^\xholder_\xi
\right\|^2
	\\&\quad\,
\underbrace{+
2\Exs_\xi \left\langle
\left(\Ib - \eta^2 \Bb_\xi^\top \Bb_\xi\right)\yholder^- + \eta \Bb_\xi^\top \xholder^- 
,
\eta \gb^\yholder_\xi -\eta^2 \Bb_\xi^\top \gb^\xholder_\xi
\right\rangle
}_{\text{cross term}}
,
\end{aligned}$$
where by independence we have the cross terms being
$$\begin{aligned}
	&
2\Exs_\xi \left\langle
\left(\Ib - \eta^2 \Bb_\xi \Bb_\xi^\top\right)\xholder^- - \eta \Bb_\xi \yholder^- 
,
- \eta \gb^\xholder_\xi - \eta^2\Bb_\xi \gb^\yholder_\xi
\right\rangle
=
0
	\\&
2\Exs_\xi \left\langle
\left(\Ib - \eta^2 \Bb_\xi^\top \Bb_\xi\right)\yholder^- + \eta \Bb_\xi^\top \xholder^- 
,
\eta \gb^\yholder_\xi -\eta^2 \Bb_\xi^\top \gb^\xholder_\xi
\right\rangle
=
0
.
\end{aligned}$$
Therefore
$$\begin{aligned}
	&\quad\,
\Exs_\xi \left[\|\xholder\|^2 + \|\yholder\|^2 \right]
	\\&=
\Exs_\xi \left\|
\left(\Ib - \eta^2 \Bb_\xi \Bb_\xi^\top\right)\xholder^- - \eta \Bb_\xi \yholder^- 
\right\|^2
+
\Exs_\xi \left\|
- \eta \gb^\xholder_\xi - \eta^2\Bb_\xi \gb^\yholder_\xi
\right\|^2
	\\&\quad\,
+
\Exs_\xi \left\|
\left(\Ib - \eta^2 \Bb_\xi^\top \Bb_\xi\right)\yholder^- + \eta \Bb_\xi^\top \xholder^- 
\right\|^2
+
\Exs_\xi \left\|
\eta \gb^\yholder_\xi -\eta^2 \Bb_\xi^\top \gb^\xholder_\xi
\right\|^2
	\\&=
\Exs_\xi\left\|
\left(\Ib - \eta^2 \Bb_\xi \Bb_\xi^\top\right)\xholder^-
\right\|^2
+
\Exs_\xi\left\|
-\eta \Bb_\xi \yholder^-
\right\|^2
+ 
\Exs_\xi\left\|
\left(\Ib - \eta^2 \Bb_\xi^\top \Bb_\xi\right)\yholder^-
\right\|^2
+
\Exs_\xi\left\|
\eta \Bb_\xi^\top \xholder^-
\right\|^2
	\\&\quad\,
\underbrace{
+
2\Exs_\xi\left\langle
\left(\Ib - \eta^2 \Bb_\xi \Bb_\xi^\top\right)\xholder^-
,
-\eta \Bb_\xi \yholder^-
\right\rangle
+
2\Exs_\xi\left\langle
\left(\Ib - \eta^2 \Bb_\xi^\top \Bb_\xi\right)\yholder^-
,
\eta \Bb_\xi^\top \xholder^-
\right\rangle
}_{\text{cross term}}
	\\&\quad\,
+
\Exs_\xi \left\|- \eta \gb^\xholder_\xi\right\|^2
+
\Exs_\xi \left\|- \eta^2\Bb_\xi \gb^\yholder_\xi\right\|^2
+
\Exs_\xi \left\|\eta \gb^\yholder_\xi\right\|^2
+
\Exs_\xi \left\|-\eta^2 \Bb_\xi^\top \gb^\xholder_\xi\right\|^2
	\\&\quad\,
\underbrace{+
2\Exs_\xi \left\langle
- \eta \gb^\xholder_\xi 
,
- \eta^2\Bb_\xi \gb^\yholder_\xi
\right\rangle
+
2\Exs_\xi \left\langle
\eta \gb^\yholder_\xi
,
-\eta^2 \Bb_\xi^\top \gb^\xholder_\xi
\right\rangle
}_{\text{cross term}}
,
\end{aligned}$$
where it is again easy to verify the cross terms are zero due to the identities
$$
\Exs_\xi\left\langle
\left(\Ib - \eta^2 \Bb_\xi \Bb_\xi^\top\right)\xholder^-
,
-\eta \Bb_\xi \yholder^-
\right\rangle
	+
\Exs_\xi\left\langle
\left(\Ib - \eta^2 \Bb_\xi^\top \Bb_\xi\right)\yholder^-
,
\eta \Bb_\xi^\top \xholder^-
\right\rangle
	=
0
,
$$
$$
\Exs_\xi \left\langle
- \eta \gb^\xholder_\xi 
,
- \eta^2\Bb_\xi \gb^\yholder_\xi
\right\rangle
	+
\Exs_\xi \left\langle
\eta \gb^\yholder_\xi
,
-\eta^2 \Bb_\xi^\top \gb^\xholder_\xi
\right\rangle
	=
0
.
$$
Finally
\beq\label{SEG_tol_g}\begin{aligned}
	&\quad\,
\Exs_\xi \left[\|\xholder\|^2 + \|\yholder\|^2 \right]
	\\&=
\Exs_\xi\left\|
\left(\Ib - \eta^2 \Bb_\xi \Bb_\xi^\top\right)\xholder^-
\right\|^2
+
\Exs_\xi\left\|
-\eta \Bb_\xi \yholder^-
\right\|^2
+ 
\Exs_\xi\left\|
\left(\Ib - \eta^2 \Bb_\xi^\top \Bb_\xi\right)\yholder^-
\right\|^2
+
\Exs_\xi\left\|
\eta \Bb_\xi^\top \xholder^-
\right\|^2
	\\&\quad\,
+
\Exs_\xi \left\|- \eta \gb^\xholder_\xi\right\|^2
+
\Exs_\xi \left\|- \eta^2\Bb_\xi \gb^\yholder_\xi\right\|^2
+
\Exs_\xi \left\|\eta \gb^\yholder_\xi\right\|^2
+
\Exs_\xi \left\|-\eta^2 \Bb_\xi^\top \gb^\xholder_\xi\right\|^2
	\\&=
(\xholder^-)^\top 
\Exs \left(\Ib - \eta^2 \Bb_\xi \Bb_\xi^\top + \left(\eta^2 \Bb_\xi \Bb_\xi^\top\right)^2 \right) 
\xholder^-
+
(\yholder^-)^\top 
\Exs \left( \Ib - \eta^2 \Bb_\xi^\top \Bb_\xi + \left(\eta^2 \Bb_\xi^\top \Bb_\xi\right)^2 \right)
\yholder^-
	\\&\quad\,
+
\eta^2\Exs_\xi \left[ (\gb^\xholder_\xi)^\top \left(\Ib + \eta^2\Bb_\xi \Bb_\xi^\top \right) \gb^\xholder_\xi \right]
+
\eta^2\Exs_\xi \left[ (\gb^\yholder_\xi)^\top \left(\Ib + \eta^2\Bb_\xi^\top \Bb_\xi \right) \gb^\yholder_\xi \right]
,
\end{aligned}\eeq
and in the last equality we use the \emph{independence} assumption of $\Bb_\xi$ and $[\gb^\xholder_\xi;\gb^\yholder_\xi]$ as in Assumption~\ref{assu_boundednoise_B}, so we have since 
$
\Exs_\xi [\Bb_{\xi}^\top\Bb_{\xi}]\preceq \lambda_{\max}(\widehat{\Mb}) \Ib_m
$
and
$
\Exs_\xi [\Bb_{\xi}\Bb_{\xi}^\top]\preceq \lambda_{\max}(\Mb) \Ib_n
$
that
$$\begin{aligned}
\Exs_\xi\left[ \|\Bb_{\xi} \gb_{\xi}^\yholder\|^2 \mid \gb_{\xi} \right]
	&=
\gb_{\xi}^\yholder
\Exs_\xi [\Bb_{\xi}^\top\Bb_{\xi}]
(\gb_{\xi}^\yholder)^\top
	\le
\gb_{\xi}^\yholder\left[
\lambda_{\max}(\widehat{\Mb}) \Ib_m
\right]
(\gb_{\xi}^\yholder)^\top
	=
\lambda_{\max}(\widehat{\Mb}) \|\gb_{\xi}^\yholder\|^2
,
\end{aligned}$$
and analogously
$$\begin{aligned}
\Exs_\xi\left[ \|\Bb_{\xi}^\top \gb_{\xi}^\xholder\|^2 \mid \gb_{\xi} \right]
	&\le
\lambda_{\max}(\Mb) \|\gb_{\xi}^\xholder\|^2
,
\end{aligned}$$
so summing up the above two and taking expectation gives, due to Assumption \ref{assu_boundednoise_B},
$$\begin{aligned}
	&\quad\,
\Exs_\xi\left[ \|\Bb_{\xi} \gb_{\xi}^\yholder\|^2 \right]
+
\Exs_\xi\left[ \|\Bb_{\xi}^\top \gb_{\xi}^\xholder\|^2 \right]
	\le
\lambda_{\max}(\widehat{\Mb}) \Exs_\xi\|\gb_{\xi}^\yholder\|^2
+
\lambda_{\max}(\Mb) \Exs_\xi\|\gb_{\xi}^\xholder\|^2
	\\&\le
\left(
\lambda_{\max}(\widehat{\Mb}) \lor \lambda_{\max}(\Mb)
\right)
\Exs\left[ \|\gb_{\xi}^\yholder\|^2 + \|\gb_{\xi}^\xholder\|^2 \right]
	=
\left(
\lambda_{\max}(\widehat{\Mb}) \lor \lambda_{\max}(\Mb)
\right)
\sigma_{\gb}^2
.
\end{aligned}$$
Therefore Eq.~\eqref{SEG_tol_g} gives, for any positive $\eta$, that
$$\begin{aligned}
\Exs_\xi\left[ \|\xholder\|^2 + \|\yholder\|^2 \right]
	&=
(\xholder^-)^\top
\left( \Ib - \eta^2 \left(\Mb - \eta^2 [\Exs_\xi\Mb_{\xi}^2]\right) \right)
\xholder^-
+
(\yholder^-)^\top
\left( \Ib - \eta^2 \left(\widehat{\Mb} - \eta^2 [\Exs_\xi\widehat{\Mb}_{\xi}^2]\right) \right)
\yholder^-
	\\&\quad\,
+
\eta^2\left[1 + \eta^2\left(\lambda_{\max}(\Mb) \lor \lambda_{\max}(\widehat{\Mb})\right)\right]
\sigma_{\gb}^2
	\\&\le
\left( 1 - \eta^2 \lambdaetaprime\right) \left( \|\xholder^-\|^2 + \|\yholder^-\|^2 \right)
+
\eta^2\left[1 + \eta^2\left(\lambda_{\max}(\Mb) \lor \lambda_{\max}(\widehat{\Mb})\right)\right]
\sigma_{\gb}^2
,
\end{aligned}$$
where $\lambdaetaprime$ was earlier defined in Eq.~\eqref{lambdaeta}.
Recursively applying this allows us to conclude
$$\begin{aligned}
\Exs \left[\|\xholder_\Tholder\|^2 + \|\yholder_\Tholder\|^2\right]
	&\le
\left(1 - \eta^2 \lambdaetaprime\right)^\Tholder \left[\|\xholder_0\|^2 + \|\yholder_0\|^2 \right]
	\\&\quad\,
+
\underbrace{
\left[\sum_{t=1}^\Tholder \left(1 - \eta^2 \lambdaetaprime\right)^{t-1}\eta^2\right]
}_{\text{
$=\eta^2\Quan_{\Tholder}(\eta)$ due to Eq.~\eqref{Quandef}
}}
\left[1 + \eta^2\left(\lambda_{\max}(\Mb) \lor \lambda_{\max}(\widehat{\Mb})\right)\right]
 \sigma_{\gb}^2
,
\end{aligned}$$
and hence concludes Eq.~\eqref{xygrowth_SEGg}.
The rest of the proof, under the condition $\eta\in (0,\eta_{\Mb}]$, follows from $\lambdaetaprime$'s definition in Eq.~\eqref{lambdaeta} along with Lemma \ref{lemm_etaMbound}.
\end{proof}

\subsection{Proof of  Theorem~\ref{theo_SEG_lower_bound}}

Theorem~\ref{theo_SEG_lower_bound} is, in fact, a variant of Proposition 1 of \citet{hsieh2020explore} under our assumptions.
Hence the proof therein applies, and we provide the statement in our work mainly for the sake of completeness.

\subsection{Auxiliary Proofs}\label{sec_proof,auxiliary1}
\subsubsection{Proof of Lemma \ref{lemm_etaMbound}}\label{sec_proof,lemm_etaMbound}
\begin{proof}[Proof of Lemma~\ref{lemm_etaMbound}]
\begin{enumerate}[label=(\arabic*)]
\item
Since $\Exs_\xi\Mb_{\xi}^2 - \Mb^2 = \Exs_\xi\left(\Mb_{\xi} - \Mb\right)^2 \succeq \mathbf{0}$ a simple discriminant argument of a quadratic $1-t+t^2$, which is greater than or equal to $3/4$ for all reals $t$, concludes
$$
\Ib - \eta^2 \left(\Mb - \eta^2 [\Exs_\xi\Mb_{\xi}^2]\right)
\succeq
\Ib - \eta^2 \left(\Mb - \eta^2 \Mb^2 \right)
\succeq
\frac34 \Ib
,
$$
and
$$
\Ib - \eta^2 \left(\widehat{\Mb} - \eta^2 [\Exs_\xi\widehat{\Mb}_{\xi}^2] \right)
\succeq
\Ib - \eta^2 \left(\widehat{\Mb} - \eta^2 \widehat{\Mb}^2 \right)
\succeq
\frac34 \Ib
,
$$
and hence
$$
1 - \eta^2\lambdaetaprime
	=
\lambda_{\max}\left(\Ib - \eta^2\left(\Mb - \eta^2 [\Exs_\xi\Mb_{\xi}^2] \right) \right)
\lor
\lambda_{\max}\left(\Ib - \eta^2\left(\widehat{\Mb} - \eta^2 [\Exs_\xi\widehat{\Mb}_{\xi}^2] \right) \right)
	\ge
\frac34
,
$$
proving Eq.~\eqref{lambdaetaprime_upper}.

\item
The definition of $\eta_{\Mb}$ as in Eq.~\eqref{etaMdef} gives for all $\eta\in \left(0,\eta_{\Mb}\right]$
$$
\eta^2 \Mb^{-1/2} [\Exs_\xi\Mb_{\xi}^2] \Mb^{-1/2}						\preceq	\Ib
	\qquad\text{and}\qquad
\eta^2\widehat{\Mb}^{-1/2} [\Exs_\xi\widehat{\Mb}_{\xi}^2] \widehat{\Mb}^{-1/2}	\preceq	\Ib
,
$$
and we have
$$
\Exs_\xi \Mb_\xi^2
	\preceq
\frac{1}{\eta_{\Mb}^2} \Mb
	\qquad\text{and}\qquad
\Exs_\xi \widehat{\Mb}_\xi^2
	\preceq
\frac{1}{\eta_{\Mb}^2} \widehat{\Mb}
$$
hold, which concludes when $\eta$ satisfies $\eta\in (0,\eta_{\Mb}]$ both
$$
\Mb - \eta^2 [\Exs_\xi\Mb_{\xi}^2]
	\succeq
\left(1-\frac{\eta^2}{\eta_{\Mb}^2}\right) \Mb
	\qquad\text{and}\qquad
\widehat{\Mb} - \eta^2 [\Exs_\xi\widehat{\Mb}_{\xi}^2]
	\succeq
\left(1-\frac{\eta^2}{\eta_{\Mb}^2}\right) \widehat{\Mb}
,
$$
and hence via Eq.~\eqref{lambdaeta}
$$\begin{aligned}
\lambdaetaprime
	=
\lambda_{\min}\left(\Mb - \eta^2 [\Exs_\xi\Mb_{\xi}^2] \right) 
\land
\lambda_{\min}\left(\widehat{\Mb} - \eta^2 [\Exs_\xi\widehat{\Mb}_{\xi}^2] \right)
	\ge
\left(1-\frac{\eta^2}{\eta_{\Mb}^2}\right) \left(\lambda_{\min}(\Mb)\land\lambda_{\min}(\widehat{\Mb})\right)
	\ge
0
\end{aligned}$$
holds for all $\eta\in (0,\eta_{\Mb}]$, which proves Eq.~\eqref{lambdaetaprime_lower}.

\item
Note $\Mb = \Exs_\xi\Mb_\xi$ and $\widehat{\Mb} = \Exs_\xi\widehat{\Mb}_\xi$, and $\eta_{\Mb} > 0$ is due to the finiteness of
$
\lambda_{\max}\left(
\Mb^{-1/2} [\Exs_\xi\Mb_\xi^2] \Mb^{-1/2}
\right)
$ under Assumption \ref{assu_boundednoise_A}.
For the second inequality note by the part (1) of the proof
$$
\Mb^{-1/2} [\Exs_\xi\Mb_\xi^2] \Mb^{-1/2}
	\succeq
\Mb^{-1/2} \Mb^2 \Mb^{-1/2}
	=
\Mb
,
$$
and
$$
\widehat{\Mb}^{-1/2} [\Exs_\xi\widehat{\Mb}_\xi^2] \widehat{\Mb}^{-1/2}
	\succeq
\widehat{\Mb}^{-1/2} \widehat{\Mb}^2 \widehat{\Mb}^{-1/2}
	=
\widehat{\Mb}
,
$$
hold due to Eq.~\eqref{sigmaA2sqinit}, and it is straightforward to check that all equalities hold in the $\Bb_\xi = \Bb$ a.s.~case, proving the second inequality of Eq.~\eqref{etaMbound}.
For the third inequality, we have
$$
\Mb - \Bb \Bb^\top			=	\Exs\left[ (\Bb_\xi - \Bb) (\Bb_\xi - \Bb)^\top \right]	\succeq	0
,
$$
and
$$
\widehat{\Mb} - \Bb^\top \Bb	=	\Exs\left[ (\Bb_\xi - \Bb)^\top (\Bb_\xi - \Bb) \right]	\succeq	0
,
$$
with equality holds when $\Bb_\xi = \Bb$ a.s.
Hence
$$
\lambda_{\max}(\Mb)
	\ge
\lambda_{\max}(\Bb\Bb^\top)
	\qquad\text{and}\qquad
\lambda_{\max}(\widehat{\Mb})
	\ge
\lambda_{\max}(\Bb^\top\Bb)
.
$$
Note $\lambda_{\max}(\Bb\Bb^\top) = \lambda_{\max}(\Bb^\top\Bb)$ as indicated by Lemma \ref{lemm_spectrum} gives the third inequality and the whole lemma.
\end{enumerate}
\end{proof}

\section{TECHNICAL ANALYSIS IN \S\ref{sec_SEGg}}\label{sec_SEGg_technical}
In this section, we collects the technical analyses and proofs of our main theoretical results.
The study of SEG in general stochastic setting \S\ref{sec_SEGg} for the averaged-iterate Theorem \ref{theo_SEG_B} and restarted-averaged-iterate Theorem \ref{theo_SEG_C}.
When narrowing down to the interpolation setting in \S\ref{sec_SEGg}, we state Theorem \ref{theo_SEGg_interpolation_C}.
For each of the theorems we first detail their full versions and accompany them with proofs, separately.

\subsection{Analysis of Theorem~\ref{theo_SEG_B}}\label{sec_proof,theo_SEG_B}

\begin{custom}{Theorem \ref{theo_SEG_B}, Full Version}
Let Assumptions~\ref{assu_boundednoise_A} and \ref{assu_boundednoise_B} hold and we assume that $\lambdaetaprime > 0$.
Under the condition on step size $\eta\in (0, \eta_{\Mb}]$ where $\eta_{\Mb}$ was earlier defined as in Eq.~\eqref{etaMdef}, we have for all $\Tholder\ge 0$ the following convergence rate holds for the averaged iterate $\overline{\xholder}_{\Tholder}, \overline{\yholder}_{\Tholder}$ defined in Theorem~\ref{theo_SEG_B}:
\beq\label{Asingleloop_SEG_B_glory}\begin{aligned}
	&\quad\,
\left(
\lambda_{\min}(\Bb\Bb^\top)\left(1 + \eta^2\lambda_{\min}(\Bb\Bb^\top)\right)
-
2\eta\sigma_{\Bb}^2 \sqrt{\lambda_{\max}(\Bb^\top\Bb)}
\right)
\Exs\left[
\left\|\overline{\xholder}_{\Tholder}\right\|^2 + \left\|\overline{\yholder}_{\Tholder}\right\|^2
\right]
	\\&\le
\Exs\left[
\left\| \Bb \overline{\yholder}_{\Tholder}
+
\eta  \Mb \overline{\xholder}_{\Tholder}
\right\|^2
+
\left\| \Bb^\top \overline{\xholder}_{\Tholder}
-
\eta \widehat{\Mb} \overline{\yholder}_{\Tholder}
\right\|^2
\right]
	\\&\le
\left(
\frac{8(1+\gamma)}{\eta^2 (\Tholder+1)^2}
+
\frac{2\left(1+\frac{1}{\gamma}\right)(\sigma_{\Bb}^2 + \eta^2\sigma_{\Bb,2}^2)}{\Tholder+1}
\right)
\left[\|\xholder_0\|^2 + \|\yholder_0\|^2\right]
	\\&\quad\,
+
\frac{6(1+\gamma) + 2\left(1+\frac{1}{\gamma}\right)(\sigma_{\Bb}^2 + \eta^2\sigma_{\Bb,2}^2) \lambdaetaprime^{-1}
}{\Tholder+1}
\left[1 + \eta^2\left(\lambda_{\max}(\Mb) \lor \lambda_{\max}(\widehat{\Mb})\right)\right] \sigma_{\gb}^2
,
\end{aligned}\eeq
where $\gamma \in (0,\infty)$ is arbitrary.
In addition when $\Bb_{\xi}, \Bb$ are square matrices, we have
\beq\label{Asingleloop_SEG_B_glory2}\begin{aligned}
\Exs\left[
\|\overline{\xholder}_{\Tholder}\|^2 + \|\overline{\yholder}_{\Tholder}\|^2
\right]
	&\le
\widehat{\prefactor}_{\Tholder+1}(\eta)
\cdot
\frac{\|\xholder_0\|^2 + \|\yholder_0\|^2}{(\Tholder+1)^2}
.
\end{aligned}\eeq
In above the prefactor is defined as%
\footnote{Here we interpret $0\cdot (+\infty)$ as $+\infty$ whenever it appears.}
$$\begin{aligned}
	&\quad\,
\widehat{\prefactor}_{\Tholder+1}(\eta)
	\\&\equiv
\left\{\begin{array}{ll}
+\infty
&
\hspace{-2.7in}
\text{if \footnotesize{$\lambda_{\min}(\Bb\Bb^\top)\left(1 + \eta^2\lambda_{\min}(\Bb\Bb^\top)\right)
\le
2\eta\sigma_{\Bb}^2 \sqrt{\lambda_{\max}(\Bb^\top\Bb)}$}
}
	\\
\frac{
8(1+\gamma)
+
\left(
2\left(1+\frac{1}{\gamma}\right)\eta^2(\sigma_{\Bb}^2 + \eta^2\sigma_{\Bb,2}^2)
+
\frac{6(1+\gamma) + 2\left(1+\frac{1}{\gamma}\right)(\sigma_{\Bb}^2 + \eta^2\sigma_{\Bb,2}^2) \lambdaetaprime^{-1}}{\|\xholder_0\|^2 + \|\yholder_0\|^2}\cdot
\eta^2\left[1 + \eta^2\left(\lambda_{\max}(\Mb) \lor \lambda_{\max}(\widehat{\Mb})\right)\right] \sigma_{\gb}^2
\right)
\cdot
(\Tholder+1)
}{
\eta^2\lambda_{\min}(\Bb\Bb^\top)\left(1 + \eta^2\lambda_{\min}(\Bb\Bb^\top)\right)
-
2\eta^3\sigma_{\Bb}^2 \sqrt{\lambda_{\max}(\Bb^\top\Bb)}
}
	&
\text{otherwise}
\end{array}\right.
,
\end{aligned}$$
and by setting $\eta = \hat\eta_{\Mb}(\alpha)$ defined earlier as in Eq.~\eqref{eta_choice}, we have
\beq\label{linearized_new}\begin{aligned}
\widehat{\prefactor}_{\Tholder+1}(\hat\eta_{\Mb}(\alpha))
	&\le
\frac{8(1+\gamma)}{(1-\alpha)\lambda_{\min}(\Bb\Bb^\top)}
\cdot
\frac{1}{\hat\eta_{\Mb}(\alpha)^2}
	\\&\quad\,
+
\frac{
2\left(1+\frac{1}{\gamma}\right)(\sigma_{\Bb}^2 + \hat\eta_{\Mb}(\alpha)^2\sigma_{\Bb,2}^2)\left[
\|\xholder_0\|^2 + \|\yholder_0\|^2 + \frac{3\sigma_{\gb}^2}{\lambda_{\min}(\Mb)\land\lambda_{\min}(\widehat{\Mb})}
\right]
+
9(1+\gamma)\sigma_{\gb}^2
}{(1-\alpha)\lambda_{\min}(\Bb\Bb^\top)}\cdot
\frac{\Tholder+1}{\|\xholder_0\|^2 + \|\yholder_0\|^2}
,
\end{aligned}\eeq
which recovers Eq.~\eqref{Asingleloop_SEG_B-appendix}.
\end{custom}

\begin{proof}[Proof of Theorem~\ref{theo_SEG_B}]
First, as long as $\eta^2 \lambdaetaprime\le 1/4$ the Eq.~\eqref{xygrowth_SEGg} is further bounded as
\beq\label{lastbddmax}\begin{aligned}
	&\quad\,
\Exs\left[\|\xholder_\Tholder\|^2 + \|\yholder_\Tholder\|^2\right]
	\\&\le
\left(1 - \eta^2 \lambdaetaprime\right)^\Tholder \left[\|\xholder_0\|^2 + \|\yholder_0\|^2 \right]
+
\eta^2\Quan_{\Tholder}(\eta)
\left[1 + \eta^2\left(\lambda_{\max}(\Mb) \lor \lambda_{\max}(\widehat{\Mb})\right)\right] \sigma_{\gb}^2
.
\end{aligned}\eeq
Depending on the behavior of $\Quan_{\Tholder}(\eta)$, the expected squared Euclidean norm admits two different upper bounds:
(i) when $\lambdaetaprime$ is bounded away from 0, uniform bound holds with its limit being bounded by a quantity that is inverse proportional to $\lambdaetaprime$;
(ii) when $\lambdaetaprime$ approaches zero, the quantity eventually grows linearly at a rate that does not depend on $\lambdaetaprime$.
In our analysis we will assume that $\lambdaetaprime$ is bounded away from zero while applying two different bounds interchangeably.

Returning to the SEG update Eq.~\eqref{SEGupdate_combined} which we repeat as
\beq\tag{\ref{SEGupdate_combined}}\begin{aligned}
\xholder_t
	&=
\xholder_{t-1} - \eta^2 \Bb_{\xi,t} \Bb_{\xi,t}^\top \xholder_{t-1} - \eta\left[ \Bb_{\xi,t} \yholder_{t-1} + \gb^\xholder_{\xi,t}\right] - \eta^2 \Bb_{\xi,t} \gb^\yholder_{\xi,t}
	\\
\yholder_t
	&= 
\yholder_{t-1} - \eta^2 \Bb_{\xi,t}^\top \Bb_{\xi,t} \yholder_{t-1} + \eta\left[ \Bb_{\xi,t}^\top \xholder_{t-1} + \gb^\yholder_{\xi,t}\right] - \eta^2 \Bb_{\xi,t}^\top \gb^\xholder_{\xi,t}
.
\end{aligned}\eeq
Setting $\eta = \hat\eta_{\Mb}(\alpha)$ as in Eq.~\eqref{eta_choice} and telescoping both sides of the update rule~Eq.~\eqref{SEGupdate_combined} for $t = 1,\ldots,\Tholder$ gives
$$\begin{aligned}
\xholder_\Tholder - \xholder_0
&=
-\eta^2 \sum_{t=1}^\Tholder  \Bb_{\xi, t} \Bb_{\xi, t}^\top \xholder_{t-1} 
- \eta \sum_{t=1}^\Tholder \left[ \Bb_{\xi, t} \yholder_{t-1} + \gb^\xholder_{\xi,t}\right] 
 - \eta^2 \sum_{t=1}^\Tholder  \Bb_{\xi, t} \gb^\yholder_{\xi, t}
\\
\yholder_\Tholder - \yholder_0
&=
-
\eta^2 \sum_{t=1}^\Tholder \Bb_{\xi, t}^\top \Bb_{\xi, t} \yholder_{t-1}
+
\eta \sum_{t=1}^\Tholder \left[ \Bb_{\xi, t}^\top \xholder_{t-1} + \gb^\yholder_{\xi,t}\right] 
-
\eta^2 \sum_{t=1}^\Tholder  \Bb_{\xi, t}^\top \gb^\xholder_{\xi, t}
.
\end{aligned}$$
Manipulating gives
\beq\label{SEG_tolequality}\begin{aligned}
\frac{1}{\Tholder}\sum_{t=1}^\Tholder \Bb_{\xi, t} \yholder_{t-1}
+
\frac{\eta}{\Tholder}\sum_{t=1}^\Tholder   \Bb_{\xi, t} \Bb_{\xi, t}^\top \xholder_{t-1} 
	&=
\frac{\xholder_\Tholder - \xholder_0}{-\eta\Tholder}
- 
\frac{1}{\Tholder}\sum_{t=1}^\Tholder \gb^\xholder_{\xi,t}
 - 
\frac{\eta}{\Tholder} \sum_{t=1}^\Tholder  \Bb_{\xi, t} \gb^\yholder_{\xi, t}
,
	\\
\frac{1}{\Tholder}\sum_{t=1}^\Tholder \Bb_{\xi, t}^\top \xholder_{t-1}
-
\frac{\eta}{\Tholder} \sum_{t=1}^\Tholder \Bb_{\xi, t}^\top \Bb_{\xi, t} \yholder_{t-1}
	&=
\frac{\yholder_\Tholder - \yholder_0}{\eta\Tholder}
-
\frac{1}{\Tholder}\sum_{t=1}^\Tholder \gb^\yholder_{\xi,t}
+
\frac{\eta}{\Tholder}\sum_{t=1}^\Tholder  \Bb_{\xi, t}^\top \gb^\xholder_{\xi, t}
.
\end{aligned}
\eeq

Now we try to bound the norm of the left hands in the above two displays.
Young's inequality gives that for fixed $\gamma>0$, $\|a+b\|^2 \le (1+\gamma)\|a\|^2 + (1+\frac{1}{\gamma})\|b\|^2$ so $\|a\|^2 \ge \frac{1}{1+\gamma}\|a+b\|^2 - \frac{1}{\gamma}\|b\|^2$ holds for two vectors $a,b$ of same dimensions,
\beq\label{x_update_bound}\begin{aligned}
	&\quad\,
\Exs\left\|
\frac{1}{\Tholder}\sum_{t=1}^\Tholder \Bb_{\xi, t} \yholder_{t-1}
+
\frac{\eta}{\Tholder}\sum_{t=1}^\Tholder   \Bb_{\xi, t} \Bb_{\xi, t}^\top \xholder_{t-1} 
\right\|^2
	\\&=
\Exs\left\|
\Bb\overline{\yholder}_{\Tholder-1}
+
\eta\Mb\overline{\xholder}_{\Tholder-1} 
+
\frac{1}{\Tholder}\sum_{t=1}^\Tholder (\Bb_{\xi, t} - \Bb) \yholder_{t-1}
+
\frac{\eta}{\Tholder}\sum_{t=1}^\Tholder \left( \Bb_{\xi, t} \Bb_{\xi, t}^\top - \Mb \right)\xholder_{t-1} 
\right\|^2
	\\&\ge
\frac{1}{1+\gamma}\Exs\left\|
\Bb \overline{\yholder}_{\Tholder-1}
+
\eta  \Mb \overline{\xholder}_{\Tholder-1}
\right\|^2
-
\frac{1}{\gamma}\Exs\left\|
\frac{1}{\Tholder}\sum_{t=1}^\Tholder (\Bb_{\xi, t} - \Bb) \yholder_{t-1}
+
\frac{\eta}{\Tholder}\sum_{t=1}^\Tholder  \left( \Bb_{\xi, t}\Bb_{\xi, t}^\top - \Mb \right)\xholder_{t-1} 
\right\|^2
.
\end{aligned}\eeq
Analogously,
\beq\label{y_update_bound}\begin{aligned}
	&\quad\,
\Exs\left\|
\frac{1}{\Tholder}\sum_{t=1}^\Tholder \Bb_{\xi, t}^\top \xholder_{t-1}
-
\frac{\eta}{\Tholder} \sum_{t=1}^\Tholder \Bb_{\xi, t}^\top \Bb_{\xi, t} \yholder_{t-1}
\right\|^2 
	\\&=
\Exs\left\|
\Bb^\top\overline{\xholder}_{\Tholder-1}
-
\eta\widehat\Mb\overline{\yholder}_{\Tholder-1} 
+
\frac{1}{\Tholder}\sum_{t=1}^\Tholder (\Bb_{\xi, t} - \Bb)^\top \xholder_{t-1}
-
\frac{\eta}{\Tholder}\sum_{t=1}^\Tholder \left(\Bb_{\xi, t}^\top \Bb_{\xi, t} - \widehat\Mb\right)\yholder_{t-1} 
\right\|^2
	\\&\ge
\frac{1}{1+\gamma}\Exs\left\| \Bb^\top \overline{\xholder}_{\Tholder-1}
-
\eta \widehat{\Mb} \overline{\yholder}_{\Tholder-1}
\right\|^2
-
\frac{1}{\gamma}\Exs\left\|
\frac{1}{\Tholder}\sum_{t=1}^\Tholder (\Bb_{\xi, t} - \Bb)^\top \xholder_{t-1}
-
\frac{\eta}{\Tholder}\sum_{t=1}^\Tholder  \left( \Bb_{\xi, t}^\top \Bb_{\xi, t} - \widehat{\Mb} \right)\yholder_{t-1} 
\right\|^2
.
\end{aligned}\eeq
Combining the above two displays Eq.~\eqref{x_update_bound},~Eq.~\eqref{y_update_bound} with Eq.~\eqref{SEG_tolequality} we have
\beq\label{generalbdd}\begin{aligned}
	&\quad\,
\frac{1}{1+\gamma}\Exs\left[
\left\| \Bb \overline{\yholder}_{\Tholder-1}
+
\eta  \Mb \overline{\xholder}_{\Tholder-1}
\right\|^2
+
\left\| \Bb^\top \overline{\xholder}_{\Tholder-1}
-
\eta \widehat{\Mb} \overline{\yholder}_{\Tholder-1}
\right\|^2
\right]
	\\&\quad\,
-
\frac{1}{\gamma}\Exs\left\|
\frac{1}{\Tholder}\sum_{t=1}^\Tholder (\Bb_{\xi, t} - \Bb) \yholder_{t-1}
+
\frac{\eta}{\Tholder}\sum_{t=1}^\Tholder  \left( \Bb_{\xi, t}\Bb_{\xi, t}^\top - \Mb \right)\xholder_{t-1} 
\right\|^2
	\\&\quad\,
-
\frac{1}{\gamma}\Exs\left\|
\frac{1}{\Tholder}\sum_{t=1}^\Tholder (\Bb_{\xi, t} - \Bb)^\top \xholder_{t-1}
-
\frac{\eta}{\Tholder}\sum_{t=1}^\Tholder  \left( \Bb_{\xi, t}^\top \Bb_{\xi, t} - \widehat{\Mb} \right)\yholder_{t-1} 
\right\|^2
	\\&\le
\Exs\left\|
\frac{\xholder_\Tholder - \xholder_0}{-\eta\Tholder}
- 
\frac{1}{\Tholder}\sum_{t=1}^\Tholder \gb^\xholder_{\xi,t}
 - 
\frac{\eta}{\Tholder} \sum_{t=1}^\Tholder  \Bb_{\xi, t} \gb^\yholder_{\xi, t}
\right\|^2 
+
\Exs\left\|
\frac{\yholder_\Tholder - \yholder_0}{\eta\Tholder}
-
\frac{1}{\Tholder}\sum_{t=1}^\Tholder \gb^\yholder_{\xi,t}
+
\frac{\eta}{\Tholder}\sum_{t=1}^\Tholder  \Bb_{\xi, t}^\top \gb^\xholder_{\xi, t}
\right\|^2
	\\&\le
(1+\beta)\Exs\left\|
\frac{\xholder_\Tholder - \xholder_0}{-\eta\Tholder}
\right\|^2 
+
(1+\beta)\Exs\left\|
\frac{\yholder_\Tholder - \yholder_0}{\eta\Tholder}
\right\|^2
	\\&\quad\,
+
\left(1+\frac{1}{\beta}\right)\Exs\left\|
- 
\frac{1}{\Tholder}\sum_{t=1}^\Tholder \gb^\xholder_{\xi,t}
 - 
\frac{\eta}{\Tholder} \sum_{t=1}^\Tholder  \Bb_{\xi, t} \gb^\yholder_{\xi, t}
\right\|^2 
+
\left(1+\frac{1}{\beta}\right)\Exs\left\|
-
\frac{1}{\Tholder}\sum_{t=1}^\Tholder \gb^\yholder_{\xi,t}
+
\frac{\eta}{\Tholder}\sum_{t=1}^\Tholder  \Bb_{\xi, t}^\top \gb^\xholder_{\xi, t}
\right\|^2
,
\end{aligned}\eeq
where the last inequality is an application of Young's that involves an arbitrary fixed number $\beta\in (0,\infty)$.
The rest of this proof follows in three steps:

\begin{enumerate}[label=(\roman*)]
\item
As a first step, we have from Eq.~\eqref{lastbddmax} along with Lemma~\ref{lemm_etaMbound}
$$\begin{aligned}
	&\quad\,
\Exs\left\|
\frac{\xholder_\Tholder - \xholder_0}{-\eta\Tholder}
\right\|^2 
+
\Exs\left\|
\frac{\yholder_\Tholder - \yholder_0}{\eta\Tholder}
\right\|^2
	\\&\le
\frac{2}{\eta^2\Tholder^2}\left(\|\xholder_\Tholder\|^2 + \|\xholder_0\|^2 \right)
+
\frac{2}{\eta^2\Tholder^2}\left(\|\yholder_\Tholder\|^2 + \|\yholder_0\|^2 \right)
	\\&\le
\frac{2}{\eta^2\Tholder^2}\left(
\left(1 - \eta^2 \lambdaetaprime\right)^\Tholder \left[\|\xholder_0\|^2 + \|\yholder_0\|^2 \right]
+
\eta^2\Quan_{\Tholder}(\eta)
\left[1 + \eta^2\left(\lambda_{\max}(\Mb) \lor \lambda_{\max}(\widehat{\Mb})\right)\right] \sigma_{\gb}^2
\right)
	\\&\quad\,
+
\frac{2}{\eta^2\Tholder^2}\left[
\|\xholder_0\|^2 + \|\yholder_0\|^2
\right]
	\\&\le
\frac{2}{\eta^2\Tholder^2}\left[
2\|\xholder_0\|^2 + 2\|\yholder_0\|^2
+
\eta^2\Quan_{\Tholder}(\eta)
\left[1 + \eta^2\left(\lambda_{\max}(\Mb) \lor \lambda_{\max}(\widehat{\Mb})\right)\right] \sigma_{\gb}^2
\right]
	\\&\le
\frac{4}{\eta^2\Tholder^2}[\|\xholder_0\|^2 + \|\yholder_0\|^2]
+
\frac{2\left[1 + \eta^2\left(\lambda_{\max}(\Mb) \lor \lambda_{\max}(\widehat{\Mb})\right)\right]}{\Tholder} \sigma_{\gb}^2
.
\end{aligned}$$

\item
A second step is, due to Assumption \ref{assu_boundednoise_B}, standard $L^2$ martingale analysis gives
$$\begin{aligned}
	&\quad\,
\Exs\left\|
- 
\frac{1}{\Tholder}\sum_{t=1}^\Tholder \gb^\xholder_{\xi,t}
 - 
\frac{\eta}{\Tholder} \sum_{t=1}^\Tholder  \Bb_{\xi, t} \gb^\yholder_{\xi, t}
\right\|^2 
+
\Exs\left\|
-
\frac{1}{\Tholder}\sum_{t=1}^\Tholder \gb^\yholder_{\xi,t}
+
\frac{\eta}{\Tholder}\sum_{t=1}^\Tholder  \Bb_{\xi, t}^\top \gb^\xholder_{\xi, t}
\right\|^2
	\\&=
\frac{1}{\Tholder^2}\sum_{t=1}^\Tholder \Exs\left\|
- 
\gb^\xholder_{\xi,t}
 - 
\eta \Bb_{\xi, t} \gb^\yholder_{\xi, t}
\right\|^2 
+
\frac{1}{\Tholder^2}\sum_{t=1}^\Tholder \Exs\left\|
-
\gb^\yholder_{\xi,t}
+
\eta \Bb_{\xi, t}^\top \gb^\xholder_{\xi, t}
\right\|^2
	\\&=
\frac{1}{\Tholder}
\Exs_\xi\left\| - \gb_{\xi}^\xholder - \eta \Bb_{\xi} \gb_{\xi}^\yholder\right\|^2
+
\frac{1}{\Tholder}
\Exs_\xi\left\| - \gb_{\xi}^\yholder + \eta \Bb_{\xi}^\top \gb_{\xi}^\xholder\right\|^2
	\\&=
\frac{1}{\Tholder}
\Exs_\xi\|\gb_{\xi}^\xholder\|^2
+
\frac{\eta^2}{\Tholder}
\Exs_\xi\|\Bb_{\xi} \gb_{\xi}^\yholder\|^2
+
\frac{1}{\Tholder}
\Exs_\xi\|\gb_{\xi}^\yholder\|^2
+
\frac{\eta^2}{\Tholder}
\Exs_\xi\|\Bb_{\xi}^\top \gb_{\xi}^\xholder\|^2
	\\&\quad\,
\underbrace{
+
\frac{2\eta}{\Tholder}
\Exs_\xi\langle \gb_{\xi}^\xholder, \Bb_{\xi} \gb_{\xi}^\yholder\rangle
-
\frac{2\eta}{\Tholder}
\Exs_\xi\langle \gb_{\xi}^\yholder, \Bb_{\xi}^\top \gb_{\xi}^\xholder\rangle
}_{
\text{cross term} = 0
}
	\\&\le
\frac{1}{\Tholder}\Exs_\xi\left[
\|\gb_{\xi}^\xholder\|^2
+
\|\gb_{\xi}^\yholder\|^2
\right]
+
\frac{\eta^2}{\Tholder}\Exs_\xi\left[
\|\Bb_{\xi} \gb_{\xi}^\yholder\|^2
+
\|\Bb_{\xi}^\top \gb_{\xi}^\xholder\|^2
\right]
	\\&\le
\frac{1 + \eta^2\left(\lambda_{\max}(\Mb) \lor \lambda_{\max}(\widehat{\Mb})\right)}{\Tholder}\sigma_{\gb}^2
,
\end{aligned}$$
where a similar analysis as in the proof of Theorem~\ref{theo_SEG_A} was adopted.

\item
A third step is that, due to $\lambdaetaprime\ge 0$ of Lemma \ref{lemm_etaMbound},
$$\begin{aligned}
	&\quad\,
\Exs\left\|
\frac{1}{\Tholder}\sum_{t=1}^\Tholder (\Bb_{\xi, t} - \Bb) \yholder_{t-1}
+
\frac{\eta}{\Tholder}\sum_{t=1}^\Tholder  \left( \Bb_{\xi, t}\Bb_{\xi, t}^\top - \Mb \right)\xholder_{t-1} 
\right\|^2
	\\&\quad\,
+
\Exs\left\|
\frac{1}{\Tholder}\sum_{t=1}^\Tholder (\Bb_{\xi, t} - \Bb)^\top \xholder_{t-1}
-
\frac{\eta}{\Tholder}\sum_{t=1}^\Tholder  \left( \Bb_{\xi, t}^\top \Bb_{\xi, t} - \widehat{\Mb} \right)\yholder_{t-1} 
\right\|^2
	\\&\le
\frac{2}{\Tholder^2}\sum_{t=1}^\Tholder \Exs\left\|
(\Bb_{\xi, t} - \Bb) \yholder_{t-1}
\right\|^2
+
\frac{2\eta^2}{\Tholder^2}\sum_{t=1}^\Tholder  \Exs\left\|
\left( \Bb_{\xi, t}\Bb_{\xi, t}^\top - \Mb \right)\xholder_{t-1} 
\right\|^2
	\\&\quad\,
+
\frac{2}{\Tholder^2}\sum_{t=1}^\Tholder \Exs\left\|
(\Bb_{\xi, t} - \Bb)^\top \xholder_{t-1}
\right\|^2
+
\frac{2\eta^2}{\Tholder^2}\sum_{t=1}^\Tholder  \Exs\left\|
\left( \Bb_{\xi, t}^\top \Bb_{\xi, t} - \widehat{\Mb} \right)\yholder_{t-1} 
\right\|^2
	\\&\le
\frac{2(\sigma_{\Bb}^2 + \eta^2\sigma_{\Bb,2}^2)}{\Tholder^2}
\sum_{t=1}^\Tholder \Exs\left[ \|\xholder_{t-1}\|^2 + \|\yholder_{t-1}\|^2 \right]
	\\&\le
\frac{2(\sigma_{\Bb}^2 + \eta^2\sigma_{\Bb,2}^2)}{\Tholder^2}\sum_{t=1}^\Tholder\left(
\left(1 - \eta^2 \lambdaetaprime\right)^{t-1} \left[\|\xholder_0\|^2 + \|\yholder_0\|^2 \right]
\right.\\&\hspace{1.7in}\,\left.
+
\eta^2\Quan_{t-1}(\eta)
\left[1 + \eta^2\left(\lambda_{\max}(\Mb) \lor \lambda_{\max}(\widehat{\Mb})\right)\right] \sigma_{\gb}^2
\right)
	\\&\le
\frac{2(\sigma_{\Bb}^2 + \eta^2\sigma_{\Bb,2}^2)}{\Tholder}
\left(
\|\xholder_0\|^2 + \|\yholder_0\|^2
+
\left[1 + \eta^2\left(\lambda_{\max}(\Mb) \lor \lambda_{\max}(\widehat{\Mb})\right)\right]\lambdaetaprime^{-1}\sigma_{\gb}^2
\right)
,
\end{aligned}$$
where since $\eta^2\lambdaetaprime \in [0,1/4]$ we applied the result of Eq.~\eqref{lastbddmax}.
\end{enumerate}

Putting the above pieces together, along with Eq.~\eqref{generalbdd}, yields for any $\Tholder\ge 0$ we have%
\footnote{For simplicity we optimize the numerical constants on $\gamma$ and take $\beta=1$.}
$$\begin{aligned}
	&\quad\,
\frac{1}{1+\gamma}\Exs\left[
\left\| \Bb \overline{\yholder}_{\Tholder-1}
+
\eta  \Mb \overline{\xholder}_{\Tholder-1}
\right\|^2
+
\left\| \Bb^\top \overline{\xholder}_{\Tholder-1}
-
\eta \widehat{\Mb} \overline{\yholder}_{\Tholder-1}
\right\|^2
\right]
	\\&\le
2\Exs\left\|
\frac{\xholder_\Tholder - \xholder_0}{-\eta\Tholder}
\right\|^2 
+
2\Exs\left\|
\frac{\yholder_\Tholder - \yholder_0}{\eta\Tholder}
\right\|^2
	\\&\quad\,
+
2\Exs\left\|
- 
\frac{1}{\Tholder}\sum_{t=1}^\Tholder \gb^\xholder_{\xi,t}
 - 
\frac{\eta}{\Tholder} \sum_{t=1}^\Tholder  \Bb_{\xi, t} \gb^\yholder_{\xi, t}
\right\|^2 
+
2\Exs\left\|
-
\frac{1}{\Tholder}\sum_{t=1}^\Tholder \gb^\yholder_{\xi,t}
+
\frac{\eta}{\Tholder}\sum_{t=1}^\Tholder  \Bb_{\xi, t}^\top \gb^\xholder_{\xi, t}
\right\|^2
	\\&\quad\,
+
\frac{1}{\gamma}\Exs\left\|
\frac{1}{\Tholder}\sum_{t=1}^\Tholder (\Bb_{\xi, t} - \Bb) \yholder_{t-1}
+
\frac{\eta}{\Tholder}\sum_{t=1}^\Tholder  \left( \Bb_{\xi, t}\Bb_{\xi, t}^\top - \Mb \right)\xholder_{t-1} 
\right\|^2
	\\&\quad\,
+
\frac{1}{\gamma}\Exs\left\|
\frac{1}{\Tholder}\sum_{t=1}^\Tholder (\Bb_{\xi, t} - \Bb)^\top \xholder_{t-1}
-
\frac{\eta}{\Tholder}\sum_{t=1}^\Tholder  \left( \Bb_{\xi, t}^\top \Bb_{\xi, t} - \widehat{\Mb} \right)\yholder_{t-1} 
\right\|^2
,
\end{aligned}$$
which is further bounded by
$$\begin{aligned}
	&\quad\quad\,
\frac{8}{\eta^2\Tholder^2}[\|\xholder_0\|^2 + \|\yholder_0\|^2]
+
\frac{4\left[1 + \eta^2\left(\lambda_{\max}(\Mb) \lor \lambda_{\max}(\widehat{\Mb})\right)\right]}{\Tholder} \sigma_{\gb}^2
	\\&\quad\,
+
\frac{2\left[1 + \eta^2\left(\lambda_{\max}(\Mb) \lor \lambda_{\max}(\widehat{\Mb})\right)\right]}{\Tholder}\sigma_{\gb}^2
	\\&\quad\,
+
\frac{1}{\gamma}\cdot\frac{2(\sigma_{\Bb}^2 + \eta^2\sigma_{\Bb,2}^2)}{\Tholder}\left(
\left[\|\xholder_0\|^2 + \|\yholder_0\|^2 \right]
+
\left[1 + \eta^2\left(\lambda_{\max}(\Mb) \lor \lambda_{\max}(\widehat{\Mb})\right)\right]\lambdaetaprime^{-1}\sigma_{\gb}^2
\right)
	\\&\le
\left(
\frac{8}{\eta^2 \Tholder^2}
+
\frac{2(\sigma_{\Bb}^2 + \eta^2\sigma_{\Bb,2}^2)}{\gamma\Tholder}
\right)
\left[\|\xholder_0\|^2 + \|\yholder_0\|^2\right]
	\\&\quad\,
+
\frac{6 + \frac{2}{\gamma}(\sigma_{\Bb}^2 + \eta^2\sigma_{\Bb,2}^2) \lambdaetaprime^{-1}
}{\Tholder}
\left[1 + \eta^2\left(\lambda_{\max}(\Mb) \lor \lambda_{\max}(\widehat{\Mb})\right)\right] \sigma_{\gb}^2
,
\end{aligned}$$
and by rearranging the terms in the last display along with Eq.~\eqref{eqbdd1} we have in finale (and shifting the time index forward by one)
$$\begin{aligned}
	&\quad\,
\left(
\lambda_{\min}(\Bb\Bb^\top)\left(1 + \eta^2\lambda_{\min}(\Bb\Bb^\top)\right)
-
2\eta\sigma_{\Bb}^2 \sqrt{\lambda_{\max}(\Bb^\top\Bb)}
\right)
\Exs\left[
\left\|\overline{\xholder}_{\Tholder}\right\|^2 + \left\|\overline{\yholder}_{\Tholder}\right\|^2
\right]
	\\&\le
\Exs\left[
\left\| \Bb \overline{\yholder}_{\Tholder}
+
\eta  \Mb \overline{\xholder}_{\Tholder}
\right\|^2
+
\left\| \Bb^\top \overline{\xholder}_{\Tholder}
-
\eta \widehat{\Mb} \overline{\yholder}_{\Tholder}
\right\|^2
\right]
	\\&\le
\left(
\frac{8(1+\gamma)}{\eta^2 (\Tholder+1)^2}
+
\frac{2\left(1+\frac{1}{\gamma}\right)(\sigma_{\Bb}^2 + \eta^2\sigma_{\Bb,2}^2)}{\Tholder+1}
\right)
\left[\|\xholder_0\|^2 + \|\yholder_0\|^2\right]
	\\&\quad\,
+
\frac{6(1+\gamma) + 2\left(1+\frac{1}{\gamma}\right)(\sigma_{\Bb}^2 + \eta^2\sigma_{\Bb,2}^2) \lambdaetaprime^{-1}
}{\Tholder+1}
\left[1 + \eta^2\left(\lambda_{\max}(\Mb) \lor \lambda_{\max}(\widehat{\Mb})\right)\right] \sigma_{\gb}^2
,
\end{aligned}$$
where we used the iterated laws of expectation at multiple occasions, as well as the property of $L^2$ martingale differences as well as the definitions Eq.~\eqref{sigmaAsq} and Eq.~\eqref{sigmaA2sqinit} in Assumption \ref{assu_boundednoise_A}.
This concludes Eq.~\eqref{Asingleloop_SEG_B_glory}.
The rest of the proof sits upon the application of Lemma \ref{lemm_etaMbound}, esp.~Eq.~\eqref{lambdaetaprime_lower} and the fact that
$
1 + \hat\eta_{\Mb}(\alpha)^2\left(\lambda_{\max}(\Mb) \lor \lambda_{\max}(\widehat{\Mb})\right)
\le
\frac32
$, concluding the whole proof of Theorem \ref{theo_SEG_B}.

\end{proof}

\paragraph{Discussion}\label{rema_K2inv_upper}
We remark that the magnitude of $\Quan_{\Tholder+1}(\eta)$ can be either $O(1)$ or $O(\Tholder)$, depending on whether $\lambdaetaprime$ is bounded away from zero or sufficiently close to zero.
When applying iteration average, one needs to maximize the step size to achieve a sharp bound in which case it is sufficient to replace $\Quan_{\Tholder+1}(\eta)$ in the bound by $\Tholder+1$ instead of $\frac{1}{\eta^2 \lambdaetaprime}$.
In special, the dependency on $\left[\|\xholder_0\|^2 + \|\yholder_0\|^2\right]$ can be improved to $O(\frac{1}{(\Tholder+1)^2})$ if we adopt the $\frac{1}{\eta^2 \lambdaetaprime}$ bound for $\Quan_{\Tholder+1}(\eta)$, achieving
\beq\label{K2inv_upper}\begin{aligned}
	&\quad\,
(1-\alpha)\lambda_{\min}(\Bb\Bb^\top)\left(1 + \eta^2\lambda_{\min}(\Bb\Bb^\top)\right)
\Exs\left[
\left\|\overline{\xholder}_{\Tholder}\right\|^2 + \left\|\overline{\yholder}_{\Tholder}\right\|^2
\right]
	\\&\le
\frac{16 + 4(\sigma_{\Bb}^2 + \eta^2\sigma_{\Bb,2}^2)\lambdaetaprime^{-1}}{\eta^2 (\Tholder+1)^2}
\left[\|\xholder_0\|^2 + \|\yholder_0\|^2\right]
	\\&\quad\,
+
\frac{12 + 4(\sigma_{\Bb}^2 + \eta^2\sigma_{\Bb,2}^2) \lambdaetaprime^{-1}
}{\Tholder+1}
\left[1 + \eta^2\left(\lambda_{\max}(\Mb) \lor \lambda_{\max}(\widehat{\Mb})\right)\right] \sigma_{\gb}^2
,
\end{aligned}\eeq
which recovers \eqref{Asingleloop_SEG_B}.
In the upcoming technical analysis for restarting, we do not, however, utilize this upper bound.

\subsection{Analysis of Theorem~\ref{theo_SEG_C}}\label{sec_proof,theo_SEG_C}
\begin{custom}{Theorem \ref{theo_SEG_C}, Full Version}
Under Assumptions~\ref{assu_boundednoise_A} and \ref{assu_boundednoise_B} and assume that $\Bb_{\xi}, \Bb$ are square matrices, we apply restarting at $\epoch=1,2,\dots,\Epoch$ after $\Tholder \ge \Tholder_{\epoch}$ steps where
\beq\label{Tholder_epoch_prime}
\begin{aligned}
&
\Tholder_{\epoch}
	=
\left\lceil
\frac{q_2 + \sqrt{q_2^2 + 4q_1q_3}}{2q_3}
\right\rceil - 1
,
\end{aligned}\eeq
with
\beq\label{q1q2q3}\begin{aligned}
q_1
	&\equiv
\frac{16}{(1-\alpha)\lambda_{\min}(\Bb\Bb^\top)}
\cdot
\frac{e^{2-2\epoch}[\|\xholder_0\|^2 + \|\yholder_0\|^2]}{\hat\eta_{\Mb}(\alpha)^2}
\\
q_2
	&\equiv
\frac{
4(\sigma_{\Bb}^2 + \hat\eta_{\Mb}(\alpha)^2\sigma_{\Bb,2}^2)\left[
e^{2-2\epoch}[\|\xholder_0\|^2 + \|\yholder_0\|^2] + \frac{3\sigma_{\gb}^2}{\lambda_{\min}(\Mb)\land\lambda_{\min}(\widehat{\Mb})}
\right]
+
18\sigma_{\gb}^2
}{(1-\alpha)\lambda_{\min}(\Bb\Bb^\top)}
\\
q_3
	&\equiv
\frac{\|\xholder_0\|^2 + \|\yholder_0\|^2}{e^{2\epoch}}
,
\end{aligned}\eeq
where we denote
$$
\Epoch
=
\left\lceil 
\dfrac12 \log\left(
\frac{\lambda_{\min}(\Mb)\land\lambda_{\min}(\widehat{\Mb})}{3\sigma_{\gb}^2}
[\|\xholder_0\|^2 + \|\yholder_0\|^2]
\right)
\right\rceil
.
$$
Then for $\widehat{\Epoch}=1,\dots,\Epoch$ where $\Tholder = \sum_{\epoch=1}^{\widehat{\Epoch}} \Tholder_{\epoch}$ the iteration has the \emph{expected squared Euclidean metric} that is discounted by a factor of $1/e^{2\widehat{\Epoch}}$:
$$\begin{aligned}
\Exs\left[\|\hat{\xholder}_{\Tholder}\|^2 + \|\hat{\yholder}_{\Tholder}\|^2\right]
	\le
\frac{1}{e^{2\widehat{\Epoch}}}\left[ \|\xholder_0\|^2 + \|\yholder_0\|^2\right]
,
\end{aligned}$$
and for $\Tholder = \sum_{\epoch=1}^{\Epoch} \Tholder_{\epoch} + \hat\Tholder$, $\hat\Tholder=0,1,\dots$ where SEG with aforementioned restarting and (tail-) iteration averaging achieves
\beq\label{Asingleloop_SEG_C_glory}\begin{aligned}
\Exs\left[\|\bar{\xholder}_{\Tholder}\|^2 + \|\bar{\yholder}_{\Tholder}\|^2\right]
	&\le
\frac{16}{(1-\alpha)\lambda_{\min}(\Bb\Bb^\top)}
\cdot
\frac{\frac{3\sigma_{\gb}^2}{\lambda_{\min}(\Mb)\land\lambda_{\min}(\widehat{\Mb})}}{\hat\eta_{\Mb}(\alpha)^2(\hat\Tholder+1)^2}
	\\&\quad\,
+
\frac{
4(\sigma_{\Bb}^2 + \hat\eta_{\Mb}(\alpha)^2\sigma_{\Bb,2}^2)\cdot
\frac{6\sigma_{\gb}^2}{\lambda_{\min}(\Mb)\land\lambda_{\min}(\widehat{\Mb})} 
+
18\sigma_{\gb}^2
}{(1-\alpha)\lambda_{\min}(\Bb\Bb^\top)}
\cdot
\frac{1}{\hat\Tholder+1}
	\\&\hspace{-1.5in}
=\left[
1 + \frac{
\frac{8}{3(\hat\Tholder+1)}
+
O(\hat\eta_{\Mb}(\alpha)^2\sigma_{\Bb}^2 + \hat\eta_{\Mb}(\alpha)^4\sigma_{\Bb,2}^2)
}{
\hat\eta_{\Mb}(\alpha)^2\left(\lambda_{\min}(\Mb)\land\lambda_{\min}(\widehat{\Mb})\right)
}
\right]
\cdot
\frac{18\sigma_{\gb}^2}{(1-\alpha)\lambda_{\min}(\Bb\Bb^\top)}
\cdot
\frac{1}{\hat\Tholder+1}
,
\end{aligned}\eeq
which recovers Eq.~\eqref{Asingleloop_SEG_C}.
\end{custom}

\begin{proof}[Proof of Theorem~\ref{theo_SEG_C}]
Without loss of generality we consider the first epoch initialized at $\xholder_0, \yholder_0$.
Recall from Eq.~\eqref{Asingleloop_SEG_B} we have
\beq\tag{\ref{Asingleloop_SEG_B}}\begin{aligned}
\Exs\left[
\|\bar{\xholder}_{\Tholder}\|^2 + \|\bar{\yholder}_{\Tholder}\|^2
\right]
	&\le
\frac{16}{(1-\alpha)\lambda_{\min}(\Bb\Bb^\top)}
\cdot
\frac{\|\xholder_0\|^2 + \|\yholder_0\|^2}{\hat\eta_{\Mb}(\alpha)^2(\Tholder+1)^2}
	\\&\quad\,
\hspace{-1.5in}
+
\frac{
4(\sigma_{\Bb}^2 + \hat\eta_{\Mb}(\alpha)^2\sigma_{\Bb,2}^2)\left[
\|\xholder_0\|^2 + \|\yholder_0\|^2 + \frac{3\sigma_{\gb}^2}{\lambda_{\min}(\Mb)\land\lambda_{\min}(\widehat{\Mb})}
\right]
+
18\sigma_{\gb}^2
}{(1-\alpha)\lambda_{\min}(\Bb\Bb^\top)}
\cdot
\frac{1}{\Tholder+1}
,
\end{aligned}\eeq
so after $\Tholder$ steps the iteration has a \emph{squared Euclidean metric} that is discounted by a factor of $1/e^2$, in the sense that the right hand of the above display is $\le \frac{1}{e^2}\left[ \|\xholder_0\|^2 + \|\yholder_0\|^2 \right]$.
This reduces to finding the solutions to the quadratic inequality
\beq\label{quadratic}
\frac{q_1}{(\Tholder+1)^2} + \frac{q_2}{\Tholder+1} - q_3 \le 0
,
\eeq
where $q_1,q_2,q_3$ were earlier defined as in Eq.~\eqref{q1q2q3} in the $\epoch=1$ case, repeated as
$$\begin{aligned}
q_1
	&\equiv
\frac{16}{(1-\alpha)\lambda_{\min}(\Bb\Bb^\top)}
\cdot
\frac{\|\xholder_0\|^2 + \|\yholder_0\|^2}{\hat\eta_{\Mb}(\alpha)^2}
\\
q_2
	&\equiv
\frac{
4(\sigma_{\Bb}^2 + \hat\eta_{\Mb}(\alpha)^2\sigma_{\Bb,2}^2)\left[
\|\xholder_0\|^2 + \|\yholder_0\|^2 + \frac{3\sigma_{\gb}^2}{\lambda_{\min}(\Mb)\land\lambda_{\min}(\widehat{\Mb})}
\right]
+
18\sigma_{\gb}^2
}{(1-\alpha)\lambda_{\min}(\Bb\Bb^\top)}
\\
q_3
	&\equiv
\frac{\|\xholder_0\|^2 + \|\yholder_0\|^2}{e^2}
.
\end{aligned}$$
The root formula gives (and omitting the infeasible solutions)
$$
\frac{1}{\Tholder+1} \le \frac{-q_2 + \sqrt{q_2^2 + 4q_1q_3}}{2q_1}
\quad\text{or equivalently}\quad
\Tholder
\ge
\left\lceil
\frac{q_2 + \sqrt{q_2^2 + 4q_1q_3}}{2q_3}
\right\rceil - 1
.
$$
This gives the epoch number Eq.~\eqref{Tholder_epoch_prime}.

To get a sensible bound on time complexity, instead of solving the quadratic formula Eq.~\eqref{quadratic} we instead consider to upper bound of $\Tholder_\epoch$ and its summation $\sum_{\epoch=1}^{\Epoch} \Tholder_{\epoch}$.
Recall first we set $\eta = \hat\eta_{\Mb}(\alpha)$ defined earlier as in Eq.~\eqref{eta_choice}.
From time to time we omit the ceilings for simplicity (which does not affect the magnitude as the terms grow large).
Since $\Bb_{\xi}$ is a square matrix, we set $\eta = \eta_{\Mb}$ as in~Eq.~\eqref{eta_choice} again and apply the following restarting schedule: 
run SEG at $\epoch=1,2,\dots$ for an iteration number of $\Tholder_{\epoch}$ defined as in Eq.~\eqref{Tholder_epoch_prime}, one can upper bound the iterate number at $\epoch=1,2,\dots$ a maximal over two terms where
\beq\tag{\ref{q1q2q3}}\begin{aligned}
q_1	&\equiv
\frac{16}{(1-\alpha)\lambda_{\min}(\Bb\Bb^\top)}
\cdot
\frac{e^{2-2\epoch}[\|\xholder_0\|^2 + \|\yholder_0\|^2]}{\hat\eta_{\Mb}(\alpha)^2}
\\
q_2	&\equiv
\frac{
4(\sigma_{\Bb}^2 + \hat\eta_{\Mb}(\alpha)^2\sigma_{\Bb,2}^2)\left[
e^{2-2\epoch}[\|\xholder_0\|^2 + \|\yholder_0\|^2] + \frac{3\sigma_{\gb}^2}{\lambda_{\min}(\Mb)\land\lambda_{\min}(\widehat{\Mb})}
\right]
+
18\sigma_{\gb}^2
}{(1-\alpha)\lambda_{\min}(\Bb\Bb^\top)}
\\
q_3	&\equiv
\frac{\|\xholder_0\|^2 + \|\yholder_0\|^2}{e^{2\epoch}}
,
\end{aligned}\eeq
which allows the following bound%
\footnote{Note it is easy to verify $
\left\lceil
\frac{q_2 + \sqrt{q_2^2 + 4q_1q_3}}{2q_3}
\right\rceil - 1
	\le
\max\left(\frac{2q_2}{q_3}, \sqrt{\frac{2q_1}{q_3}}\right)
	\le
\frac{2q_2}{q_3} + \sqrt{\frac{2q_1}{q_3}}
$
by considering the two cases of $\frac{2q_2}{q_3}\le \sqrt{\frac{2q_1}{q_3}}$ and $\frac{2q_2}{q_3} \ge \sqrt{\frac{2q_1}{q_3}}$, separately.}
$$\begin{aligned}
	&\quad\,
\Tholder_{\operatorname{complexity}}
	\le
\sum_{\epoch=1}^{\Epoch} \left[\frac{2q_2}{q_3} + \sqrt{\frac{2q_1}{q_3}}\right]
	\\&\le
\sum_{\epoch=1}^{\Epoch}
\left[
\frac{4e^2(\sigma_{\Bb}^2 + \hat\eta_{\Mb}(\alpha)^2\sigma_{\Bb,2}^2)}{(1-\alpha)\lambda_{\min}(\Bb\Bb^\top)}
+
\frac{
18\sigma_{\gb}^2
+
4(\sigma_{\Bb}^2 + \hat\eta_{\Mb}(\alpha)^2\sigma_{\Bb,2}^2) \frac{3\sigma_{\gb}^2}{\lambda_{\min}(\Mb)\land\lambda_{\min}(\widehat{\Mb})}
}{
(1-\alpha)\lambda_{\min}(\Bb\Bb^\top)[\|\xholder_0\|^2 + \|\yholder_0\|^2]
}
\cdot e^{2\epoch}
\right]
	\\&\quad\,
+
\sum_{\epoch=1}^{\Epoch}
\sqrt{
\frac{16e^2}{(1-\alpha) \hat\eta_{\Mb}(\alpha)^2 \lambda_{\min}(\Bb\Bb^\top)}
}
,
\end{aligned}$$
which is further bounded by
\def\LOG{\textsf{LOG}\xspace}
$$\begin{aligned}
	&\quad\,
\Tholder_{\operatorname{complexity}}
	\le
\left(
\sqrt{
\frac{16e^2}{(1-\alpha) \hat\eta_{\Mb}(\alpha)^2 \lambda_{\min}(\Bb\Bb^\top)}
}
	+
\frac{4e^2(\sigma_{\Bb}^2 + \hat\eta_{\Mb}(\alpha)^2\sigma_{\Bb,2}^2)}{(1-\alpha)\lambda_{\min}(\Bb\Bb^\top)}
\right)
	\cdot
\Epoch
	\\&\quad\,
+
\frac{
18\sigma_{\gb}^2
+
4(\sigma_{\Bb}^2 + \hat\eta_{\Mb}(\alpha)^2\sigma_{\Bb,2}^2) \frac{3\sigma_{\gb}^2}{\lambda_{\min}(\Mb)\land\lambda_{\min}(\widehat{\Mb})}
}{
(1-\alpha)\lambda_{\min}(\Bb\Bb^\top)[\|\xholder_0\|^2 + \|\yholder_0\|^2]
}
\cdot 
\sum_{\epoch=1}^{\Epoch}
e^{2\epoch}
	\\&=
\left(
\sqrt{
\frac{16e^2}{(1-\alpha) \hat\eta_{\Mb}(\alpha)^2 \lambda_{\min}(\Bb\Bb^\top)}
}
	+
\frac{4e^2(\sigma_{\Bb}^2 + \hat\eta_{\Mb}(\alpha)^2\sigma_{\Bb,2}^2)}{(1-\alpha)\lambda_{\min}(\Bb\Bb^\top)}
\right)
	\cdot
\left\lceil 
\dfrac12 \LOG
\right\rceil
	\\&\quad\,
+
\frac{
6(\lambda_{\min}(\Mb)\land\lambda_{\min}(\widehat{\Mb}))
+
4(\sigma_{\Bb}^2 + \hat\eta_{\Mb}(\alpha)^2\sigma_{\Bb,2}^2)
}{
(1-\alpha)\lambda_{\min}(\Bb\Bb^\top)[\|\xholder_0\|^2 + \|\yholder_0\|^2]
}
\cdot 
\frac{\|\xholder_0\|^2 + \|\yholder_0\|^2}{1 - e^{-2}}
,
\end{aligned}$$
where
$$
\LOG
\equiv
\log\left(
\frac{\lambda_{\min}(\Mb)\land\lambda_{\min}(\widehat{\Mb})}{3\sigma_{\gb}^2}
[\|\xholder_0\|^2 + \|\yholder_0\|^2]
\right)
,
$$
and we applied the fact
$$
\sum_{\epoch=1}^{\Epoch} e^{2\epoch} 
	=
\frac{e^{2\Epoch-2} - 1}{1 - e^{-2}}
	\le
\frac{\|\xholder_0\|^2 + \|\yholder_0\|^2}{1 - e^{-2}}\cdot
\frac{\lambda_{\min}(\Mb)\land\lambda_{\min}(\widehat{\Mb})}{3\sigma_{\gb}^2}
,
$$
hence concluding Eq.~\eqref{Tcomp_prime} and the theorem.
\end{proof}

\subsection{Analysis of Theorem \ref{theo_SEGg_interpolation_C}}\label{sec_proof,theo_SEGg_interpolation_C}

\begin{custom}{Theorem \ref{theo_SEGg_interpolation_C}, Full Version}
Let Assumptions~\ref{assu_boundednoise_A} and \ref{assu_boundednoise_B} hold with $\sigma_{\gb} = 0$.
When $\Bb_{\xi}, \Bb$ are square matrices, for any prescribed $\alpha\in (0,1)$ choosing the step size $\eta = \bar\eta_{\Mb}(\alpha)$ defined as in Eq.~\eqref{eta_choice}, for an iteration number of $\Tholder \ge \Tholder_{\operatorname{thres}}(\alpha)$ defined as in Eq.~\eqref{Tholder_epoch}.
Then we have for all $\Tholder \ge 1$ that is divisible by $\Tholder_{\operatorname{thres}}(\alpha)$ the following convergence rate for $\hat\xholder_{\Tholder}, \hat\yholder_{\Tholder}$ (outputs of Algorithm \ref{algo_iasgd_restart}) holds
\beq\label{Asingleloop_SEG_interpolation_C_glory}
\Exs \left[\|\hat\xholder_{\Tholder}\|^2 + \|\hat\yholder_{\Tholder}\|^2\right]
\le
\exp\left(-\frac{2}{\Tholder_{\operatorname{thres}}(\alpha)}\cdot \Tholder\right) \left[\|\xholder_0\|^2 + \|\yholder_0\|^2\right]
.
\eeq
In above we adopt the fine-grained iteration number per epoch
\beq\tag{\ref{Tholder_epoch}'}
\footnotesize{
\Tholder_{\operatorname{thres}}(\alpha)
\equiv
\left(
\frac{4}{
\sqrt{2\bar\eta_{\Mb}(\alpha)^2\left(\sigma_{\Bb}^2 + \bar\eta_{\Mb}(\alpha)^2 \sigma_{\Bb,2}^2 \right) + 8\RATE}
-
\sqrt{2\bar\eta_{\Mb}(\alpha)^2\left(\sigma_{\Bb}^2 + \bar\eta_{\Mb}(\alpha)^2 \sigma_{\Bb,2}^2 \right)}
}
\right)^2
}
,
\eeq
where
\beq\label{Tholder_epoch_rhs}
\RATE
	\equiv
\sqrt{\frac{(1-\alpha)\bar\eta_{\Mb}(\alpha)^2\lambda_{\min}(\Bb\Bb^\top)}{e^2}
}
.
\eeq
\end{custom}
In the regime of $\sigma_{\Bb}, \sigma_{\Bb,2}\to 0^+$ we do asymptotic expansion and get%
\footnote{We used the Taylor's asymptotic expansion $
(\sqrt{x+a} - \sqrt{x})^2
=
a(\sqrt{1+x/a} - \sqrt{x/a})^2
=
a - O(\sqrt{ax})
$ as $x\to 0^+$ for fixed positive $a$.}
$$\begin{aligned}
&\lefteqn{
\frac{2}{\Tholder_{\operatorname{thres}}}
    =
\frac18\left(
\sqrt{2\bar\eta_{\Mb}(\alpha)^2\left(\sigma_{\Bb}^2 + \bar\eta_{\Mb}(\alpha)^2 \sigma_{\Bb,2}^2 \right) + 8\RATE}
-
\sqrt{2\bar\eta_{\Mb}(\alpha)^2\left(\sigma_{\Bb}^2 + \bar\eta_{\Mb}(\alpha)^2 \sigma_{\Bb,2}^2 \right)}
\right)^2
}
	\\&=
\frac{1}{e}\sqrt{(1-\alpha)\bar\eta_{\Mb}(\alpha)^2\lambda_{\min}(\Bb\Bb^\top)}
-
O\left(
\sqrt[4]{\bar\eta_{\Mb}(\alpha)^2\lambda_{\min}(\Bb\Bb^\top)}
\cdot
\sqrt{\bar\eta_{\Mb}(\alpha)^2\sigma_{\Bb}^2 + \bar\eta_{\Mb}(\alpha)^4\sigma_{\Bb,2}^2}
\right)
,
\end{aligned}$$
which along with Eq.~\eqref{Asingleloop_SEG_interpolation_C_glory} gives Eq.~\eqref{Asingleloop_SEG_interpolation_C} of Theorem \ref{theo_SEGg_interpolation_C}.

Now for the interpolation setting, in the case of square matrices $\Bb_{\xi}, \Bb$, we turn to consider the convergence rate of SEG with iteration averaging, stating the following lemma:

\begin{lemma}\label{theo_SEGg_interpolation_B}
Let Assumptions~\ref{assu_boundednoise_A} and \ref{assu_boundednoise_B} hold with $\sigma_{\gb} = 0$.
Under the condition on step size $\eta\in (0, \eta_{\Mb}]$ where $\eta_{\Mb}$ was earlier defined as in Eq.~\eqref{etaMdef}, we conclude for all $\Tholder\ge 0$ the following convergence rate for $\overline{\xholder}_{\Tholder}, \overline{\yholder}_{\Tholder}$ holds
\beq\label{Asingleloop_SEG_interpolation_B_glory}\begin{aligned}
	&\quad\,
\left(
\lambda_{\min}(\Bb\Bb^\top)\left(1 + \eta^2\lambda_{\min}(\Bb\Bb^\top)\right)
-
2\eta\sigma_{\Bb}^2 \sqrt{\lambda_{\max}(\Bb^\top\Bb)}
\right)
\Exs\left[
\left\|\overline{\xholder}_{\Tholder}\right\|^2 + \left\|\overline{\yholder}_{\Tholder}\right\|^2
\right]
	\\&\le
\Exs\left[
\left\| \Bb \overline{\yholder}_{\Tholder}		+	\eta  \Mb \overline{\xholder}_{\Tholder} \right\|^2
+
\left\| \Bb^\top \overline{\xholder}_{\Tholder}	-	\eta \widehat{\Mb} \overline{\yholder}_{\Tholder} \right\|^2
\right]
	\\&\le
\left(\frac{2}{\eta} + \sqrt{2(\sigma_{\Bb}^2 + \eta^2 \sigma_{\Bb,2}^2)\Quan_{\Tholder+1}(\eta)}\right)^2
\frac{\|\xholder_0\|^2 + \|\yholder_0\|^2}{(\Tholder+1)^2}
.
\end{aligned}\eeq
In addition when $\Bb_{\xi}, \Bb$ are square matrices, we have
\beq\label{Asingleloop_SEG_interpolation_B_glory2}\begin{aligned}
\Exs\left[
\|\overline{\xholder}_{\Tholder}\|^2 + \|\overline{\yholder}_{\Tholder}\|^2
\right]
	&\le
\prefactor_{\Tholder+1}(\eta)
\cdot
\frac{\|\xholder_0\|^2 + \|\yholder_0\|^2}{(\Tholder+1)^2}
.
\end{aligned}\eeq
In above the prefactor is defined as%
\footnote{Here we interpret $0\cdot (+\infty)$ as $+\infty$ whenever it occurs.}
$$
\prefactor_{\Tholder+1}(\eta)
	\equiv
\left\{\begin{array}{ll}
+\infty
&
\hspace{-1.5in}
\text{if \footnotesize{$\lambda_{\min}(\Bb\Bb^\top)\left(1 + \eta^2\lambda_{\min}(\Bb\Bb^\top)\right)
\le
2\eta\sigma_{\Bb}^2 \sqrt{\lambda_{\max}(\Bb^\top\Bb)}$}
}
	\\
\frac{
\left(2 + \sqrt{2\eta^2\left(\sigma_{\Bb}^2 + \eta^2 \sigma_{\Bb,2}^2 \right)\Quan_{\Tholder+1}(\eta)}\right)^2
}{
\eta^2\lambda_{\min}(\Bb\Bb^\top)\left(1 + \eta^2\lambda_{\min}(\Bb\Bb^\top)\right)
-
2\eta^3\sigma_{\Bb}^2 \sqrt{\lambda_{\max}(\Bb^\top\Bb)}
}
	&
\text{otherwise}
\end{array}\right.
,
$$
where $\Quan_{\Tholder}(\eta)$ was earlier defined as in Eq.~\eqref{Quandef}, and by setting $\eta$ as
\beq\tag{\ref{eta_choice}}\begin{aligned}
\bar\eta_{\Mb}(\alpha)
	&\equiv
\eta_{\Mb}
\land
\frac{\alpha\lambda_{\min}(\Bb\Bb^\top)}{2\sigma_{\Bb}^2 \sqrt{\lambda_{\max}(\Bb^\top\Bb)}}
,
\end{aligned}\eeq
we have
\beq\label{linearized}\begin{aligned}
\prefactor_{\Tholder+1}(\bar\eta_{\Mb}(\alpha))
	&\le
\frac{2}{(1-\alpha)\lambda_{\min}(\Bb\Bb^\top)}
\left(
\sqrt{\frac{2}{\bar\eta_{\Mb}(\alpha)^2}}
+
\sqrt{\left(\sigma_{\Bb}^2 + \bar\eta_{\Mb}(\alpha)^2 \sigma_{\Bb,2}^2\right)(\Tholder+1)}
\right)^2
.
\end{aligned}\eeq
\end{lemma}
Lemma \ref{theo_SEGg_interpolation_B} can be seen as a fine-grained version of Theorem~\ref{theo_SEG_B}, and its proof is provided in \S\ref{sec_proof,theo_SEGg_interpolation_B}.
To understand it consider Theorem \ref{theo_SEGg_interpolation_B} in the case where $\Bb_\xi$ is nonstochastic so $\sigma_{\Bb} = \sigma_{\Bb,2} = 0$, taking $\eta$ as the maximal $\eta_{\Mb} = \frac{1}{\sqrt{\lambda_{\max}(\Bb^\top\Bb)}}$ then Eq.~\eqref{linearized} achieves the optimal prefactor which is bounded by the quadruple condition number of $\Bb^\top\Bb$.
In the general case where $\Bb_\xi$ is stochastic, the convergence rate upper bound Eq.~\eqref{Asingleloop_SEG_interpolation_B_glory2} has the nonrandom component as $O(1/\Tholder^2)$ as well as the random component of $O(1/\Tholder)$.

To prepare the proof we first introduce the following ``metric conversion'' lemma that translates bounds between two metrics:

\begin{lemma}\label{lemm_normconvert}
We have for any $\xholder\in \RR^n$, $\yholder\in \RR^m$ that
\beq\label{eqbdd1}\begin{aligned}
	&\quad\,
\left\| \Bb \yholder + \eta  \Mb \xholder \right\|^2
+
\left\| \Bb^\top \xholder - \eta \widehat{\Mb} \yholder \right\|^2
	=
\left\|\begin{bmatrix}
\Bb^\top		& - \eta  \widehat{\Mb}
\\
\eta \Mb		& \Bb
\end{bmatrix}
\begin{bmatrix}
\xholder
\\
\yholder
\end{bmatrix}\right\|^2 
	\\&\ge
\left(
\lambda_{\min}(\Bb\Bb^\top)\left(1 + \eta^2 \lambda_{\min}(\Bb\Bb^\top)\right)
-
2\eta\sigma_{\Bb}^2 \sqrt{\lambda_{\max}(\Bb^\top\Bb)}
\right)
\left[
\left\|\xholder\right\|^2 + \left\|\yholder\right\|^2
\right]
.
\end{aligned}\eeq
\end{lemma}
Lemma~\ref{lemm_normconvert} establishes a lower bound on a modified version of the Hamiltonian metric by (a constant multiple of) the squared Euclidean norm metric.
The proof of the above inequality is due to an estimation of the spectral lower bound of a matrix.
To take a first glance note in the nonrandom case $\sigma_{\Bb} = 0$, $\widehat{\Mb} = \Bb^\top\Bb$ and $\Mb = \Bb\Bb^\top$ and we conclude Eq.~\eqref{eqbdd1} in the form
$$\begin{aligned}
	&\quad\,
\left\|\Bb \yholder + \eta \Mb\xholder\right\|^2 
+
\left\|\Bb^\top\xholder - \eta \widehat{\Mb}\yholder\right\|^2
	=
\left\|\Bb \yholder + \eta \Bb\Bb^\top\xholder\right\|^2 
+
\left\|\Bb^\top\xholder - \eta \Bb^\top\Bb\yholder\right\|^2
	\\&=
\xholder^\top \left(\Bb\Bb^\top + \eta^2 (\Bb\Bb^\top)^2\right) \xholder
+
\yholder^\top \left(\Bb^\top\Bb + \eta^2 (\Bb^\top\Bb)^2\right) \yholder
	\\&\ge
\lambda_{\min}(\Bb\Bb^\top)\left(1 + \eta^2 \lambda_{\min}(\Bb\Bb^\top)\right)
\left[
\left\|\xholder\right\|^2 + \left\|\yholder\right\|^2
\right]
.
\end{aligned}$$
We prove the lemma for the general stochastic $\Bb_\xi$ case; details are deferred to \S\ref{sec_proof,lemm_normconvert}.

We are ready for the proof of Theorem~\ref{theo_SEGg_interpolation_C}.
\begin{proof}[Proof of Theorem~\ref{theo_SEGg_interpolation_C}]
From Lemma \ref{theo_SEGg_interpolation_B} we have
$$\begin{aligned}
\Exs\left[
\|\overline{\xholder}_{\Tholder}\|^2 + \|\overline{\yholder}_{\Tholder}\|^2
\right]
	\le
\frac{2}{(1-\alpha)\lambda_{\min}(\Bb\Bb^\top)}
\left(
\sqrt{\frac{2}{\bar\eta_{\Mb}(\alpha)^2 (\Tholder+1)^2}} 
+ 
\sqrt{\frac{\sigma_{\Bb}^2 + \bar\eta_{\Mb}(\alpha)^2 \sigma_{\Bb,2}^2}{\Tholder+1}}
\right)^2
\left[\|\xholder_0\|^2 + \|\yholder_0\|^2 \right]
,
\end{aligned}$$
so after $\Tholder$ steps the metric $
\Exs\left[\|\overline{\xholder}_{\Tholder}\|^2 + \|\overline{\yholder}_{\Tholder}\|^2\right]
	\le
\frac{1}{e^2}\left[\|\xholder_0\|^2 + \|\yholder_0\|^2\right]
$, i.e.~we only need
$$
\sqrt{\frac{2}{(1-\alpha)\lambda_{\min}(\Bb\Bb^\top)}}\left(
\sqrt{\frac{2}{\bar\eta_{\Mb}(\alpha)^2 (\Tholder+1)^2}}
+
\sqrt{\frac{\sigma_{\Bb}^2 + \bar\eta_{\Mb}(\alpha)^2 \sigma_{\Bb,2}^2}{\Tholder+1}}
\right)
	\le
\frac{1}{e}
,
$$
i.e.
$$
\frac{2}{\Tholder+1}
+
\sqrt{2\bar\eta_{\Mb}(\alpha)^2\left(\sigma_{\Bb}^2 + \bar\eta_{\Mb}(\alpha)^2 \sigma_{\Bb,2}^2 \right)}
\cdot\frac{1}{\sqrt{\Tholder+1}}
	\le
\sqrt{
\frac{(1-\alpha)\bar\eta_{\Mb}(\alpha)^2\lambda_{\min}(\Bb\Bb^\top)}{e^2}
}
	\equiv
\RATE
,
$$
solving the above inequality and ignoring the infeasible solutions gives for any prescribed $\alpha\in (0,1)$
$$\begin{aligned}
\Tholder_{\operatorname{thres}}(\alpha)+1
	=
\left(
\frac{\sqrt{2\bar\eta_{\Mb}(\alpha)^2\left(\sigma_{\Bb}^2 + \bar\eta_{\Mb}(\alpha)^2 \sigma_{\Bb,2}^2 \right)}
+
\sqrt{2\bar\eta_{\Mb}(\alpha)^2\left(\sigma_{\Bb}^2 + \bar\eta_{\Mb}(\alpha)^2 \sigma_{\Bb,2}^2 \right)+8\RATE}
}{2\RATE}
\right)^2
,
\end{aligned}$$
which reduces to (\ref{Tholder_epoch}') after rationalizing the numerator.
\end{proof}

\subsection{Auxiliary Proofs}\label{sec_proof,auxiliary2}
\subsubsection{Proof of Lemma \ref{theo_SEGg_interpolation_B}}\label{sec_proof,theo_SEGg_interpolation_B}
For the proof of Lemma \ref{theo_SEGg_interpolation_B}, our analysis lends help of Young's inequality via coefficients $1+\gamma$, $1 + \frac{1}{\gamma}$ with optimized coefficient $\gamma\in (0,\infty)$.

\begin{proof}[Proof of Lemma \ref{theo_SEGg_interpolation_B}]
Turning to Eq.~\eqref{Asingleloop_SEG_B}, setting $\eta$ as in Eq.~\eqref{eta_choice} and telescoping both sides of the update rule~Eq.~\eqref{SEGupdate_combined} with $\gb^\xholder_{\xi,t} = 0$ and $\gb^\yholder_{\xi,t} = 0$ for $t = 1,\ldots,\Tholder$ gives
$$\begin{aligned}
- \eta^2 \sum_{t=1}^\Tholder  \Bb_{\xi, t} \Bb_{\xi, t}^\top \xholder_{t-1} 
- \eta \sum_{t=1}^\Tholder  \Bb_{\xi, t} \yholder_{t-1}
&=
\xholder_\Tholder - \xholder_0
	\\
- \eta^2 \sum_{t=1}^\Tholder \Bb_{\xi, t}^\top \Bb_{\xi, t} \yholder_{t-1}
+ \eta \sum_{t=1}^\Tholder \Bb_{\xi, t}^\top \xholder_{t-1}
&=
\yholder_\Tholder - \yholder_0
.
\end{aligned}$$
Manipulating gives
$$\begin{aligned}
	&\quad\,
\Bb\overline{\yholder}_{\Tholder-1}
+
\eta\Mb\overline{\xholder}_{\Tholder-1} 
+
\frac{1}{\Tholder}\sum_{t=1}^\Tholder (\Bb_{\xi, t} - \Bb) \yholder_{t-1}
+
\frac{\eta}{\Tholder}\sum_{t=1}^\Tholder \left( \Bb_{\xi, t} \Bb_{\xi, t}^\top - \Mb \right)\xholder_{t-1} 
	\\&=
\frac{1}{\Tholder}\sum_{t=1}^\Tholder \Bb_{\xi, t} \yholder_{t-1}
+
\frac{\eta}{\Tholder}\sum_{t=1}^\Tholder   \Bb_{\xi, t} \Bb_{\xi, t}^\top \xholder_{t-1} 
	=
\frac{\xholder_\Tholder - \xholder_0}{-\eta\Tholder}
,
\end{aligned}$$
and
$$\begin{aligned}
	&\quad\,
\Bb^\top\overline{\xholder}_{\Tholder-1}
-
\eta\widehat\Mb\overline{\yholder}_{\Tholder-1} 
+
\frac{1}{\Tholder}\sum_{t=1}^\Tholder (\Bb_{\xi, t} - \Bb)^\top \xholder_{t-1}
-
\frac{\eta}{\Tholder}\sum_{t=1}^\Tholder \left(\Bb_{\xi, t}^\top \Bb_{\xi, t} - \widehat\Mb\right)\yholder_{t-1} 
	\\&=
\frac{1}{\Tholder}\sum_{t=1}^\Tholder \Bb_{\xi, t}^\top \xholder_{t-1}
-
\frac{\eta}{\Tholder} \sum_{t=1}^\Tholder \Bb_{\xi, t}^\top \Bb_{\xi, t} \yholder_{t-1}
	=
\frac{\yholder_\Tholder - \yholder_0}{\eta\Tholder}
.
\end{aligned}$$
Now we try to bound the sum of squared norms of the first part (i.e.~first two terms) on the left hands in the above two displays:
applying Young's inequality gives for any fixed $\gamma\in (0,\infty)$
$$\begin{aligned}
	&\quad\,
\Exs\left[
\left\| \Bb \overline{\yholder}_{\Tholder-1}
+
\eta  \Mb \overline{\xholder}_{\Tholder-1}
\right\|^2
+
\left\| \Bb^\top \overline{\xholder}_{\Tholder-1}
-
\eta \widehat{\Mb} \overline{\yholder}_{\Tholder-1}
\right\|^2
\right]
	\\&\le
(1+\gamma)\Exs\left\|\frac{\xholder_\Tholder - \xholder_0}{-\eta\Tholder}\right\|^2 
+
(1+\gamma)\Exs\left\|\frac{\yholder_\Tholder - \yholder_0}{\eta\Tholder}\right\|^2
	\\&\quad\,
+
\left(1+\frac{1}{\gamma}\right)\Exs\left\|
\frac{1}{\Tholder}\sum_{t=1}^\Tholder (\Bb_{\xi, t} - \Bb) \yholder_{t-1}
+
\frac{\eta}{\Tholder}\sum_{t=1}^\Tholder  \left( \Bb_{\xi, t}\Bb_{\xi, t}^\top - \Mb \right)\xholder_{t-1} 
\right\|^2
	\\&\quad\,
+
\left(1+\frac{1}{\gamma}\right)\Exs\left\|
\frac{1}{\Tholder}\sum_{t=1}^\Tholder (\Bb_{\xi, t} - \Bb)^\top \xholder_{t-1}
-
\frac{\eta}{\Tholder}\sum_{t=1}^\Tholder  \left( \Bb_{\xi, t}^\top \Bb_{\xi, t} - \widehat{\Mb} \right)\yholder_{t-1} 
\right\|^2
	\\&\le
\frac{4(1+\gamma)}{\eta^2\Tholder^2}\left[\|\xholder_0\|^2 + \|\yholder_0\|^2 \right]
	\\&\quad\,
+
\frac{2\left(1+\frac{1}{\gamma}\right)}{\Tholder^2}\sum_{t=1}^\Tholder \left[
\Exs\left\|(\Bb_{\xi, t} - \Bb) \yholder_{t-1}\right\|^2
+
\Exs\left\|(\Bb_{\xi, t} - \Bb)^\top \xholder_{t-1}\right\|^2
\right]
	\\&\quad\,
+
\frac{2\left(1+\frac{1}{\gamma}\right)\eta^2}{\Tholder^2}\sum_{t=1}^\Tholder  \left[
\Exs\left\|\left( \Bb_{\xi, t}\Bb_{\xi, t}^\top - \Mb \right)\xholder_{t-1} \right\|^2
+
\Exs\left\|\left( \Bb_{\xi, t}^\top \Bb_{\xi, t} - \widehat{\Mb} \right)\yholder_{t-1} \right\|^2
\right]
,
\end{aligned}$$
where for each $t=1,\dots,\Tholder$ we have, by applying Eq.~\eqref{sigmaAsq} and Eq.~\eqref{sigmaA2sqinit} in Assumption \ref{assu_boundednoise_A} on operator norms, that
$$\begin{aligned}
	&\quad\,
\Exs_\xi\left\| (\Bb_{\xi, t} - \Bb) \yholder_{t-1} \right\|^2
+
\Exs_\xi\left\| (\Bb_{\xi, t} - \Bb)^\top \xholder_{t-1} \right\|^2
	\\&=
(\yholder_{t-1})^\top		\Exs_\xi\left[
(\Bb_{\xi, t} - \Bb)^\top (\Bb_{\xi, t} - \Bb)
\right]
\yholder_{t-1}
+
(\xholder_{t-1})^\top		\Exs_\xi\left[
(\Bb_{\xi, t} - \Bb) (\Bb_{\xi, t} - \Bb)^\top
\right]
\xholder_{t-1}
	\\&\le
\left\|\Exs_\xi\left[(\Bb_{\xi, t} - \Bb)^\top(\Bb_{\xi, t} - \Bb)\right]\right\|_{op} \left\|\yholder_{t-1}\right\|^2
+
\left\|\Exs_\xi\left[(\Bb_{\xi, t} - \Bb)(\Bb_{\xi, t} - \Bb)^\top\right]\right\|_{op} \left\|\xholder_{t-1}\right\|^2
	\\&\le
\sigma_{\Bb}^2 \left[ \|\xholder_{t-1}\|^2 + \|\yholder_{t-1}\|^2 \right]
,
\end{aligned}$$
and
$$\begin{aligned}
	&\quad\,
\Exs_\xi\left\|
\left( \Bb_{\xi, t}\Bb_{\xi, t}^\top - \Mb \right)\xholder_{t-1} 
\right\|^2
+
\Exs_\xi\left\|
\left( \Bb_{\xi, t}^\top \Bb_{\xi, t} - \widehat{\Mb} \right)\yholder_{t-1} 
\right\|^2
	\\&=
(\xholder_{t-1})^\top
\Exs_\xi\left( \Bb_{\xi, t}\Bb_{\xi, t}^\top - \Mb \right)^2
\xholder_{t-1}
+
(\yholder_{t-1})^\top
\Exs_\xi\left( \Bb_{\xi, t}^\top \Bb_{\xi, t} - \widehat{\Mb} \right)^2
\yholder_{t-1}
	\\&\le
\left\|\Exs_\xi
\left( \Bb_{\xi, t}^\top \Bb_{\xi, t} - \widehat{\Mb} \right)^2
\right\|_{op} \left\|\yholder_{t-1}\right\|^2
+
\left\|\Exs_\xi
\left( \Bb_{\xi, t}\Bb_{\xi, t}^\top - \Mb \right)^2
\right\|_{op} \left\|\xholder_{t-1}\right\|^2
	\\&\le
\sigma_{\Bb,2}^2 \left[ \|\xholder_{t-1}\|^2 + \|\yholder_{t-1}\|^2 \right]
.
\end{aligned}$$
Taking expectation once again gives, by applying Eq.~\eqref{xygrowth_SEG_A} that for $t=1,\dots,\Tholder$
$$\begin{aligned}
	&\quad\,
\Exs\left\| (\Bb_{\xi, t} - \Bb) \yholder_{t-1} \right\|^2
+
\Exs\left\| (\Bb_{\xi, t} - \Bb)^\top \xholder_{t-1} \right\|^2
	\\&\le
\sigma_{\Bb}^2 \Exs\left[ \|\xholder_{t-1}\|^2 + \|\yholder_{t-1}\|^2 \right]
	\le
\sigma_{\Bb}^2 \left(1 - \eta^2 \lambdaetaprime\right)^{t-1} \left[\|\xholder_0\|^2 + \|\yholder_0\|^2 \right]
,
\end{aligned}$$
and
$$\begin{aligned}
	&\quad\,
\Exs\left\|
\left( \Bb_{\xi, t}\Bb_{\xi, t}^\top - \Mb \right)\xholder_{t-1} 
\right\|^2
+
\Exs\left\|
\left( \Bb_{\xi, t}^\top \Bb_{\xi, t} - \widehat{\Mb} \right)\yholder_{t-1} 
\right\|^2
	\\&\le
\sigma_{\Bb,2}^2 \Exs \left[ \|\xholder_{t-1}\|^2 + \|\yholder_{t-1}\|^2 \right]
	\le
\sigma_{\Bb,2}^2 \left(1 - \eta^2 \lambdaetaprime\right)^{t-1} \left[\|\xholder_0\|^2 + \|\yholder_0\|^2 \right]
,
\end{aligned}$$
so denoting $\Quan_{\Tholder}(\eta) \equiv \sum_{t=1}^\Tholder \left(1 - \eta^2 \lambdaetaprime\right)^{t-1}$ as in Eq.~\eqref{Quandef} concludes
$$\begin{aligned}
	&\quad\,
\Exs\left[
\left\| \Bb \overline{\yholder}_{\Tholder-1}
+
\eta  \Mb \overline{\xholder}_{\Tholder-1}
\right\|^2
+
\left\| \Bb^\top \overline{\xholder}_{\Tholder-1}
-
\eta \widehat{\Mb} \overline{\yholder}_{\Tholder-1}
\right\|^2
\right]
	\\&\le
\frac{4(1+\gamma)}{\eta^2\Tholder^2}\left[\|\xholder_0\|^2 + \|\yholder_0\|^2 \right]
+
\frac{2\left(1+\frac{1}{\gamma}\right)}{\Tholder^2}\sum_{t=1}^\Tholder \left[
\Exs\left\|(\Bb_{\xi, t} - \Bb) \yholder_{t-1}\right\|^2
+
\Exs\left\|(\Bb_{\xi, t} - \Bb)^\top \xholder_{t-1}\right\|^2
\right]
	\\&\quad\,
+
\frac{2\left(1+\frac{1}{\gamma}\right)\eta^2}{\Tholder^2}\sum_{t=1}^\Tholder  \left[
\Exs\left\|\left( \Bb_{\xi, t}\Bb_{\xi, t}^\top - \Mb \right)\xholder_{t-1} \right\|^2
+
\Exs\left\|\left( \Bb_{\xi, t}^\top \Bb_{\xi, t} - \widehat{\Mb} \right)\yholder_{t-1} \right\|^2
\right]
	\\&\le
\frac{4(1+\gamma)}{\eta^2\Tholder^2}\left[\|\xholder_0\|^2 + \|\yholder_0\|^2 \right]
+
\frac{2\left(1+\frac{1}{\gamma}\right)}{\Tholder^2}
\sum_{t=1}^\Tholder
\sigma_{\Bb}^2 \left(1 - \eta^2 \lambdaetaprime\right)^{t-1} \left[\|\xholder_0\|^2 + \|\yholder_0\|^2 \right]
	\\&\quad\,
+
\frac{2\left(1+\frac{1}{\gamma}\right)\eta^2}{\Tholder^2}
\sum_{t=1}^\Tholder
\sigma_{\Bb,2}^2 \left(1 - \eta^2 \lambdaetaprime\right)^{t-1} \left[\|\xholder_0\|^2 + \|\yholder_0\|^2 \right]
	\\&\le
\frac{4(1+\gamma)}{\eta^2\Tholder^2}\left[\|\xholder_0\|^2 + \|\yholder_0\|^2 \right]
+
\frac{2\left(1+\frac{1}{\gamma}\right)}{\Tholder^2} \left(\sigma_{\Bb}^2 + \eta^2 \sigma_{\Bb,2}^2 \right)
\Quan_{\Tholder}(\eta) \left[\|\xholder_0\|^2 + \|\yholder_0\|^2 \right]
.
\end{aligned}$$
In above we used the iterated laws of expectation as well as the property of $L^2$ martingale at multiple occasions.
Therefore
\beq\label{Asingleloop_SEG_interpolation_B_glory_prime}\begin{aligned}
	&\quad\,
\left(
\lambda_{\min}(\Bb\Bb^\top)\left(1 + \eta^2\lambda_{\min}(\Bb\Bb^\top)\right)
-
2\eta\sigma_{\Bb}^2 \sqrt{\lambda_{\max}(\Bb^\top\Bb)}
\right)
\Exs\left[
\left\|\overline{\xholder}_{\Tholder-1}\right\|^2 + \left\|\overline{\yholder}_{\Tholder-1}\right\|^2
\right]
	\\&\le
\Exs\left[
\left\| \Bb \overline{\yholder}_{\Tholder-1}		+	\eta  \Mb \overline{\xholder}_{\Tholder-1} \right\|^2
+
\left\| \Bb^\top \overline{\xholder}_{\Tholder-1}	-	\eta \widehat{\Mb} \overline{\yholder}_{\Tholder-1} \right\|^2
\right]
	\\&\le
\inf_{\gamma\in (0,\infty)}\left(
\frac{4(1+\gamma)}{\eta^2}
+
2\left(1+\frac{1}{\gamma}\right)\left(\sigma_{\Bb}^2 + \eta^2 \sigma_{\Bb,2}^2 \right)
\Quan_{\Tholder}(\eta)
\right)
\frac{\|\xholder_0\|^2 + \|\yholder_0\|^2}{\Tholder^2}
.
\end{aligned}\eeq
Note by optimizing over $\gamma\in (0,\infty)$ in above the identity is
$$
\inf_{\gamma\in (0,\infty)}\left(
\frac{4\gamma}{\eta^2}
+
\frac{2}{\gamma}\left(\sigma_{\Bb}^2 + \eta^2 \sigma_{\Bb,2}^2 \right)
\Quan_{\Tholder}(\eta)
\right)
=
2\sqrt{
\frac{8}{\eta^2}\left(\sigma_{\Bb}^2 + \eta^2 \sigma_{\Bb,2}^2 \right)
\Quan_{\Tholder}(\eta)
}
,
$$
so the prefactor on the right hand of Eq.~\eqref{Asingleloop_SEG_interpolation_B_glory_prime} reduces to
$$
\frac{4}{\eta^2} + 2\left(\sigma_{\Bb}^2 + \eta^2 \sigma_{\Bb,2}^2 \right)\Quan_{\Tholder}(\eta)
+
2\sqrt{
\frac{8}{\eta^2}\left(\sigma_{\Bb}^2 + \eta^2 \sigma_{\Bb,2}^2 \right)
\Quan_{\Tholder}(\eta)
}
	=
\left(\frac{2}{\eta} + \sqrt{2\left(\sigma_{\Bb}^2 + \eta^2 \sigma_{\Bb,2}^2 \right)\Quan_{\Tholder}(\eta)}\right)^2
,
$$
concluding Eq.~\eqref{Asingleloop_SEG_interpolation_B_glory} by replacing $\Tholder$ by $\Tholder+1$.

Now to finish the proof, by setting $\eta$ as $\bar\eta_{\Mb}(\alpha)$ defined as in Eq.~\eqref{eta_choice}, we have a tight upper bound of the prefactor as the step size $\eta$ over interval $\left(0,\bar\eta_{\Mb}(\alpha)\right]$ for a prescribed $\alpha\in (0,1)$, as
$$\begin{aligned}
\prefactor_{\Tholder+1}(\eta)
	&=
\frac{
\left(2 + \sqrt{2\eta^2\left(\sigma_{\Bb}^2 + \eta^2 \sigma_{\Bb,2}^2 \right)\Quan_{\Tholder+1}(\eta)}\right)^2
}{
\eta^2\lambda_{\min}(\Bb\Bb^\top)\left(1 + \eta^2\lambda_{\min}(\Bb\Bb^\top)\right)
-
2\eta^3\sigma_{\Bb}^2 \sqrt{\lambda_{\max}(\Bb^\top\Bb)}
}
	\\&\le
\frac{
\left(2 + \sqrt{2\eta^2\left(\sigma_{\Bb}^2 + \eta^2 \sigma_{\Bb,2}^2 \right) (\Tholder+1)} \right)^2
}{
(1-\alpha)\eta^2\lambda_{\min}(\Bb\Bb^\top)\left(1 + \eta^2\lambda_{\min}(\Bb\Bb^\top)\right)
}
	\\&\le
\frac{2}{(1-\alpha)\lambda_{\min}(\Bb\Bb^\top)}
\left(
\sqrt{\frac{2}{\eta^2}}
+
\sqrt{\left(\sigma_{\Bb}^2 + \eta^2 \sigma_{\Bb,2}^2\right)(\Tholder+1)}
\right)^2
,
\end{aligned}$$
which is just
\beq\tag{\ref{linearized}}\begin{aligned}
\prefactor_{\Tholder+1}(\eta)
	&\le
\frac{2}{(1-\alpha)\lambda_{\min}(\Bb\Bb^\top)}
\Bigg(
\underbrace{
\frac{2}{\eta^2}
+
\frac{2\sqrt{2(\Tholder+1)}}{\eta}\cdot \sqrt{\sigma_{\Bb}^2 + \eta^2 \sigma_{\Bb,2}^2}
}_{\text{linearization}}
+
\underbrace{
(\Tholder+1)\cdot\left(\sigma_{\Bb}^2 + \eta^2 \sigma_{\Bb,2}^2\right)
}_{\text{higher-order term}}
\Bigg)
.
\end{aligned}\eeq
It is straightforward to verify that for a prescribed $\alpha\in (0,1)$, $\bar\eta_{\Mb}(\alpha)$ simply minimizes the upper bound in the last line of Eq.~\eqref{linearized} when dropping the higher-order term in $\sqrt{\sigma_{\Bb}^2 + \eta^2 \sigma_{\Bb,2}^2}$, since such a \emph{linearized prefactor} is nonincreasing over $\eta\in \left(0,\bar\eta_{\Mb}(\alpha)\right]$.
This completes the proof of Eq.~\eqref{Asingleloop_SEG_interpolation_B_glory} and the full version of Lemma \ref{theo_SEGg_interpolation_B}.%
\footnote{
Note one can further optimize the above prefactor over $\alpha\in (0,1)$  (so that the convergence rate upper bound is minimized), but finding an interpretable closed-form solution can be unrealistic.
An initial attempt on this thread is to optimize a surrogate function $
\frac{\eta_{\Mb}^2}{(1-\alpha)\bar\eta_{\Mb}(\alpha)^2}
$, $\alpha\in (0,1)$, which is upper-bounded by
$$
\frac{1}{1-\alpha}
+
\frac{1}{\alpha^2(1-\alpha)}\left(
\frac{2\eta_{\Mb}\sigma_{\Bb}^2 \sqrt{\lambda_{\max}(\Bb^\top\Bb)}}{\lambda_{\min}(\Bb\Bb^\top)}
\right)^2
	\equiv
\frac{1}{1-\alpha}
+
\frac{\mathcal{A}}{\alpha^2(1-\alpha)}
,
$$
In the regime of $\mathcal{A}\to 0^+$, its closed-form solution is available but hard to interpret, and standard asymptotic analysis indicates that $
\alpha \sim \mathcal{A}^{1/3} = \left(
\frac{2\eta_{\Mb}\sigma_{\Bb}^2 \sqrt{\lambda_{\max}(\Bb^\top\Bb)}}{\lambda_{\min}(\Bb\Bb^\top)}
\right)^{2/3}
$ minimizes the above display.
}
\end{proof}

\subsubsection{Proof of Lemma \ref{lemm_normconvert}}\label{sec_proof,lemm_normconvert}

\begin{proof}[Proof of Lemma~\ref{lemm_normconvert}]
The left hand of Eq.~\eqref{eqbdd1} reads
$$\begin{aligned}
	&\hspace{-.1in}
\left\| \Bb \yholder
+
\eta  \Mb \xholder
\right\|^2
+
\left\| \Bb^\top \xholder
-
\eta \widehat{\Mb} \yholder
\right\|^2
	=
\left\|\begin{bmatrix}
\Bb^\top		& - \eta  \widehat{\Mb}
\\
\eta \Mb		& \Bb
\end{bmatrix}
\begin{bmatrix}
\xholder
\\
\yholder
\end{bmatrix}\right\|^2 
	\\&=
\begin{bmatrix}
\xholder^\top &	\yholder^\top
\end{bmatrix}
\begin{bmatrix}
\Bb				& \eta \Mb
\\
- \eta  \widehat{\Mb}	& \Bb^\top
\end{bmatrix}
\begin{bmatrix}
\Bb^\top		& - \eta  \widehat{\Mb}
\\
\eta \Mb		& \Bb
\end{bmatrix}
\begin{bmatrix}
\xholder
\\
\yholder
\end{bmatrix}
	\\&=
\begin{bmatrix}
\xholder^\top &	\yholder^\top
\end{bmatrix}\begin{bmatrix}
\Bb\Bb^\top + \eta^2 \Mb^2					&	-\eta \Bb \widehat{\Mb} + \eta  \Mb^\top \Bb
\\
-\eta \widehat{\Mb} \Bb^\top + \eta \Bb^\top \Mb		&	\Bb^\top\Bb + \eta^2\widehat{\Mb}^2
\end{bmatrix}
\begin{bmatrix}
\xholder
\\
\yholder
\end{bmatrix}
	\\&=
\begin{bmatrix}
\xholder^\top &
\yholder^\top
\end{bmatrix}\begin{bmatrix}
\Bb\Bb^\top + \eta^2 (\Bb\Bb^\top)^2			&	0
\\
0									&	\Bb^\top\Bb + \eta^2(\Bb^\top\Bb)^2
\end{bmatrix}
\begin{bmatrix}
\xholder\\
\yholder
\end{bmatrix}
	\\&\quad\, 
+
\begin{bmatrix}
\xholder^\top &
\yholder^\top
\end{bmatrix}\begin{bmatrix}
\eta^2 \Mb^2 - \eta^2 (\Bb\Bb^\top)^2		&	0
\\
0									&	\eta^2\widehat{\Mb}^2 - \eta^2(\Bb^\top\Bb)^2
\end{bmatrix}
\begin{bmatrix}
\xholder\\
\yholder
\end{bmatrix}
	\\&\quad\, 
+
\begin{bmatrix}
\xholder^\top				&	\yholder^\top
\end{bmatrix}\begin{bmatrix}
0									&	-\eta \Bb \widehat{\Mb} + \eta  \Mb^\top \Bb
\\
-\eta \widehat{\Mb} \Bb^\top + \eta \Bb^\top \Mb	&	0
\end{bmatrix}
\begin{bmatrix}
\xholder
\\
\yholder
\end{bmatrix}
	\\&=
\xholder^\top\left[
\Bb\Bb^\top + \eta^2 (\Bb\Bb^\top)^2
\right]\xholder
+
\yholder^\top\left[
\Bb^\top\Bb + \eta^2(\Bb^\top\Bb)^2
\right]\yholder
	\\&\quad\, 
+
\xholder^\top\left[
\eta^2 \Mb^2 - \eta^2 (\Bb\Bb^\top)^2
\right]\xholder
+
\yholder^\top\left[
\eta^2\widehat{\Mb}^2 - \eta^2(\Bb^\top\Bb)^2
\right]\yholder
	\\&\quad\, 
+ 
2\yholder^\top\left(-\eta \widehat{\Mb} \Bb^\top + \eta \Bb^\top \Mb\right)\xholder 
	\\&\equiv
\mbox{I}
+
\mbox{II}
+
\mbox{III}
,
\end{aligned}$$
where
$$\begin{aligned}
\mbox{I}
	&=
\xholder^\top\left[
\Bb\Bb^\top + \eta^2 (\Bb\Bb^\top)^2
\right]\xholder
+
\yholder^\top\left[
\Bb^\top\Bb + \eta^2(\Bb^\top\Bb)^2
\right]\yholder
	\\&=
\left\| \Bb \yholder
+
\eta  \Bb\Bb^\top \xholder
\right\|^2
+
\left\| \Bb^\top \xholder
-
\eta \Bb^\top\Bb \yholder
\right\|^2
.
\end{aligned}$$
In addition, we have from Assumption \ref{assu_boundednoise_A} that
$
\Mb - \Bb\Bb^\top
	=
\Exs\left[ \Bb_{\xi}\Bb_{\xi}^\top \right] - \Bb\Bb^\top
	=
\Exs\left[ (\Bb_{\xi} - \Bb) (\Bb_{\xi} - \Bb)^\top \right]
	\succeq
\mathbf{0}
$
and analogously
$
\widehat{\Mb} - \Bb^\top\Bb
	=
\Exs\left[ \Bb_{\xi}^\top \Bb_{\xi}\right] - \Bb^\top \Bb
	=
\Exs\left[ (\Bb_{\xi} - \Bb)^\top (\Bb_{\xi} - \Bb) \right]
	\succeq
0
$
that almost surely
$$\begin{aligned}
\mbox{II}
	&\ge
\xholder^\top\left[
\eta^2 \Mb^2 - \eta^2 (\Bb\Bb^\top)^2
\right]\xholder
+
\yholder^\top\left[
\eta^2\widehat{\Mb}^2 - \eta^2(\Bb^\top\Bb)^2
\right]\yholder
	\ge
0
.
\end{aligned}$$
The third term
$$\begin{aligned}
\mbox{III}
=
2\yholder^\top\left(-\eta \widehat{\Mb} \Bb^\top + \eta \Bb^\top \Mb\right)\xholder 
\end{aligned}$$
satisfies
$$\begin{aligned}
\left|\mbox{III}\right|
	&=
\left|
2\yholder^\top\left(-\eta \widehat{\Mb} \Bb^\top + \eta \Bb^\top \Mb\right)\xholder
\right| 
	\\&=
2\left|
\yholder^\top(-\eta \widehat{\Mb} \Bb^\top + \eta \Bb^\top\Bb \Bb^\top + \eta \Bb^\top \Mb - \eta \Bb^\top \Bb \Bb^\top )\xholder
\right|
	\\&\le
2\eta \left|
-\yholder^\top\left(\widehat{\Mb} - \Bb^\top\Bb\right) \Bb^\top \xholder 
\right| 
+ 
2\eta \left|
\yholder^\top\Bb^\top \left(\Mb - \Bb \Bb^\top \right)\xholder
\right|
	\\&\le
2\eta
\left\|\left(\widehat{\Mb} - \Bb^\top\Bb\right)\yholder\right\|
\left\|\Bb^\top\xholder\right\|
+ 
2\eta
\left\|\left(\Mb - \Bb\Bb^\top\right)\xholder\right\|
\left\|\Bb\yholder\right\|
	\\&\le
2\eta\sigma_{\Bb}^2
\left\|\yholder\right\|
\left\|\Bb^\top\xholder\right\|
+ 
2\eta\sigma_{\Bb}^2
\left\|\xholder\right\|
\left\|\Bb\yholder\right\|
	\\&\le
2\eta\sigma_{\Bb}^2 \sqrt{\lambda_{\max}(\Bb^\top\Bb)}\left[
\left\|\xholder\right\|^2
+
\left\|\yholder\right\|^2
\right]
,
\end{aligned}$$
where we have from Eq.~\eqref{sigmaAsq} and Eq.~\eqref{sigmaA2sqinit} of Assumption \ref{assu_boundednoise_A} that for all $\xholder\in \RR^n, \yholder\in \RR^m$ that
$
\left\| ( \widehat{\Mb} - \Bb^\top\Bb ) \yholder \right\|
	\le
\sigma_{\Bb}^2 \left\|\yholder\right\|
$ and $
\left\| ( \Mb - \Bb \Bb^\top ) \xholder \right\|
	\le
\sigma_{\Bb}^2 \left\|\xholder\right\|
$.

One final piece is that for any given $\xholder\in\RR^n, \yholder\in\RR^m$ and $\eta>0$
\beq\label{lambdaMINbdd}\begin{aligned}
	&\quad\,
\left\| \Bb \yholder + \eta  \Bb\Bb^\top \xholder \right\|^2
+
\left\| \Bb^\top \xholder - \eta \Bb^\top\Bb \yholder \right\|^2
	\\&=
\xholder^\top\left[
\Bb\Bb^\top + \eta^2 (\Bb\Bb^\top)^2
\right]\xholder
+
\yholder^\top\left[
\Bb^\top\Bb + \eta^2(\Bb^\top\Bb)^2
\right]\yholder
	\\&\ge
\left(1 + \eta^2 \lambda_{\min}(\Bb\Bb^\top)\right)
\left[
\left\|\Bb^\top \xholder\right\|^2 + \left\|\Bb\yholder\right\|^2
\right]
	\ge
\lambda_{\min}(\Bb\Bb^\top)\left(1 + \eta^2 \lambda_{\min}(\Bb\Bb^\top)\right)
\left[
\left\|\xholder\right\|^2 + \left\|\yholder\right\|^2
\right]
.
\end{aligned}\eeq
Now to conclude Eq.~\eqref{eqbdd1}, we apply Eq.~\eqref{lambdaMINbdd}, and the above analysis (taking the expectation) gives the following result
$$\begin{aligned}
	&\quad\,
\left\| \Bb \yholder + \eta  \Mb \xholder \right\|^2
+
\left\| \Bb^\top \xholder - \eta \widehat{\Mb} \yholder \right\|^2
	=
\mbox{I} + \mbox{II} + \mbox{III}
	\ge
\mbox{I} - |\mbox{III}|
	\\&=
\left\| \Bb \yholder + \eta  \Bb\Bb^\top \xholder \right\|^2
+
\left\| \Bb^\top \xholder - \eta \Bb^\top\Bb \yholder \right\|^2
-
2\eta\sigma_{\Bb}^2 \sqrt{\lambda_{\max}(\Bb^\top\Bb)} \left[
\left\| \xholder \right\|^2 + \left\| \yholder \right\|^2
\right]
,
\end{aligned}$$
which is no less than
$$
\left(
\lambda_{\min}(\Bb\Bb^\top)\left(1 + \eta^2 \lambda_{\min}(\Bb\Bb^\top)\right)
-
2\eta\sigma_{\Bb}^2 \sqrt{\lambda_{\max}(\Bb^\top\Bb)}
\right)
\left[
\left\|\xholder\right\|^2 + \left\|\yholder\right\|^2
\right]
,
$$
again due to Eq.~\eqref{lambdaMINbdd}.
\end{proof}

\section{ADDITIONAL EXPERIMENTAL RESULTS}\label{sec:additional-exp-results-appendix}

\paragraph{Experiments on GANs.}
We conduct GANs experiments on MNIST~\citep{lecun1998mnist} understand more the empirical performance of our restarted iteration-averaged SEG in non-convex non-concave minimax optimization problems.
Specifically, we adopt the DCGAN network architecture proposed in \citet{radford2015unsupervised} and use the loss proposed in \citet{goodfellow2014generative}.
We adopt ExtraAdam~\citep{gidel2019variational} as the optimizer and denote the ExtraAdam with averaging by SEG-Avg.
Meanwhile, we apply restarting at iteration 2000 (halfway the total run) for SEG-Avg and denote the restarting method by SEG-Avg-Restart.
We apply Fréchet Inception distance (FID)~\citep{heusel2017gans} score for measuring GAN performances, and compare the performance of SEG-Avg and SEG-Avg-Restart in Figure~\ref{fig:GAN}.
We observe that proper restarting schedule improves the model performance w.r.t.~FID score (lower scores indicate better performance).

\begin{figure*}[h]
\begin{center}
\subfigure[FID comparison.]{
\includegraphics[width=0.38\textwidth]{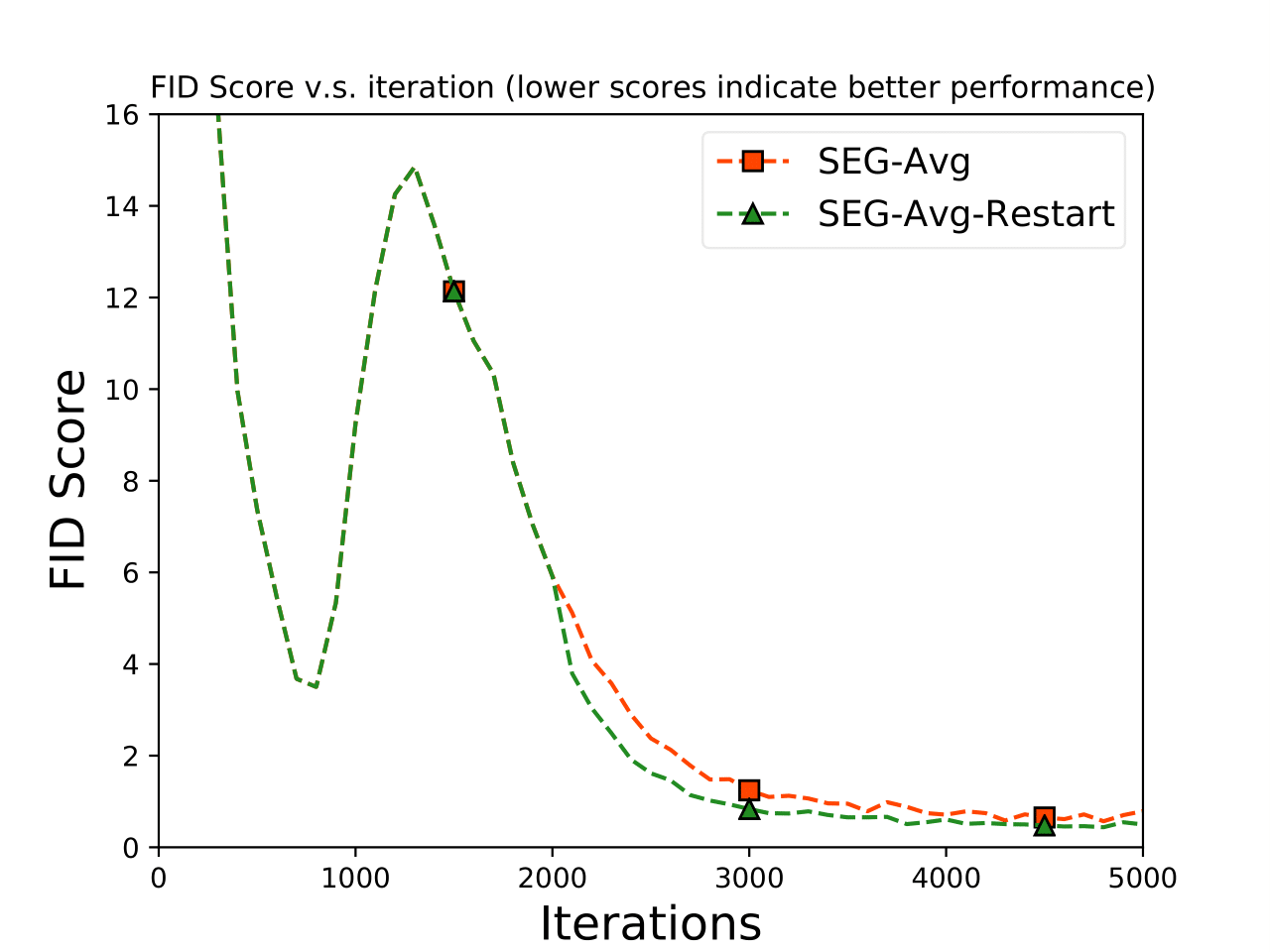}
}
\subfigure[SEG-Avg.]{
\includegraphics[width=0.28\textwidth]{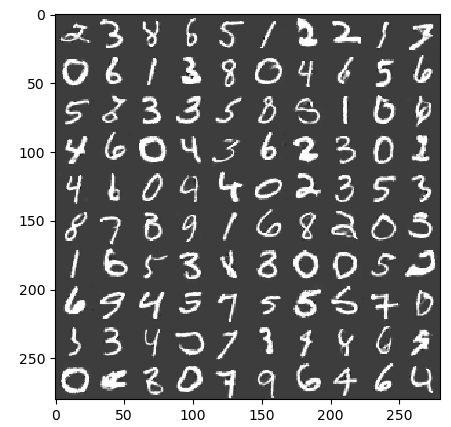}
}
\subfigure[SEG-Avg-Restart.]{
\includegraphics[width=0.28\textwidth]{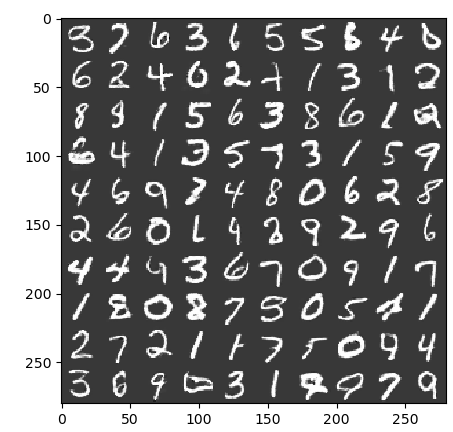}
}
\caption{GAN experimental results. (\textbf{a}). Comparing GAN performance (measure by FID score) of SEG-Avg and SEG-Avg-Restart.  (\textbf{b}). Images generated by SEG-Avg. (\textbf{c}). Images generated by SEG-Avg-Restart.  }
\label{fig:GAN}
\end{center}
\vspace{-0.3in}
\end{figure*}

\end{document}